\documentclass{amsart}
\usepackage{amssymb}
\usepackage{graphicx}
\usepackage{amscd}

\allowdisplaybreaks
\numberwithin{figure}{section}
\numberwithin{table}{section}
\numberwithin{equation}{section}

\newtheorem{thm}{Theorem}[section]
\newtheorem{prop}[thm]{Proposition}
\newtheorem{lem}[thm]{Lemma}
\newtheorem{cor}[thm]{Corollary}
\newtheorem{remark}{Remark}[section]
\newcommand\C{\mathbb{C}}
\newcommand\N{\mathbb{N}}
\newcommand\R{\mathbb{R}}
\newcommand\Z{\mathbb{Z}}
\newcommand\A{\mathcal A}
\newcommand\B{\mathcal B}
\newcommand\D{\mathcal D}
\newcommand\w{\omega}
\renewcommand\Im{\operatorname{Im}}
\renewcommand\Re{\operatorname{Re}}
\newcommand\err{\operatorname{err}}
\newcommand\sgn{\operatorname{sgn}}
\newcommand\Lc{\mathcal L}
\newcommand\Lip{\operatorname{Lip}}
\newcommand\diam{\operatorname{diam}}

\begin{document}
 \title[Computation of Hausdorff Dimension]
{$C^m$ Eigenfunctions of Perron-Frobenius Operators and a 
New Approach to Numerical Computation of Hausdorff Dimension}
\author{Richard S. Falk}
\address{Department of Mathematics,
Rutgers University, Piscataway, NJ 08854}
\email{falk@math.rutgers.edu}
\urladdr{http://www.math.rutgers.edu/\char'176falk/}

\author{Roger D. Nussbaum}
\address{Department of Mathematics,
Rutgers University, Piscataway, NJ 08854}
\email{nussbaum@math.rutgers.edu}
\urladdr{http://www.math.rutgers.edu/\char'176nussbaum/}
\thanks{The work of the second author was supported by
NSF grant DMS-1201328.}
\date{December 1, 2015}

\begin{abstract}

  We develop a new approach to the computation of the Hausdorff
  dimension of the invariant set of an iterated function system or
  IFS.  In the one dimensional case, our methods require only $C^3$
  regularity of the maps in the IFS.  The key idea, which has been
  known in varying degrees of generality for many years, is to
  associate to the IFS a parametrized family of positive, linear,
  Perron-Frobenius operators $L_s$. The operators $L_s$ can typically
  be studied in many different Banach spaces. Here, unlike most of the
  literature, we study $L_s$ in a Banach space of real-valued,
  $C^k$ functions, $k \ge 2$; and we note that $L_s$ is not compact,
  but has a strictly positive eigenfunction $v_s$ with positive
  eigenvalue $\lambda_s$ equal to the spectral radius of $L_s$. Under
  appropriate assumptions on the IFS, the Hausdorff dimension of the
  invariant set of the IFS is the value $s=s_*$ for which $\lambda_s
  =1$.  This eigenvalue problem is then approximated by a collocation
  method using continuous piecewise linear functions (in one
  dimension) or bilinear functions (in two dimensions).  Using the
  theory of positive linear operators and explicit a priori bounds on
  the derivatives of the strictly positive eigenfunction $v_s$, we
  give rigorous upper and lower bounds for the Hausdorff
  dimension $s_*$, and these bounds converge to $s_*$ as the mesh size
  approaches zero.

\end{abstract}

\maketitle

\section{Introduction}
\label{sec:intro}

Our interest in this paper is in finding rigorous estimates for the
Hausdorff dimension of invariant sets for (possibly infinite) iterated
function systems or IFS's.  The case of graph directed IFS's (see
\cite{I} and \cite{H}) is also of
great interest and can be studied by our methods, but for simplicity we
shall restrict attention here to the IFS case.

Let $D \subset \R^n$ be a nonempty compact set, $\rho$ a metric on $D$ which
gives the topology on $D$, and $\theta_j: D \to D$, $1 \le j \le m$,
a contraction mapping, i.e., a Lipschitz mapping (with respect to $\rho$)
with Lipschitz constant $\Lip(\theta_j)$, satisfying $\Lip(\theta_j):= c_j <1$.
If $m < \infty$ and the above assumption holds, it is known that there
exists a unique, compact, nonempty set $C \subset D$ such that
$C = \cup_{j =1}^m \theta_j(C)$.  The set $C$ is called the invariant set for
the IFS $\{\theta_j \, | \, 1 \le j \le m\}$. If $m = \infty$ and
$\sup \{c_j \, | \, 1 \le j \le m\} = c <1$, there is a naturally defined
nonempty invariant set $C \subset D$ such that 
$C = \cup_{j =1}^\infty \theta_j(C)$,  but $C$ need not be compact.  It is
useful to note that the Lipschitz condition of the IFS can be weakened,
and we address this in a subsequent section 
(cf. (H6.1) in Section~\ref{sec:1d-deriv}).

Although we shall eventually specialize, it may be helpful to describe
initially some function analytic results in the generality of the previous
paragraph. Let $H$ be a bounded, open, mildly regular subset of $\R^n$ and let
$C^k(\bar H)$ denote the real Banach space of $C^k$ real-valued maps, all of
whose partial derivatives of order $\nu \le k$ extend continuously to $\bar
H$.  For a given positive integer $N$, assume that $b_j: \bar H \to (0,
\infty)$ are strictly positive $C^N$ functions for $1 \le j \le m < \infty$
and $\theta_j: \bar H \to \bar H$, $1 \le j \le m$, are $C^N$ maps and
contractions.  For $s >0$ and integers $k$, $0 \le k \le N$, one can define a
bounded linear map $L_{s,k}: C^k(\bar H) \to C^k(\bar H)$ by the formula
\begin{equation}
\label{intro1.2}
(L_{s,k} f)(x) = \sum_{j=1}^m [b_j(x)]^s f(\theta_j(x)).
\end{equation}
Linear maps like $L_{s,k}$ are sometimes called positive transfer
operators or Perron-Frobenius operators and arise in many contexts
other than computation of Hausdorff dimension: see, for example,
\cite{Baladi}. If $r(L_{s,k})$ denotes the spectral radius of
$L_{s,k}$, then $\lambda_s = r(L_{s,k})$ is positive and independent of
$k$ for $0 \le k \le N$; and $\lambda_s$ is an algebraically simple
eigenvalue of $L_{s,k}$ with a corresponding unique, normalized strictly
positive eigenfunction $v_s \in C^N(\bar H)$.  Furthermore, the map
$s \mapsto \lambda_s$ is continuous.   If $\sigma(L_{s,k}) \subset \C$ denotes
the spectrum of the complexification of $L_{s,k}$, $\sigma(L_{s,k})$ depends
on $k$, but for $1 \le k \le N$,
\begin{equation}
\label{intro1.3}
\sup\{|z|: z \in \sigma(L_{s,k})\setminus\{\lambda_s\}\} < \lambda_s.
\end{equation}
If $k=0$, the strict inequality in \eqref{intro1.3} may fail. A more precise
version of the above result in stated in Theorem~\ref{thm:1.1} of this paper
and Theorem~\ref{thm:1.1} is a special case of results in \cite{E}.  The
method of proof involves ideas from the theory of positive linear operators,
particularly generalizations of the Kre{\u\i}n-Rutman theorem to noncompact
linear operators; see \cite{Krein-Rutman}, \cite{Bonsall},
\cite{Schaefer-Wolff}, \cite{L}, and \cite{Mallet-Paret-Nussbaum}. We do not
use the thermodynamic formalism (see \cite{Ruelle}) and often our operators
cannot be studied in Banach spaces of analytic functions.

The linear operators which are relevant for the computation of Hausdorff
dimension comprise a small subset of the transfer operators described in
\eqref{intro1.2}, but the analysis problem which we shall consider here can be
described in the generality of \eqref{intro1.2} and is of interest in this
more general context. We want to find rigorous methods to estimate
$r(L_{s,k})$ accurately and then use these methods to estimate $s_*$, where,
in our applications, $s_*$ will be the unique number $s \ge 0$ such that
$r(L_{s,k})=1$.  Under further assumptions, we shall see that $s_*$ equals
$\dim_H(C)$, the Hausdorff dimension of the invariant set associated to the
IFS.  This observation about Hausdorff dimension has been made, in varying
degrees of generality by many authors. See, for example, \cite{Bumby1},
\cite{Bumby2}, \cite{Bowen}, \cite{Cusick1}, \cite{Cusick2}, \cite{Falconer},
\cite{Good}, \cite{Hensley1}, \cite{Hensley2}, \cite{Hensley3},
\cite{J}, \cite{Jenkinson}, \cite{Jenkinson-Pollicott},
\cite{MR1902887}, \cite{H}, \cite{Mauldin-Urbanski}, \cite{N-P-L},
\cite{Ruelle}, \cite{Ruelle2}, \cite{Rugh}, and \cite{Schief}.

In the applications in this paper, $H$ will always be a bounded open
subset of $\R^n$ for $n=1$ or $2$.  When $n=1$, we shall assume that
$H$ is a finite union of bounded open intervals, that $\theta_j: \bar
H \to \bar H$ is a $C^N$ contraction mapping, where $N \ge 3$, (or
more generally satisfies (H6.1)) and $\theta_j^{\prime}(x) \neq 0$ for
all $x \in \bar H$. In the notation of \eqref{intro1.2}, we define
$b_j(x) = |\theta_j^{\prime}(x)|$.  When $n=2$, we assume that $H$ is
a bounded, open mildly regular subset of $\R^2 = \C$ and that
$\theta_j$, $1 \le j \le m$ are analytic or conjugate analytic
contraction maps (or more generally satisfy (H6.1)), defined on an open
neighborhood of $\bar H$ and satisfying $\theta_j(H) \subset H$.
We define $D \theta_j(z) = \lim_{h \rightarrow 0} |[\theta_j(z+h)- \theta_j(z)]/h|$,
where $h \in \C$ in the limit, and we assume that $D \theta_j(z)\neq 0$ for
$z \in \bar H$. In this case, $L_{s,k}$ is defined by \eqref{intro1.2},
with $x$ replaced by $z$, and $b_j(z)= |D \theta_j(z)|^s$.

Given the existence of a strictly positive $C^N$ eigenfunction $v_s$
for \eqref{intro1.2} when $H \subset R$, we show in
Section~\ref{sec:1d-deriv} for $p=1$ and $p=2$, that one can obtain
explicit upper and lower bounds for the quantity $D^p v_s(x)/v_s(x)$
for $x \in \bar H$, where $D^p$ denotes the $p$-th derivative of
$v_s$. Such bounds can also be obtained for $p=3$ and $p=4$, but the
arguments and calculations are more complicated. When $H \subset
\R^2$, it is also possible to obtain explicit upper and lower bounds
for $D_1^p v_s(x_1,x_2))/v_s(x_1,x_2)$ and $D_2^p
v_s(x_1,x_2))/v_s(x_1,x_2)$, where $D_1 = \partial/\partial x_1$ and
$D_2 = \partial/\partial x_2$. However, for simplicity we restrict
ourselves to the choice $\theta_j(z) = (z + \beta_j)^{-1}$, where
$\beta_j \in \C$ and $\Re(\beta_j) > 0$. In this case we obtain in
Section~\ref{sec:mobius} explicit upper and lower bounds for $D_k^p
v_s(x_1,x_2))/v_s(x_1,x_2)$ for $1 \le p \le 4$, $1 \le k \le 2$, and
$x_1 >0$.  In both the one and two dimensional cases, these estimates
play a crucial role in allowing us to obtain rigorous upper and lower
bounds for the Hausdorff dimension.

The basic idea of our numerical scheme is to cover $\bar H$ by nonoverlapping
intervals of length $h$ if $H \subset R$ or by nonoverlapping squares of side
$h$ if $H \subset R^2$.  We then approximate the strictly positive, $C^2$
eigenfunction $v_s$ by a continuous piecewise linear function (if $H \subset
\R$) or a continuous piecewise bilinear function (if $H \subset \R^2$).  Using
the explicit bounds on the unmixed derivatives of $v_s$ of order $2$, we are
then able to associate to the operator $L_{s,k}$, square matrices $A_s$ and
$B_s$, which have nonnegative entries and also have the property that $r(A_s)
\le \lambda_s \le r(B_s)$. A key role here is played by an elementary fact
which is not as well known as it should be.  If $M$ is a nonnegative matrix
and $v$ is a strictly positive vector and $M v \le \lambda v$,
(coordinate-wise), then $r(M) \le \lambda$.  An analogous statement is true if
$M v \ge \lambda v$. We emphasize that our approach is robust and allows us to
study the case $H \subset \R$ when $\theta_j(\cdot)$, $1 \le j \le m$, is only
$C^3$.

If $s_*$ denotes the unique value of $s$ such that $r(L_{s_*}) =
\lambda_{s_*} = 1$, so that $s_*$ is the Hausdorff dimension of the
invariant set for the IFS under study, we proceed as follows.  If we
can find a number $s_1$ such that $r(B_{s_1}) \le 1$, then, since the
map $s \mapsto \lambda_s$ is decreasing, $\lambda_{s_1} \le r(B_{s_1})
\le 1$, and we can conclude that $s_* \le s_1$.  Analogously, if we
can find a number $s_2$ such that $r(A_{s_2}) \ge 1$, then
$\lambda_{s_2} \ge r(A_{s_2}) \ge 1$, and we can conclude that $s_*
\ge s_2$.  By choosing the mesh size for our approximating piecewise
polynomials to be sufficiently small, we can make $s_1-s_2$ small,
providing a good estimate for $s_*$.  For a given $s$, $r(A_s)$ and
$r(B_s)$ are easily found by variants of the power method for
eigenvalues, since (see Section~\ref{sec:compute-sr}) the largest eigenvalue has
multiplicity one and is the only eigenvalue of its modulus.  When the
IFS is infinite, the procedure is somewhat more complicated, and we
include the necessary theory to deal with this case.

If the coefficients $b_j(\cdot)$ and the maps $\theta_j(\cdot)$ in
\eqref{intro1.2} are $C^N$ with $N >2$, it is natural to approximate
$v_s(\cdot)$ by piecewise polynomials of degree $N-1$ when $H \subset \R$
and by corresponding higher order approximations when $H \subset \R^2$.
The corresponding matrices $A_s$ and $B_s$ may no longer have all
nonnegative entries and the arguments of this paper are no longer
directly applicable.  However, we hope to prove in a future paper
that the inequality $r(A_s) \le \lambda_s \le r(B_s)$ remains true and
leads to much improved upper and lower bounds for $r(L_s)$.  Heuristic
evidence for this assertion is given in Table~\ref{tb:t2} of 
Section~\ref{subsec:cantor-num}.

We illustrate our new approach by first considering in 
Section~\ref{sec:1dexps} the computation of the Hausdorff dimension
of invariant sets in $[0,1]$ arising from classical continued fraction
expansions.  In this much studied case, one defines $\theta_m =
1/(x+m)$, for $m$ a positive integer and $x \in [0,1]$; and for a
subset $\B \subset \N$, one considers the IFS $\{\theta_m \, | \, m
\in \B\}$ and seeks estimates on the Hausdorff dimension of the
invariant set $C =C(\B)$ for this IFS.  This problem has previously
been considered by many authors. See \cite{Bourgain-Kontorovich},
\cite{Bumby1}, \cite{Bumby2}, \cite{Good}, \cite{Hensley1},
\cite{Hensley2}, \cite{Hensley3}, \cite{Jenkinson},
\cite{Jenkinson-Pollicott}, and \cite{Heinemann-Urbanski}.  In this case,
\eqref{intro1.2} becomes
\begin{equation*}
(L_{s,k}v)(x) = \sum_{m \in \B} \Big(\frac{1}{x+m}\Big)^{2s} 
v\Big(\frac{1}{x+m}\Big), \qquad 0 \le x \le 1,
\end{equation*}
and one seeks a value $s \ge 0$ for which $\lambda_s:= r(L_{s,k}) =1$.
Table~\ref{tb:t1} in Section~\ref{subsec:cantor-num} gives upper and lower
bounds for the value $s$ such that $\lambda_s =1$ for various sets
$\B$. Jenkinson and Pollicott \cite{Jenkinson-Pollicott} use a completely
different method and obtain, when $|\B|$ is small, high accuracy estimates for
$\dim_H(C(\B))$, in which successive approximations converge at a
super-exponential rate.  It is less clear (see \cite{Jenkinson})
how well the approximation scheme in \cite{Jenkinson-Pollicott} or
\cite{Jenkinson} works when $|\B|$ is moderately large or when different real
analytic functions $\hat \theta_j: [0,1] \to [0,1]$ are used.  Here, in the
one dimensional case, we present an alternative approach with much wider
applicability that only requires the maps in the IFS to be $C^3$.  As an
illustration, we consider in Section~\ref{subsec:lowreg} perturbations of the
IFS for the middle thirds Cantor set for which the corresponding contraction
maps are $C^3$, but not $C^4$.

In Section~\ref{sec:2dexp}, we consider the computation of the Hausdorff
dimension of some invariant sets arising for complex continued fractions.
Suppose that $\B$ is a subset of $I_1 = \{m+ni \, | \, m \in \N, n\in \Z\}$,
and for each $b \in \B$, define $\theta_b(z) = (z+b)^{-1}$. Note
that $\theta_b$ maps $\bar G = \{z \in \C \, | \, |z-1/2| \le 1/2\}$
into itself.  We are interested in the Hausdorff dimension of the invariant
set $C = C(\B)$ for the IFS $\{\theta_b \, |\, b \in \B\}$.
This is a two dimensional problem and we allow the possibility that
$\B$ is infinite. In general (contrast work in \cite{Jenkinson-Pollicott} 
and \cite{Jenkinson}), it does not seem possible in this case to replace
$L_{s,k}$, $k \ge 2$, by an operator $\Lambda_s$ acting on a Banach space
of analytic functions of one complex variable and satisfying
$r(\Lambda_s) = r(L_{s,k})$.  Instead, we work in $C^2(\bar G)$ and apply
our methods to obtain rigorous upper and lower bounds for the Hausdorff
dimension $\dim_H(C(\B))$ for several examples. The case $\B = I_1$ has
been of particular interest and is one motivation for this paper.
In \cite{Gardner-Mauldin}, Gardner and Mauldin proved that 
$d:= \dim_H(C(I_1)) <2$, in \cite{Mauldin-Urbanski}, Mauldin and Urbanski proved
that $d \le 1.885$, and in \cite{Priyadarshi}, Priyadarshi proved that
$d \ge 1.78$.  In Section~\ref{subsec:method}, we prove that
$1.85550 \le d \le 1.85589$.

The square matrices $A_s$ and $B_s$ mentioned above and described in
more detail in Section~\ref{sec:1dexps} have nonnegative entries and
satisfy $r(A_s) \le \lambda_s \le r(B_s)$.  To apply standard
numerical methods, it is useful to know that all eigenvalues $\mu \neq
r(A_s)$ of $A_s$ satisfy $|\mu| < r(A_s)$ and that $r(A_s)$ has
algebraic multiplicity one and that corresponding results hold for
$r(B_s)$.  Such results are proved in Section~\ref{sec:compute-sr} in
the one dimensional case when the mesh size, $h$, is sufficiently
small, and a similar argument can be used in the two dimensional case.
Note that this result does not follow from the standard theory of
nonnegative matrices, since $A_s$ and $B_s$ typically have zero
columns and are not primitive. We also prove that $r(A_s) \le r(B_s)
\le (1 + C_1 h^2) r(A_s)$, where the constant $C_1$ can be explicitly
estimated.  In Section~\ref{sec:logconvex}, we prove that the map $s
\mapsto \lambda_s$ is log convex and strictly decreasing; and the same
result is proved for $s \mapsto r(M_s)$, where $M_s$ is a naturally
defined matrix such that $A_s \le M_s \le B_s$.

Although many of the key results in the paper are described above, the paper
is long and an outline summarizing the sections may be helpful. In
Section~\ref{sec:prelim}, we recall the definition of Hausdorff dimension and
present some mathematical preliminaries. In Sections~\ref{sec:1dexps} and
\ref{sec:2dexp}, we present the details of our approximation scheme for
Hausdorff dimension, explain the crucial role played by estimates on
derivatives of order $\le 2$ of $v_s$, and give the aforementioned estimates
for Hausdorff dimension.  We emphasize that this is a feasibility study. We
have limited the accuracy of our approximations to what is easily found using
the standard precision of {\it Matlab} and have run only a limited number of
examples, using mesh sizes that allow the programs to run fairly quickly. In
addition, we have not attempted to exploit the special features of our
problems, such as the fact that our matrices are sparse.  Thus, it is clear
that one could write a more efficient code that would also speed up the
computations.  However, the {\it Matlab} programs we have developed are
available on the web at 
{\tt www.math.rutgers.edu/\char'176falk/hausdorff/codes.html},
and we hope other researchers will run other examples of interest to them.

The theory underlying the work in Sections~\ref{sec:1dexps} and
\ref{sec:2dexp} is deferred to
Sections~\ref{sec:exist}--\ref{sec:logconvex}.  In
Section~\ref{sec:exist} we describe some results concerning existence
of $C^m$ positive eigenfunctions for a class of positive (in the sense
of order-preserving) linear operators.  We remark that
Theorem~\ref{thm:1.1} in Section~\ref{sec:exist} was only proved in
\cite{E} for finite IFS's.  As a result, some care is needed in
dealing with infinite IFS's: see Theorem~\ref{thm:1.9} and
Corollary~\ref{cor:1.4}.  In Section~\ref{sec:1d-deriv}, we derive
explicit bounds on the derivatives of the eigenfunction $v_s$ of $L_s$
in the one-dimensional case and in Section~\ref{sec:mobius}, we derive
explicit bounds on the derivatives of eigenfunctions of operators in
which the mappings $\theta_{\beta}$ are given by M\"obius
transformations which map a given bounded open subset $H$ of $\C:=
\R^2$ into $H$.  In Section~\ref{sec:compute-sr}, we verify some
spectral properties of the approximating matrices which justify
standard numerical algorithms for computing their spectral
radii. Finally, in Section~\ref{sec:logconvex}, we show the log
convexity of the spectral radius $r(L_s)$, which we exploit in our
numerical approximation scheme.

\section{Preliminaries}
\label{sec:prelim}
We recall the definition of the Hausdorff dimension, $\dim_H(K)$, of
a subset $K \subset \R^N$.  To do so, we first define for a given $s \ge0$
and each set $K \subset \R^N$, 
\begin{equation*}
H_{\delta}^s(K) = \inf\{\sum_i |U_i|^s: \{U_i\} \text{ is a } \delta \text{ cover
of } K\},
\end{equation*}
where $|U|$ denotes the diameter of $U$ and a countable collection $\{U_i\}$
of subsets of $\R^N$ is a $\delta$-cover of $K \subset \R^N$ if
$K \subset \cup_i U_i$ and $0 < |U_i| < \delta$.  We then define
the $s$-dimensional Hausdorff measure
\begin{equation*}
H^s(K) = \lim_{\delta \rightarrow 0+} H_{\delta}^s(K).
\end{equation*}
Finally, we define the Hausdorff dimension of $K$, $\dim_H(K)$, as
\begin{equation*}
\dim_H(K) = \inf\{s: H^s(K) =0\}.
\end{equation*}

We now state the main result connecting Hausdorff dimension to 
the spectral radius of the map defined by \eqref{intro1.2}.  To do so,
we first define the concept of  an {\it infinitesimal similitude}.
Let $(S,d)$ be a compact, perfect metric space. If $\theta:S \to S$, then
$\theta$ is an infinitesimal similitude at $t \in S$ if for any sequences
$(s_k)_k$ and $(t_k)_k$ with $s_k \neq t_k$ for $k \ge 1$ and $s_k \rightarrow
t$, $t_k \rightarrow t$, the limit
\begin{equation*}
\lim_{k \rightarrow \infty} \frac{d(\theta(s_k), \theta(t_k)}{d(s_k,t_k)}
=: (D \theta)(t)
\end{equation*}
exists and is independent of the particular sequences 
$(s_k)_k$ and $(t_k)_k$.  Furthermore, $\theta$ is an infinitesimal similitude
on $S$ if $\theta$ is an infinitesimal similitude at $t$ for all $t \in S$.

This concept generalizes the concept of affine linear similitudes, which
are affine linear contraction maps $\theta$ satisfying
for all $x,y \in \R^n$
\begin{equation*}
d(\theta(x), \theta(y)) = c d(x,y), \quad c < 1.
\end{equation*}
In particular, the examples discussed in this paper, such as maps
of the form $\theta(x) = 1/(x+m)$, with $m$ a positive integer,
are infinitesimal similitudes. More generally, if $S$ is a compact subset
of $\R^1$ and $\theta:S \to S$ extends to a $C^1$ map defined on an
open neighborhood of $S$ in $\R^1$, then $\theta$ is an infinitesimal
similitude. If $S$ is a compact subset of $\R^2:=\C$ and  $\theta:S \to S$
extends to an analytic or conjugate analytic map defined on an open
neighborhood of $S$ in $\C$, $\theta$ is an infinitesimal similitude.

\begin{thm} (Theorem 1.2 of \cite{N-P-L}.)
Let $\theta_i:S \to S$ for $1 \le i \le N$ be infinitesimal similitudes
and assume that the map $t \mapsto (D\theta_i)(t)$ is a strictly positive
H\"older continuous function on $S$.  Assume that $\theta_i$ is a Lipschitz
map with Lipschitz constant $c_i \le c <1$ and let
$C$ denote the unique, compact, nonempty invariant set such that
\begin{equation*}
C = \cup_{i=1}^N \theta_i(C).
\end{equation*}
Further, assume that $\theta_i$ satisfy
\begin{equation*}
\theta_i(C) \cap \theta_j(C) = \emptyset, \text{ for } 1 \le i,j \le N.
\ i \neq j
\end{equation*}
and are one-to-one on $C$.  Then the Hausdorff dimension of $C$ is
given by the unique $\sigma_0$ such that $r(L_{\sigma_0}) = 1$.
\end{thm}

\section{Examples in one dimension}
\label{sec:1dexps}
\subsection{Continued fraction Cantor sets}
\label{subsec:cantor}
We first consider the problem of computing the Hausdorff dimension of some
Cantor sets arising from continued fraction expansions.  More precisely, given
any number $0<x<1$, we can consider its continued fraction expansion
\begin{equation*}
x = [a_1,a_2,a_3, \ldots] = 
\cfrac{1}{a_1 + \cfrac{1}{a_2 + \cfrac{1}{a_3 + \cdots}}},
\end{equation*}
where $a_1,a_2,a_3, \ldots \in \N$.  We then consider the Cantor set $E_{[m_1,
  \ldots, m_p]}$, of all points in $[0,1]$ where we restrict the coefficients
$a_i$ to the values $m_1, \ldots, m_p$. A number of papers (e.g.,
\cite{Bumby1}, \cite{Bumby2}, \cite{Good}, \cite{Hensley1}, \cite{Hensley2},
\cite{Jenkinson-Pollicott}) have considered this problem in the case of the
set $E_{1,2}$, consisting of all points in $[0,1]$ for which each $a_i$ has
the value $1$ or $2$.  In \cite{Jenkinson-Pollicott}, a method is presented
that computes this dimension to 25 decimal places.  Computations are also
presented in that paper and in \cite{Jenkinson} for other choices of the
values $m_1, \ldots, m_p$. In \cite{Bourgain-Kontorovich}, the Hausdorff
dimension of the Cantor set $E_{2,4,6,8,10}$ is computed to three decimal
places (0.517).

Corresponding to the choices of $m_i$, we associate contraction maps
$\theta_m(x) = 1/(x+m)$.  A key fact is that the Cantor sets we consider
can be generated as limit points of sequences of these contraction maps.  For
example, the set $E_{1.2}$ can be generated using the maps
$\theta_1(x) = 1/(x+1)$ and $\theta_2(x) = 1/(x+2)$ as the set of limit points
of sequences $\theta_{m_1} \ldots \theta_{m_n}(0)$, for $m_1, m_2, \ldots
\in \{1,2\}$.

For $v \in C[0,1]$, we define
\begin{equation}
\label{Lsdef}
(L_sv)(x) = \sum_{j=1}^p \Big|\theta^{\prime}_{m_j}(x)\Big|^s v(\theta_{m_j}(x)).
\end{equation}
Our computations are based on the following result, which we shall prove
in subsequent sections.
\begin{thm}
\label{thm:posev1d}
For all $s >0$, $L_s$ has a unique strictly positive eigenvector $v_s$ with
$L_s v_s = \lambda_s v_s$, where $\lambda_s >0$ and $\lambda_s = r(L_s)$, the
spectral radius of $L_s$.  Furthermore, the map $s \mapsto \lambda_s$ is
strictly decreasing and continuous, and for all $p >0$, $(-1)^p D^{(p)} v_s(x)
>0$ for all $x \in [0,1]$ and
\begin{equation}
\label{prop1}
|D^{(p)} v_s(x)| \le (2s)(2s+1) \cdots (2s+p-1)(\gamma^{-p}) v_s(x),
\\
\end{equation}
where $\gamma = \min_j m_j$.
Finally, the Hausdorff dimension of the Cantor set generated from the maps
$\theta_{m_1}$, $\ldots$, $\theta_{m_p}$ is the unique value of $s$ with
$\lambda_s=1$.
\end{thm}
Note that it follows easily from \eqref{prop1} when $p=1$ and $x_1, x_2 \in
[0,1]$ , that
\begin{equation}
\label{prop2}
v_s(x_2) \le v_s(x_1) \exp(2s|x_2-x_1|/\gamma).
\end{equation}
To see this, write
\begin{equation*}
\log\frac{v_s(x_2)}{v_s(x_1)} = \log v_s(x_2) - \log v_s(x_1) =
\int_{x_1}^{x_2} \frac{d}{dx} \log v_s(x) \, dx = \int_{x_1}^{x_2}
\frac{v_s^{\prime}(x)}{v_s(x)} \, dx,
\end{equation*}
apply the bound in \eqref{prop1}, and exponentiate the result.

To obtain approximations of the dimension of the Cantor sets described
in this section, we first approximate a function $f \in C^2[0,1]$ by
a continuous, piecewise linear function defined on a mesh of interval size
$h$ on $[0,1]$. More specifically, we approximate
$f(x)$, $x_k \le x \le x_{k+1}$ by its piecewise linear interpolant $f^I(x)$
given by
\begin{equation*}
f^I(x) = \frac{x_{k+1}-x}{h} f(x_k) + \frac{x-x_k}{h} f(x_{k+1}),
\quad x_k \le x \le x_{k+1},
\end{equation*}
where the mesh points $x_k$ satisfy $0 = x_0 < x_1 , \dots < x_N =1$,
with $x_{k+1} - x_k = h = 1/N$. 
The goal is to reduce the infinite dimensional eigenvalue
problem to a finite dimensional one.  Standard results for the error
in linear interpolation on an interval $[a,b]$ assert that
\begin{equation*}
f^I(x) - f(x) = \frac{1}{2}(b - x)(x-a) f^{\prime\prime}(\xi)
\end{equation*}
for some $\xi \in [a,b]$.
If $x_{r_j} \le \theta_{m_j}(x) \le x_{r_j+1}$, we get
\begin{equation*}
v_s^I(\theta_{m_j}(x)) =
\frac{[x_{r_j+1} - \theta_{m_j}(x)]}{h} v_s(x_{r_j})
+ \frac{[\theta_{m_j}(x)-x_{r_j}]}{h} v_s(x_{r_j +1}).
\end{equation*}
We can also use the properties in Theorem~\ref{thm:posev1d} to bound the
interpolation error.
Letting $f(x) = v_s(x)$, we obtain from Theorem~\ref{thm:posev1d} that
\begin{equation*}
0 < v_s^{\prime\prime}(\theta_{m_j}(x)) \le 2s (2s+1)\gamma^{-2} v_s(\theta_{m_j}(x)).
\end{equation*}
Using the interpolation error estimate and \eqref{prop2}, we get
for $x_{r_j} \le \theta_{m_j}(x) \le x_{r_j+1}$,
\begin{multline*}
0 < v_s^I(\theta_{m_j}(x)) - v_s(\theta_{m_j}(x)) 
\\
\le  [x_{r_j+1} - \theta_{m_j}(x)][\theta_{m_j}(x) -x_{r_j}]
s (2s+1) \gamma^{-2} \max_{[x_{r_j}, x_{r_j+1}]} \hskip -8.2pt v_s(\xi).
\\
\le 
[x_{r_j+1} - \theta_{m_j}(x)][\theta_{m_j}(x) -x_{r_j}]
s (2s+1) \gamma^{-2} \exp(2sh/\gamma) \, v_s^I(\theta_{m_j}(x)),
\end{multline*}
since the point at which the maximum occurs is within $h$
of either of the two endpoints of the subinterval.

Using this estimate, we have precise upper and lower bounds on the error
in the interval $[x_{r_j},x_{r_j+1}]$ that only depend on the function values of
$v_s$ at $x_{r_j}$ and $x_{r_j+1}$.  Letting
\begin{equation*}
\err_j(x) = [x_{r_j+1} - \theta_{m_j}(x)][\theta_{m_j}(x) -x_{r_j}]
s (2s+1) \gamma^{-2} \exp(2sh/\gamma),
\end{equation*}
we have for each mesh point $x_k$, with
 $x_{r_j} \le \theta_{m_j}(x_k) \le x_{r_j+1}$,
\begin{equation*}
[1 - \err_j(x_k)] v_s^I(\theta_{m_j}(x_k)) \le v_s(\theta_{m_j}(x_k)) 
\le v_s^I(\theta_{m_j}(x_k)).
\end{equation*}

Since for each mesh point $x_k$, $r(L_s) v_s(x_k) = (L_s v_s)(x_k)$, 
we can use \eqref{Lsdef} and the above result to to see that
\begin{multline*}
r(L_s) v_s(x_k) = L_s v_s(x_k) 
= \sum_{j=1}^p \Big|\theta^{\prime}_{m_j}(x_k)\Big|^s v_s(\theta_{m_j}(x_k))
\\
\le \sum_{j=1}^p \Big|\theta^{\prime}_{m_j}(x_k)\Big|^s v_s^I(\theta_{m_j}(x_k))
\end{multline*}
and
\begin{equation*}
r(L_s) v_s(x_k) \ge \sum_{j=1}^p \Big|\theta^{\prime}_{m_j}(x_k)\Big|^s 
[1 - \err_j(x_k)] v_s^I(\theta_{m_j}(x_k)).
\end{equation*}
Let $w_s$ be a vector with $(w_s)_k = v_s(x_k)$, $k=0, \ldots N$.
Define $(N+1) \times (N+1)$ matrices $B_s$ and $A_s$ by
\begin{align*}
(B_s w_s)_k &= \sum_{j=1}^p \Big|\theta^{\prime}_{m_j}(x_k)\Big|^s 
w_s^I(\theta_{m_j}(x_k)),
\\
(A_s w_s)_k &= \sum_{j=1}^p \Big|\theta^{\prime}_{m_j}(x_k)\Big|^s 
[1 - \err_j(x_k)] w_s^I(\theta_{m_j}(x_k)),
\end{align*}
where, if $x_{r_j} \le \theta_{m_j}(x) \le x_{r_j+1}$, we define
\begin{equation*}
w_s^I(\theta_{m_j}(x)) =
\frac{[x_{r_j+1} - \theta_{m_j}(x)]}{h} (w_s)_{r_j}
+ \frac{[\theta_{m_j}(x)-x_{r_j}]}{h} (w_s)_{r_j +1}.
\end{equation*}
Note that all of the entries of
$B_s$ will be nonnegative and since $\err_j(x) = O(h^2)$, this is true for $A_s$
as well, provided $h$ is sufficiently small.

Since $v_s(x_k) >0$ for $k=0, \ldots, N$, we can apply the following result
about nonnegative matrices to see that \begin{equation*}
r(A_s) \le r(L_s) \le r(B_s).
\end{equation*}

\begin{lem}
\label{lem:nonneg}
Let $M$ be an $(N+1) \times (N+1)$ matrix with non-negative
entries and $w$ an $N+1$ vector with strictly positive components.
\begin{align*}
\text{If  } (Mw)_k &\ge \lambda w_k, \quad k =0, \ldots N, \qquad
\text{then }  r(M) \ge \lambda,
\\
\text{If  } (Mw)_k &\le \lambda w_k, \quad k =0, \ldots N, \qquad
\text{then }   r(M) \le \lambda.
\end{align*}
\end{lem}
Since this result is crucial to our approximation scheme, we supply the proof
below to keep our presentation self-contained.  Note, however, that
Lemma~\ref{lem:nonneg} is actually a special case of much more general results
concerning order-preserving, homogeneous cone mappings: see Lemma 2.2 in
\cite{C} and Theorem 2.2 in \cite{B}.
If we let $D$ denote the positive diagonal $(N+1) \times (N+1)$ matrix
with diagonal entries $w_j$, $1 \le j \le N+1$, $r(M) = r(D^{-1}MD)$; and
Lemma~\ref{lem:nonneg} can also be obtained by applying Theorem 1.1 on page 24
of \cite{D} to $D^{-1}MD$.
\begin{proof}
If $(Mw)_k \ge \lambda w_k$, $k =0, \ldots N$, it easily follows that
 $(M^nw)_k \ge \lambda^n w_k$ and so $\|M^n w\|_{\infty} \ge \lambda^n
\|w\|_{\infty}$.  Let $e$ be vector with all $e_i=1$. Then
\begin{equation*}
\|M^n\|_{\infty} = \|M^n e\|_{\infty} \ge \|M^n w\|_{\infty}/\|w\|_{\infty}
\ge \lambda^n.
\end{equation*}
Hence,
\begin{equation*}
r(M) = \lim_{n \rightarrow \infty} \|M^n\|_{\infty}^{1/n} \ge \lambda.
\end{equation*}

If $(Mw)_k \le \lambda w_k$, $k =0, \ldots N$, it easily follows that
 $(M^nw)_k \le \lambda^n w_k$.  Let $k$ be chosen so that
$\|M^n\|_{\infty} = \sum_j (M^n)_{k,j}$. 
Since $[r(M)]^n = r(M^n) \le \|M^n\|_{\infty}$,
\begin{equation*}
\min_j w_j [r(M)]^n \le \min_j w_j \sum_j (M^n)_{k,j}
\le \sum_j (M^n)_{k,j} w_j = (M^nw)_k \le \lambda^n w_k.
\end{equation*}
So,
\begin{equation*}
\min_j w_j \le [\lambda/r(M)]^n w_k.
\end{equation*}
If $r(M) > \lambda$, then letting $n \rightarrow \infty$, we get that
$\min_j w_j \le 0$, which contradicts the fact that all $w_j > 0$.
Hence, $r(M) \le \lambda$.
\end{proof}

As described in Section~\ref{sec:intro}, if $s_*$ denotes the unique value of
$s$ such that $r(L_{s_*}) = \lambda_{s_*} = 1$, then $s_*$ is the Hausdorff
dimension of the set $E_{[m_1, \ldots, m_p]}$.  If we can find a number $s_1$
such that $r(B_{s_1}) \le 1$, then $r(L_{s_1}) \le r(B_{s_1}) \le 1$, and we can
conclude that $s_* \le s_1$.  Analogously, if we can find a number $s_2$ such
that $r(A_{s_2}) \ge 1$, then $r(L_{s_2}) \ge r(A_{s_2}) \ge 1$, and we can
conclude that $s_* \ge s_2$.  By choosing the mesh sufficiently fine, we can
make $s_1-s_2$ small, providing a good estimate for $s_*$.

We can also reduce the number of computations by first iterating the maps
$\theta_{m_i}$ to produce a smaller initial domain that we need to
approximate.  For example, if we seek the Hausdorff dimension of the set
$E_{1,2}$, since $\theta_1([0,1]) = [1/2,1]$ and $\theta_2([0,1]) =
[1/3,1/2]$, the maps $\theta_1$ and $\theta_2$ map $[1/3,1] \mapsto [1/3,1]$,
so we can restrict the problem to this subinterval.  Further iterating, we see
that $\theta_1([1/3,1]) = [1/2,3/4]$ and $\theta_2([1/3,1]) = [1/3,3/7]$.
Hence the maps $\theta_1$ and $\theta_2$ map $[1/3,3/7] \cup [1/2,3/4]$ to
itself and we can further restrict the problem to this domain.

\subsection{Continued fraction Cantor sets -- numerical results}
\label{subsec:cantor-num}
In this section, we report in Table~\ref{tb:t1} the results of the application
of the algorithm described above to the computation of the Hausdorff dimension
of a sample of continued fraction Cantor sets.  Where the true value was known
to sufficient accuracy, it is not hard to check that the rate of convergence
as $h$ is refined is $O(h^2)$. Although the theory developed above does not
apply to higher order piecewise polynomial approximation, since one cannot
guarantee that the approximate matrices have nonnegative entries, we also
report in Table~\ref{tb:t2} and Table~\ref{tb:t22} the results of higher order
piecewise polynomial approximation to demonstrate the promise of this
approach.  In this case, we only provide the results for $B_s$, which does not
contain any corrections for the interpolation error.  In a future paper we
hope to prove that rigorous upper and lower bounds for the Hausdorff dimension
can also be obtained when higher order piecewise polynomial approximations are
used.

\begin{table}[ht]
\footnotesize
\caption{Computation of Hausdorff dimension $s$ of some continued fraction 
Cantor sets.}
\label{tb:t1}
\begin{center}
\begin{tabular}{|c|c|c|c|c|c|}
\hline
Set    &   h  &  lower $s$ & upper $s$ & lower err & upper err  \\
\hline \hline 
E[1,2] &  .0001 & 0.531280505099895  & 0.531280506539767
& 1.18e-09  & 2.63e-10  \\
       &   .00005 & 0.531280505981423 & 0.531280506343388
& 2.96e-10 &  6.62e-11  \\
\hline\hline  
E[1,3] & .0001 & 0.454489076859422 & 0.454489077843624 & 8.02e-10 & 1.82e-10
\\
&   .00005 & 0.454489077459035 & 0.454489077707546 & 2.03e-10 & 4.57e-11 
 \\
\hline \hline 
E[1,4]  & .0001 & 0.411182724095752 & 0.411182724934834 & 6.79e-10 & 1.60e-10 
 \\
&   .00005 & 0.411182724603313 & 0.411182724815117 & 1.71e-10  & 4.03e-11 
\\
\hline\hline 
E[2,3] & .0001 & 0.337436780744847 & 0.337436780851139 & 6.12e-11 & 4.51e-11
\\
&   .00005 & 0.337436780790228 & 0.337436780817793 & 1.58e-11& 1.17e-11 
\\
\hline\hline 
E[2,4] & .0001 & 0.306312767993699 & 0.306312768092506 & 5.91e-11 & 3.97e-11
\\
&   .00005 & 0.306312768039239 &  0.306312768061760 & 1.35e-11 & 8.98e-12 
\\
\hline \hline 
E[3,4] & .0001 & 0.263737482885901 & 0.263737482913807 & 1.15e-11 & 1.64e-11
\\
&   .00005 & 0.263737482894486 & 0.263737482901574 & 2.94e-12& 4.15e-12
\\
\hline \hline 
E[10,11] & .0002 & 0.146921235390446 & 0.146921235393309 & 3.37e-13& 2.53e-12
\\
&   .00005 & 0.146921235390764 & 0.146921235390925 & 1.95e-14& 1.42e-13 
\\
\hline \hline 
E[100,10000] & .0004 & 0.052246592638657 & 0.052246592638662 & 1.88e-15 & 
3.12e-15 \\
&   .0001 & 0.052246592638659 & 0.052246592638659 & 1.25e-16 & 1.25e-16 
\\
\hline \hline 
E[2,4,6,8,10] & .0001 & 0.517357030830725 & 0.517357030987649 & & \\
&   .00005 & 0.517357030911231 & 0.517357030949266 & &  \\
\hline \hline 
E[1,\ldots,10] &  .0001 & 0.925737589218857 & 0.925737591547918 & &  \\
&   .00005 & 0.925737590664670 & 0.925737591246997 & &  \\
\hline\hline 
E[1,3, 5, \ldots, 33]& .0001 & 0.770516007582087 & 0.770516008987138 & & \\
&   .00005 & 0.770516008433225 & 0.770516008784885 & & \\
\hline\hline
E[2, 4, 6, \ldots, 34] & .0001 & 0.633471970121772 & 0.633471970288076 & & \\
&   .00005 & 0.633471970211609 & 0.633471970252711 & & \\
\hline\hline
E[1, \ldots,34] & .0001 & 0.980419623378987 & 0.980419625624112 & & \\
&   .00005 & 0.980419624765058 & 0.980419625326256 & &  \\
\hline 
\end{tabular}
\end{center}
\end{table}

The errors are computed based on the results reported in 
\cite{Jenkinson-Pollicott}. For the last five entries,
we do not have independent results for the true solution correct to
a sufficient number of decimal places to compute the error.

\begin{table}[!ht]
\footnotesize
\caption{Computation of Hausdorff dimension $s$ of E[1,2]
using higher order piecewise polynomials.}
\label{tb:t2}
\begin{center}
\begin{tabular}{|c|c|c|c|c|c|c|}
\hline
degree & h  &  $s$ & error \\
\hline \hline 
1 & .01 & 0.531282991861209 & 2.49 e-06 \\
2 & .02 & 0.531280509905739 & 3.63 e-09  \\
4 & .04 & 0.531280506277708 & 5.03 e-13  \\
5 & .05 & 0.531280506277197 & 7.99 e-15 \\
\hline 
\end{tabular}
\end{center}
\end{table}
In the computations shown using higher order piecewise polynomials, since the
number of unknowns for a continuous, piecewise polynomial of degree $k$ on $N$
uniformly spaced subintervals of width $h$ is given by $k N + 1$, to get a
fair comparison, we have adjusted the mesh sizes so that each computation
involves the same number of unknowns. For this problem, the eigenfunction
$v_s$ is smooth and the computations show a dramatic increase in the accuracy
of the approximation as the degree of the approximating piecewise polynomial
is increased.

\begin{table}[!ht]
\footnotesize
\caption{Computation of Hausdorff dimension $s$ of E[2,4,6,8,10]
using piecewise cubic polynomials.}
\label{tb:t22}
\begin{center}
\begin{tabular}{|c|c|c|c|c|c|c|}
\hline
h  &  $s$ \\
\hline \hline 
0.1    &    0.517357031893604       \\
.05    &    0.517357031040156       \\  
.02    &    0.517357030941730       \\
.01    &    0.517357030937108       \\
.005   &    0.517357030937029       \\
.002   &    0.517357030937018       \\
.001   &    0.517357030937018       \\
\hline 
\end{tabular}
\end{center}
\end{table}


\subsection{An example with less regularity}
\label{subsec:lowreg}

For $0 \le \lambda \le 1$, we consider the maps
\begin{equation}
\label{lrmaps}
\theta_1(x) = \frac{1}{3 + 2 \lambda}(x + \lambda x^{7/2}), \qquad
\theta_2(x) = \frac{1}{3 + 2 \lambda}(x + \lambda x^{7/2}) +
\frac{2 + \lambda}{3 + 2 \lambda},
\end{equation}
which map the unit interval to itself. Both these maps $\in C^3([0,1]$, but
$\notin C^4([0,1]$. We note that because of the lack of regularity, the
methods of \cite{Jenkinson-Pollicott} and \cite{Jenkinson} cannot be applied.
When $\lambda=0$, these maps become
\begin{equation*}
\theta_1(x) = \frac{x}{3}, \qquad
\theta_2(x) = \frac{x}{3} + \frac{2}{3},
\end{equation*}
and the corresponding Cantor set has Hausdorff dimension $\ln 2/\ln 3
\hfill\break \approx 0.630929753571458$.  

Our computations, shown in Table~\ref{tb:t3}, are based on the following
result, which we shall prove in subsequent sections.
\begin{thm}
\label{thm:posev1dreg}
Let
\begin{equation*}
(L_s v)(x) = \sum_{j=1}^2 |\theta_j^{\prime}(x)|^s v(\theta_j(x)),
\end{equation*}
where $\theta_1$ and $\theta_2$ are given by \eqref{lrmaps}.
For all $s >0$, $L_s$ has a unique (up to normalization) strictly positive
$C^2$ eigenvector $v_s$ with $L_s v_s = \lambda_s v_s$, where $\lambda_s >0$
and $\lambda_s = r(L_s)$, the spectral radius of $L_s$.  Furthermore, the map
$s \mapsto \lambda_s$ is strictly decreasing and continuous, and for all $x_1,
x_2 \in [0,1]$, we have the estimate
\begin{multline*}
0 < \frac{D^2 v_s(x)}{v_s(x)} \le s^2 [C_1(\lambda)]^2
\Big(\frac{6+4 \lambda}{4 - 3 \lambda}\Big)
\\
+ s \frac{(6+4 \lambda)^2}{(4 - 3 \lambda)(8 + 11 \lambda)}
\left[C_2(\lambda) + C_1(\lambda) M_0(\lambda) 
\frac{(6+4 \lambda)}{(4 - 3 \lambda)}\right],
\end{multline*}
where $C_1$, $C_2$, and $M_0$ are defined by \eqref{3.6}, \eqref{3.23}, and
\eqref{3.14}, respectively.
Finally, the Hausdorff dimension of the Cantor set generated from the maps
$\theta_1$ and $\theta_2$ is the unique value of $s$ with $\lambda_s= r(L_s)
=1$.
\end{thm}

\begin{table}[!ht]
\caption{Computation of Hausdorff dimension $s$ of less regular examples.} 
\label{tb:t3}
\begin{center}
\begin{tabular}{|c|c|c|c|c|}
\hline
$\lambda$    &   $h$  &  lower $s$ & upper $s$ & upper $s$ - lower $s$ \\
\hline 
$0.0$  &  .0001 &  $0.630929753571458$ & $0.630929753571458$ & $0$ \\
$0.25$ &  .0001 &  $0.691029102085966$ & $0.691029110502743$ & $8.4168e-09$ \\
$0.5$  &  .0001 &  $0.733474587362570$ & $0.733474622222681$ & $3.4860e-08$ \\
$0.75$ &  .0001 &  $0.767207161950980$ & $0.767207292955634$ & $1.3100e-07$  \\
$1.0$  &  .0001 &  $0.796727161816835$ & $0.796727861914653$ & $7.0010e-07$ \\
\hline
\end{tabular}
\end{center}
\end{table}

\section{Examples in two dimensions}
\label{sec:2dexp}

\subsection{The problems}
\label{subsec:problem}

Let $H= \{(x,y) \in \R^2: (x-1/2)^2 + y^2 \le 1/4, y \ge 0\}$.
Writing $z = x+ iy$, we can consider $H$ as a subset of the complex plane.

Let $C_{\R}(H)$ denote the Banach space of real-valued, continuous functions
$f:H \to \R$ in the sup norm.  Let $I_1 = \{b = m + ni: m \in \N,
n \in \Z\}$ and for $b \in I_1$ and $z \in \C$, let $\theta_b(z) = 1/(z+b)$.
If $D = \{z \in \C: |z-1/2| \le 1/2\}$, it is known that for $b \in I_1$,
$\theta_b(D) \subset D$ and $\theta_{b_1}(D\setminus \{0\}) \cap
\theta_{b_2}(D\setminus \{0\}) = \emptyset$ for $b_1,b_2 \in I_1$, $b_1 \neq
b_2$.  Clearly, $\theta_b(D) \subset D\setminus \{0\}$ for $b \in I_1$. If
we identify $H$ with $\{z \in D: \Im(z) \ge 0\}$, and if $b \in I_1$ and
$\Im(b) \ge 0$, $\theta_b(H) \subset \{z \in D: \Im(z) \le 0\}$. Hence
$\overline{1/(z+b)} \in H$ if $z \in H$, $b \in I_1$, and $\Im(b) \ge 0$.
If $z \in H$, $b \in I_1$, and $\Im(b) <0$, one can show that
$\theta_b(z) \in \{z \in D: \Im(z) >0\}$.

Let $C_{\R}(D)$ denote the Banach space of real-valued, continuous functions
$v:D \to \R$ and let $\B$ denote a subset of $I_1$. If $\B$ is a finite set
and $s \ge 0$, one can define a bounded linear map
$L_s: C_{\R}(D) \to C_{\R}(D)$ by
\begin{equation}
\label{Lsdefz}
(L_sv)(z) = \sum_{b \in \B} \Big|\frac{d}{dz} \theta_b(z)\Big|^s
v(\theta_b(z)) = \sum_{b \in \B} \frac{v(\theta_b(z))}{|z+b|^{2s}}.
\end{equation}
If $\B$ is infinite, one can prove (see Section 5 of \cite{A}) that if,
for some $s >0$, the infinite series $\sum_{b \in \B} [1/|z+b|^{2s}]$ converges
for some $z \in D$, then it converges for all $z \in D$ and
$z \mapsto [1/|z+b|^{2s}]$ is a continuous function on $D$.  It then follows
with the aid of Dini's theorem that $L_s$ given by \eqref{Lsdefz} defines
a bounded linear map of $C_{\R}(D) \to C_{\R}(D)$.

If we define $\sigma = \sigma(\B):= \inf\{s >0 \, | \, \exists z \in D
\text{ such that } \sum_{b \in \B} [1/|z+b|^{2s}] < \infty\}$, it follows
from the above remarks that for all $s > \sigma(\B)$, $L_s$ defined
by \eqref{Lsdefz} gives a bounded linear map of $C_{\R}(D) \to C_{\R}(D)$.
If $s= \sigma$, it may or may not happen that
$\sum_{b \in \B} [1/|z+b|^{2s}] < \infty$ for some $z \in D$.  In any event,
it is not hard to prove that if $s > 1$, 
$\sum_{b \in \B} [1/|z+b|^{2s}] < \infty$ for all $z \in D$.

Our computational results are based on the following
theorems, which are special cases of results which we shall prove in 
subsequent sections of the paper.
\begin{thm}
\label{thm:evthm}
Let $\B$ be a subset of $I_1$, and for $s > \sigma(\B) = \sigma$, let
$L_s: C_{\R}(D) \to C_{\R}(D)$ be defined by \eqref{Lsdefz}.  For each
$s>\sigma(\B)$, there exists a unique (to within scalar multiples)
strictly positive Lipschitz eigenvector $v_s$ of $L_s$, i.e., $L_s v_s =
\lambda_s v_s$, where $\lambda_s >0$ and $\lambda_s= r(L_s)$, the
spectral radius of $L_s$ defined by $r(L_s):= \lim_{k \rightarrow
  \infty} \|L_s^k\|^{1/k}$.  If $\bar \B:= \{\bar b \, | \, b \in \B\}$,
then $v_s(\bar z) = v_s(z)$ for all $z \in D$.
If $\B$ is finite, $v_s(x,y)$ is $C^{\infty}$ on $D$ and $x \mapsto v_x(x,y)$
is decreasing for $(x,y) \in D$.
\end{thm}

If $\B \subset I_1$, let $\B_{\infty} = \{\w = (b_1, \ldots, b_k, \ldots)
\, | \, b_j \in \B \ \forall j \ge 1\}$.  Given $z \in D$ and
$\w  (b_1, \ldots, b_k, \ldots) \in \B_{\infty}$, one can prove that
$\lim_{k \rightarrow \infty}(\theta_{b_1} \circ \theta_{b_2} \circ \cdots \circ 
\theta_{b_k})(z) := \pi(\w) \in D$ exists and is independent of $z$.
Define $K = \{\pi(\w) \, | \, \w \in \B_{\infty}\}$.  It is not hard to
prove that $K = \cup_{b \in \B} \theta_b(K)$.  In general $K$ is
not compact, but if $\B$ is finite, $K$ is compact and is the unique
compact, nonempty set $K$ such that $K = \cup_{b \in \B} \theta_b(K)$.
We shall call $K$ the invariant set associated to $\B$.

\begin{thm}
\label{thm:hdim}
Let $\B$ be a subset of $I_1$ and let $K$ be the invariant set
associated to $\B$. The Hausdorff dimension $s_*$ of $K$ is given by
$s_* = \inf\{s >0 \, | \, r(L_s) = \lambda_s <1\}$ and $r(L_{s_*}) = 1$
if $\B$ is finite or $L_{s_*}$ is defined.  The map $s
\mapsto \lambda_s$, $s >1$, is a continuous, strictly decreasing
function for $s > \sigma(\B)$.
\end{thm}
In all examples which we shall consider, $L_s$ is a bounded linear map of
$C_{\R}(D) \to C_{\R}(D)$ for $s = s_*$ and $r(L_{s_*}) = 1$.

Theorems~\ref{thm:evthm} and \ref{thm:hdim} essentially reduce the
problem of estimating the Hausdorff dimension of the invariant set $K$ for
$\B \in I_1$ to the problem of
estimating the value of $s$ for which $r(L_s) =1$.  If $\bar \B = \B$ and
if we use the fact
that $v_s(\bar z) = v_s(z)$ for $z \in H$, we find that
\begin{multline}
\label{breakup}
\lambda_s v_s(z) =  \sum_{\substack{b \in \B, |b| \le R \\ \Im(b) \ge 0}} 
\frac{1}{|z+b|^{2s}} v_s(1/(\bar z + \bar
b))
\\
+ \sum_{\substack{b \in \B, |b| \le R \\ \Im(b) < 0}} \frac{1}{|z+b|^{2s}} v_s(1/(z + b))
+ \sum_{b \in \B, |b| > R} \frac{1}{|z+b|^{2s}} v_s(1/(z + b)).
\end{multline}

If $\B = I_1$, it was stated in \cite{Mauldin-Urbanski} that the
Hausdorff dimension of the invariant set $K$ is $\le 1.885$ and in \cite
{Priyadarshi}, it was shown that the Hausdorff dimension of the set
$K$ is $\ge 1.78$. We shall give much sharper estimates below.
We shall also give estimates for the Hausdorff dimension of the invariant
set of $\B \subset I_1$, for some other choices of $\B$, e.g.,
$\B = I_2:= \{b = m+ ni: m \in \N, n \in \N \cup {0}\}$
and $\B = I_3:= \{b = m+ ni: m \in \{1,2\}\}, n \in \{0, \pm 1, \pm 2\}\}$.

\subsection{Numerical Method}
\label{subsec:method}
For an integer $N>0$, we define a mesh domain $D_h \supset D$, consisting of
squares of sides $h = 1/N$. $D_h$ is chosen to have the property that if
$(x,y) \in D$, then the four corners of the mesh square of side $h$ containing
$(x,y)$ are mesh points in $D_h$.  Although we could simply choose $D_h$ to be
the rectangle $[0,1] \times [0,1/2]$, that choice would add unknowns we do not
use.  We also note that in the case $\B = I_3$, there is a smaller domain $E
\subset D$ such that $\theta_b(E) \subset E \setminus \{0\}$ and although we
have not done so, we could have reduced the size of the approximate problem by
using a mesh domain $E_h \supset E$.

\begin{figure}[htb]
\centerline{\includegraphics[width=14cm]{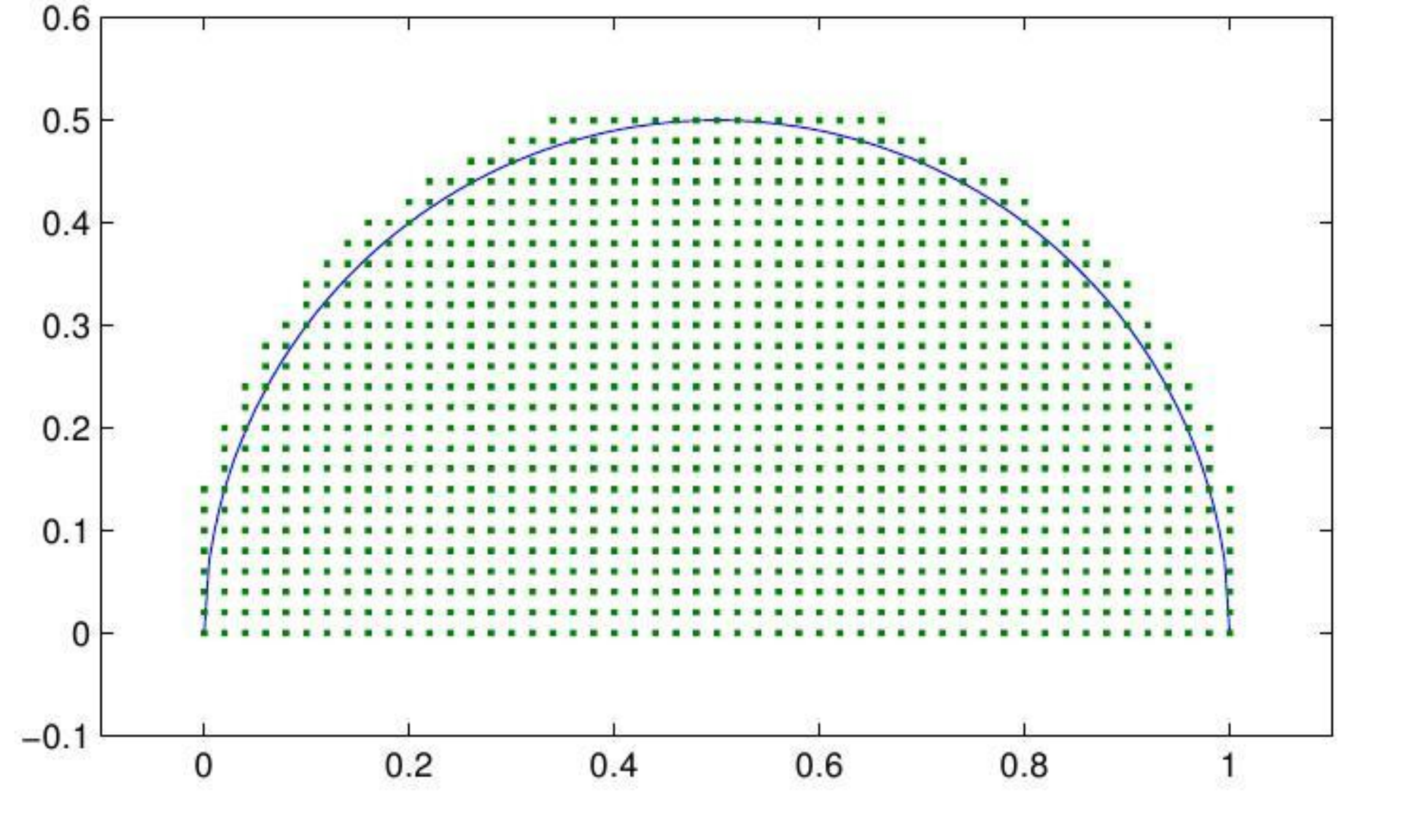}}
\caption[]{Domain $D$ and mesh domain $D_h$}.
\label{fg:f1}
\end{figure}

We then approximate the function $v_s$ by a piecewise bilinear function
defined on the mesh $D_h$ so that we can approximate the infinite
dimensional eigenvalue problem by a finite dimensional one.  In order to
obtain a finite dimensional problem, we also need to restrict the range of
$b$ to the set $|b| \le R$ for a suitably chosen value of $R$ for which
the error in restricting the sum can be given a precise bound which
is sufficiently small.

More precisely, our goal is to again define matrices $A_s$ and $B_s$ such
that
\begin{equation*}
r(A_s) \le r(L_s) \le r(B_s), \quad s >1.
\end{equation*}
We then use the same procedure as for the one-dimensional problems.
If $s_*$ denotes the unique value of $s$ such that 
$r(L_{s_*}) = \lambda_{s_*} = 1$, then $s_*$ is the Hausdorff dimension of
the set $K$.  If we can find a number $s_1$ such that $r(B_{s_1}) \le 1$, then
$r(L_{s_1}) \le r(B_{s_1}) \le 1$, and we can conclude that $s_* \le s_1$.
Analogously, if we can find a number $s_2$ such that $r(A_{s_2}) \ge 1$, then
$r(L_{s_2}) \ge r(A_{s_2}) \ge 1$, and we can conclude that $s_* \ge s_2$.
By choosing the mesh sufficiently fine, we can make $s_1-s_2$ small,
providing a good estimate for $s_*$.

We next describe how to construct the matrices $A_s$ and $B_s$, once
we have defined the mesh $D_h$. To do this, we use the following
results (proved in Section~\ref{sec:mobius}).
\begin{align}
\label{vrelation}
v_s(z_1) &\le v_s(z_2) \exp(\sqrt{5}s|z_1-z_2|), \quad z_1,z_2 \in D,
\\
\label{Dxxbound}
-\frac{s}{4\gamma^2(s+1)} v_s(x,y) &\le  D_{xx} v_s(x,y) \le 
\frac{2s(2s+1)}{\gamma^2} v_s(x,y),
\\
\label{Dyybound}
-\frac{2s}{\gamma^2} v_s(x,y) &\le  D_{yy} v_s(x,y) 
\le \frac{2s(2s+1)}{4\gamma^2} v_s(x,y).
\end{align}
Here we suppose that $v_s(z)$ is as in \eqref{breakup} and that $\Re(b) \ge
\gamma >0$ for all $v \in \B$.

We also use some standard results about bilinear interpolation.
On the mesh square
\begin{equation*}
R_{k,l} = \{(x,y): x_k \le x \le x_{k+1}, y_l \le y \le y_{l+1}\},
\end{equation*}
where $x_{k+1} - x_k = y_{l+1} - y_l =h$, the bilinear interpolant
$f^I(x,y)$ of a function $f(x,y)$ is given by:
\begin{multline*}
f^I(x,y) = \Big[\frac{x_{k+1} -x}{h}\Big]
 \Big[\frac{y_{l+1} -y}{h}\Big]f(x_k,y_l)
+  \Big[\frac{x- x_{k}}{h}\Big]
\Big[\frac{y_{l+1} -y}{h}\Big] f(x_{k+1},y_l)
\\
+  \Big[\frac{x_{k+1} -x}{h}\Big]
\Big[ \frac{y - y_{l}}{h}\Big] f(x_k,y_{l+1})
+  \Big[\frac{x- x_{k}}{h}\Big]
 \Big[\frac{y - y_{l}}{h} \Big] f(x_{k+1},y_{l+1}).
\end{multline*}
The error in bilinear interpolation satisfies for all $(x,y) \in R_{k,l}$
and some points $(a_k,b_l)$ and $(c_k,d_l) \in  R_{k,l}$,
\begin{multline*}
f^I(x,y) - f(x,y)
= 1/2)\Big[ (x_{k+1} -x)(x- x_{k}) (D_{xx}f)(a_k,b_l)
\\
+ (y_{l+1} -y)(y - y_{l}) (D_{yy} f)(c_k,d_l)\Big].
\end{multline*}

For $z = x+iy$, let $f(x,y) = v_s(\theta_{b}(z))$. Further let $z_{k,l} = x_k
+ i y_l$.  If $(\tilde x, \tilde y) = (\Re \theta_{b}(z), \Im  \theta_{b}(z))
\in R_{k,l}$, (which we will sometimes abbreviate by $\theta_{b}(z) \in
R_{k,l}$), we get
\begin{multline*}
v_s^I(\theta_{b}(z)) = 
\Big[\frac{x_{k+1} -\tilde x}{h}\Big]
 \Big[\frac{y_{l+1} -\tilde y}{h}\Big]v_s(z_{k,l})
+  \Big[\frac{\tilde x- x_{k}}{h}\Big]
\Big[\frac{y_{l+1} -\tilde y}{h}\Big] v_s(z_{k+1,l})
\\
+  \Big[\frac{x_{k+1} -\tilde x}{h}\Big]
\Big[ \frac{\tilde y - y_{l}}{h}\Big] v_s(z_{k,l+1})
+  \Big[\frac{\tilde x- x_{k}}{h}\Big]
 \Big[\frac{\tilde y - y_{l}}{h} \Big] v_s(z_{k+1,l+1}).
\end{multline*}
Defining
\begin{equation*}
\Psi_b(z) = 1/(\bar z + \bar b),
\end{equation*}
we have an analogous formula for $v_s^I(\Psi_{b}(z))$, with
$(\tilde x, \tilde y) = (\Re \Psi_{b}(z), \Im  \Psi_{b}(z))$.

We next use inequalities \eqref{vrelation}, \eqref{Dxxbound}, and
\eqref{Dyybound} to obtain bounds on the interpolation error. 
By \eqref{Dxxbound} and \eqref{Dyybound}, we get
for $\theta_{b}(z) = \tilde x + i \tilde y$ where $(\tilde x, \tilde
y) \in R_{k,l}$,
\begin{multline*}
-\Big[\frac{s}{8\gamma^2(s+1)} + \frac{s}{\gamma^2}\Big] 
\left([x_{k+1} - \tilde x][\tilde x -x_{k}]v_s(a_k,b_l) 
+ [y_{l+1} - \tilde y][\tilde y -y_{l}]v_s(c_k,d_l) \right) 
\\
\le  v_s^I(\theta_{b}(z)) - v_s(\theta_{b}(z)) 
\\
\le \frac{s(2s+1)}{\gamma^2}
\left([x_{k+1} - \tilde x][\tilde x -x_{k}]v_s(a_k,b_l)
+ [y_{l+1} - \tilde y][\tilde y -y_{l}]v_s(c_k,d_l)\right).
\end{multline*}
Applying \eqref{vrelation}, we then obtain
\begin{multline*}
-\Big[\frac{s}{8\gamma^2(s+1)} + \frac{s}{\gamma^2}\Big] 
\left([x_{k+1} - \tilde x][\tilde x -x_{k}] 
+ [y_{l+1} - \tilde y][\tilde y -y_{l}] \right) 
\exp(\sqrt{10} sh) v_s^I(\theta_{b}(z))
\\
\le  v_s^I(\theta_{b}(z)) - v_s(\theta_{b}(z)) 
\\
\le \frac{s(2s+1)}{\gamma^2}
\left([x_{k+1} - \tilde x][\tilde x -x_{k}]
+ [y_{l+1} - \tilde y][\tilde y -y_{l}]\right) \exp(\sqrt{10} sh)
v_s^I(\theta_{b}(z)).
\end{multline*}
since any point in $R_{k,l}$ is within $\sqrt{2}h$
of each of the four corners of the square $R_{k,l}$.  An analogous result
holds for $v_s(\Psi_b(z))$.

Using this estimate, we have precise upper and lower bounds on the error
in the mesh square $R_{k,l}$ that only depend on the function values of
$v_s$ at the four corners of the square and the value of $b$.  Letting
\begin{align*}
\err_b^1(\theta_b(z)) &=\Big([x_{k+1} - \tilde x][\tilde x -x_{k}]
+ [y_{l+1} - \tilde y][\tilde y -y_{l}]\Big)
\frac{s (2s+1)}{\gamma^2} \exp(\sqrt{10}sh),
\\
\err_b^2(\theta_b(z)) &=\Big([x_{k+1} - \tilde x][\tilde x -x_{k}]
+ [y_{l+1} - \tilde y][\tilde y -y_{l}]\Big)
\frac{s}{\gamma^2}\Big[\frac{9+8s}{8+8s}\Big] 
\exp(\sqrt{10}sh),
\end{align*}
(where again $\theta_{b}(z) = \tilde x + i \tilde y$),
we have for each mesh point $z_{i,j} = x_i + i y_j$, with
$\theta_{b}(z_{i,j}) \in R_{k,l}$,
\begin{equation*}
[1 - \err_b^1(z_{i,j})] v_s^I(\theta_{b}(z_{i,j})) \le v_s(\theta_{b}(z_{i,j})) 
\le [1 + \err_b^2(z_{i,j})]v_s^I(\theta_{b}(z_{i,j})).
\end{equation*}
Again, the analogous result holds for $v_s(\Psi_b(z))$.
Before using this result as in the one dimensional examples to find upper and
lower matrices that can be used to find upper and lower bounds
on the Hausdorff dimension of the set $K$, we must first deal with the
final expression in \eqref{breakup} where the sum is taken over $|b| > R$.

\begin{lem}
\label{lem:largeb}
For $s >1$, we have
\begin{multline*}
\sum_{b \in I_1, |b| > R} \frac{1}{|z+b|^{2s}} v_s(\theta_b(z))
\le \exp\Big(\frac{2s}{\sqrt{R^2-R}}\Big)\Big(\frac{R}{R-1}\Big)^s
\\
\cdot \left[\Big(\frac{1}{2s-1}\Big)\Big(\frac{1}{R-1}\Big)^{2s-1}
+ \Big(\frac{\pi}{2}\Big)\Big(\frac{1}{s-1}\Big)
\Big(\frac{1}{R-\sqrt{2}}\Big)^{2s-2}
\right]v_s(0).
\end{multline*}
\begin{multline*}
\sum_{b \in I_2, |b| > R} \frac{1}{|z+b|^{2s}} v_s(\theta_b(z))
\le \exp\Big(\frac{2s}{\sqrt{R^2-R}}\Big)\Big(\frac{R}{R-1}\Big)^s
\\
\cdot \left[\Big(\frac{1}{2s-1}\Big)\Big(\frac{1}{R-1}\Big)^{2s-1}
+ \Big(\frac{\pi}{4}\Big)\Big(\frac{1}{s-1}\Big)
\Big(\frac{1}{R-\sqrt{2}}\Big)^{2s-2}
\right]v_s(0).
\end{multline*}
\end{lem}
\begin{proof}
Using \eqref{vrelation}, we have
\begin{equation*}
v_s(\theta_b(z)) \le \exp(2s |\theta_b(z)|) v_s(0).
\end{equation*}
Now for $ z = x + iy \in D$ and $b = m+ in \in I_1$, we have
\begin{multline*}
\min_{(x-1/2)^2 + y^2 \le 1/4} (x+m)^2 + (y+n)^2
\ge \min_{0 \le x \le 1}(x+m)^2 + \min_{|y| \le 1/2} (y+n)^2
\\
\ge m^2 + (|n|-1/2)^2 \ge m^2 + n^2 - |n|.
\end{multline*}
Hence, for $z \in D$,
\begin {equation*}
\frac{1}{|z+b|^2} = \frac{1}{(x+m)^2 + (y+n)^2}
\le \frac{1}{m^2 + n^2 - |n|}.
\end{equation*}
Also, it is easy to check that if $m^2 + n^2 \ge R^2 >1$,
\begin{equation*}
\frac{1}{m^2 + n^2 - |n|} \le \frac{R}{R-1} \frac{1}{m^2 + n^2}
\le \frac{1}{R^2-R}.
\end{equation*}
Hence, for $m^2 + n^2 \ge R^2 >1$ and $z \in D$,
\begin{equation*}
\exp(2s |\theta_b(z)|) \le \exp\Big(\frac{2s}{\sqrt{m^2 + n^2 - |n|}}\Big)
\le \exp\Big(\frac{2s}{\sqrt{R^2-R}}\Big).
\end{equation*}
It follows that
\begin{multline*}
\sum_{b \in I_1, |b| > R} \frac{1}{|z+b|^{2s}} \exp(2s \theta_b(z))
\\
\le \exp\Big(\frac{2s}{\sqrt{R^2-R}}\Big) \Big(\frac{R}{R-1}\Big)^s
\sum_{b \in I_1, |b| > R} \Big(\frac{1}{m^2 + n^2}\Big)^s.
\end{multline*}
Now for $n=0$ and $m \ge R$,
\begin{equation*}
\sum_{m \ge R} \frac{1}{m^{2s}} \le \int_{R-1}^{\infty} \frac{1}{r^{2s}} \, ds
= \frac{1}{2s-1} \Big(\frac{1}{R-1}\Big)^{2s-1}.
\end{equation*}
For $b = m+ in \in I_1$ with $m \ge 1$, $n\ge 1$, and $|b| \ge R$, let
\begin{equation*}
B(m,n) = \{(\xi,\eta): m \le \xi \le m +1, n \le \eta \le n+1\}.
\end{equation*}
Then for $(u,v) \in B(m,n)$,
\begin{equation*}
\frac{1}{(u-1)^2 + (v-1)^2} \ge \frac{1}{m^2 + n^2}.
\end{equation*}
Also,
\begin{multline*}
(u-1)^2 + (v-1)^2 \ge (m-1)^2 + (n-1)^2 = m^2 + n^2 -2 (m+n) + 2
\\
\ge m^2 + n^2 -2 \sqrt{2} \sqrt{m^2 + n^2} + 2 = (\sqrt{m^2 + n^2} -
\sqrt{2})^2 \ge (R- \sqrt{2})^2 \equiv R_1^2.
\end{multline*}
Hence,
\begin{multline*}
\sum_{\substack{m \ge 1, n \ge 1 \\ m^2 + n^2 > R^2}}\ 
\Big(\frac{1}{m^2 + n^2}\Big)^s 
\le \sum_{\substack{m \ge 1, n \ge 1 \\ m^2 + n^2 > R^2}}\ 
\iint\limits_{B(m,n)} \Big(\frac{1}{(u-1)^2 + (v-1)^2}\Big)^s \, du \, dv
\\
\le \iint\limits_{\substack{u \ge 0, v \ge0 \\ u^2 + v^2 \ge R_1^2} }
\Big(\frac{1}{u^2 + v^2}\Big)^s \, du \, dv
= \frac{\pi}{2} \int_{R_1}^{\infty} \frac{1}{r^{2s}} r \, dr
= \frac{\pi}{2} \frac{r^{2-2s}}{2-2s}\Big|_{R_1}^{\infty}
\\
=\frac{\pi}{2} \frac{1}{2s-2}\frac{1}{R_1^{2s-2}}
= \frac{\pi}{4}\frac{1}{s-1}\left(\frac{1}{R- \sqrt{2}}\right)^{2s-2}.
\end{multline*}
A similar argument shows that
\begin{equation}
\label{negn}
\sum_{\substack{m \ge 1, n \le -1 \\ m^2 + n^2 > R^2}}\ 
\Big(\frac{1}{m^2 + n^2}\Big)^s 
\le \frac{\pi}{4}\frac{1}{s-1}\Big(\frac{1}{R- \sqrt{2}}\Big)^{2s-2}.
\end{equation}
Combining these estimates, we obtain
\begin{multline*}
\sum_{b \in I_1, |b| > R} \frac{1}{|z+b|^{2s}} \exp(2s \theta_b(z))
\le \exp\Big(\frac{2s}{\sqrt{R^2-R}}\Big) \Big(\frac{R}{R-1}\Big)^s
\\
\cdot \left[ \frac{1}{2s-1} \Big(\frac{1}{R-1}\Big)^{2s-1}
\hskip-5pt
+ \frac{\pi}{2}\frac{1}{s-1}\Big(\frac{1}{R- \sqrt{2}}\Big)^{2s-2}
\right] =: c_{R,s}.
\end{multline*}
and a similar estimate for the sum over $I_2$, where the factor
$\pi/2$ is replaced by $\pi/4$, since we no longer include the bound
in \eqref{negn}. The lemma follows immediately.
\end{proof}
For $s = 1.85$, evaluating the above expression gives
$0.000796$ for $R=100$, $0.000236$ for $R=200$, and
$0.000117$ for $R=300$.  For $s = 1.60$, the corresponding expression for
the set $I_2$ gives $0.005582$ for $R=100$, $0.002347$ for $R=200$, and
$0.001427$ for $R=300$.

To use these results, we proceed for the finite sum analogously
to Section~\ref{sec:1dexps} to get matrices $A_s$ and $B_s$.  To account for
the terms where $|b| >R$, we note that for the lower matrix $A_s$, we can
simply drop all terms where $|b| >R$, while for the upper matrix $B_s$,
we add to the operator a term of the form
$c_{R,s} v(0)$ so that we are now approximating the operator
\begin{equation*}
(L_s^0v)(z) = \sum_{b \in \B, |b|\le R} \Big|\frac{d}{dz} \theta_b(z)\Big|^s
v(\theta_b(z)) + c_{R,s} v(0) 
= \sum_{b \in \B, |b|\le R} \frac{v(\theta_b(z))}{|z+b|^{2s}}
+ c_{R,s} v(0),
\end{equation*}
where $c_{R,s}$ is one of the constants in Lemma~\ref{lem:largeb}, depending
on whether we are interested in $I_1$ or $I_2$.

\begin{table}[!ht]
\caption{Computation of Hausdorff dimension $s$  
for several values of $h$ and $R$ (rounded to 5 decimal places).}
\label{tb:t4}
\begin{center}
\begin{tabular}{|c|c|c|c|c|}
\hline
Set & $h$ & $R$    &   lower $s$  & upper $s$ \\
\hline \hline
$I_1$ &  $.02$ & $100$ &  1.85459   & 1.85609\\
$I_1$ &  $.01$ & $100$ &  1.85507   & 1.85595 \\
$I_1$ &  $.005$ & $100$ & 1.85518    & 1.85591 \\
\hline
$I_1$ &  $.02$ & $200$ &  1.85503   & 1.85604 \\
$I_1$ &  $.01$ & $200$ &  1.85550   & 1.85589 \\
\hline
$I_1$ &  $.02$ & $300$ & 1.85513   & 1.85603 \\
\hline \hline
$I_2$ &  $.02$ & $100$ & 1.60240    & 1.60677 \\
$I_2$ &  $.01$ & $100$ &  1.60270   & 1.60668\\
$I_2$ &  $.005$ & $100$ & 1.60277   & 1.60666 \\
\hline
$I_2$ &  $.02$ & $200$ &  1.60444   & 1.60654  \\
$I_2$ &  $.01$ & $200$ &  1.60474   & 1.60644 \\
\hline
$I_2$ &  $.02$ & $300$ &  1.60504   & 1.60650 \\
\hline \hline
$I_3$ &  $.02$  &  & 1.53705 &  1.53790  \\
$I_3$ &  $.01$  &  & 1.53754 &  1.53774  \\
$I_3$ &  $.005$ &  & 1.53765 &  1.53770  \\
\hline
\end{tabular}
\end{center}
\end{table}


\section{Existence of $C^m$ positive eigenvectors}
\label{sec:exist}
In this section we shall describe some results concerning existence of
$C^m$ positive eigenvectors for a class of positive (in the sense of
order-preserving) linear operators.  We shall later indicate how one
can often obtain explicit bounds on partial derivatives of the
positive eigenvectors.  As noted above, such estimates play a crucial
role in our numerical method and therefore in obtaining rigorous
estimates of Hausdorff dimension for invariant sets associated with
iterated function systems.

The methods we shall describe can also be applied to the important case of
graph directed iterated function systems, but for simplicity we shall restrict
our attention in this paper to a class of linear operators arising in the
iterated function system case.

The starting point of our analysis is Theorem 5.5 in \cite{E}, which we
now describe for a simple case. If $H$ is a bounded open subset of
$\R^n$ and $m$ is a positive integer, $C^m(\bar H)$ will denote the
set of real-valued $C^m$ maps $f:H \to \R$ such that all partial
derivatives $D^{\alpha} f$ with $|\alpha| \le m$ extend continuously
to $\bar H$. (Here $\alpha = (\alpha_1, \ldots, \alpha_n)$ is a
multi-index with $\alpha_j \ge 0$ for all $j$, $D_j
= \partial/\partial x_j$ for $1 \le j \le n$ and $D^{\alpha} =
D_1^{\alpha_1} \cdots D_n^{\alpha_n}$), $C^m(\bar H)$ is a real Banach
space with $\|f\| = \sup\{|D^{\alpha} f(x)|: x \in H, |\alpha| \le
m\}$.

We say that $H$ is {\it mildly regular} if there exist $\eta >0$ and
$M \ge 1$ such that whenever $x,y \in H$ and $\|x-y\| < \eta$, there exists
a Lipschitz map $\psi:[0,1] \to H$ with $\psi(0) = x$, $\psi(1) = y$
and
\begin{equation}
\label{1.1}
\int_0^1 \|\psi^{\prime}(t)\| \, dt \le M \|x-y\|.
\end{equation}
(Here $\|\cdot\|$ denotes any fixed norm on $\R^n$. If the norm is changed,
\eqref{1.1} remains valid, but with a different constant $M$.)

Let $\B$ denote a finite index set with $|\B| = p$.  For $\beta \in \B$, we
assume
\begin{align*}
&\text{(H5.1)} \ \,
b_{\beta} \in C^m(\bar H) \text{ for all } \beta \in \B \text{ and }
b_{\beta} > 0 \text{ for all } x \in \bar H \text{ and all } \beta \in \B.
\\
&\text{(H5.2)} \ \,
\theta_{\beta}:H \to H \text{ is a } C^m \text{ map for all } \beta \in \B,
\text{ i.e., if } \theta_{\beta}(x) = (\theta_{\beta_1}(x),
\ldots \theta_{\beta_n}(x)),
\\
&\ \, \qquad \qquad \text{ then } \theta_{\beta_k} \in C^m(\bar H)
\text{ for all } \beta \in \B \text{ and for } 1 \le k \le n.
\end{align*}
In (H5.1) and (H5.2), we always assume that $m \ge 1$.

We define $\Lambda: C^m(\bar H) \to C^m(\bar H)$ by
\begin{equation}
\label{1.2}
(\Lambda(f))(x) = \sum_{\beta \in B} b_{\beta}(x) f(\theta_{\beta}(x)).
\end{equation}
For integers $\mu \ge 1$, we define $\B_{\mu} := \{\w = (j_1, \ldots j_{\mu})
\, | \, j_k \in \B \text{ for } 1 \le k \ \le \mu\}$. For
$\w = (j_1, \ldots j_{\mu}) \in \B_{\mu}$, we define $\w_{\mu} = \w$,
$\w_{\mu -1} = (j_1, \ldots j_{\mu-1})$,
$\w_{\mu -2} = (j_1, \ldots j_{\mu-2})$, $\cdots$, $\w_1 = j_1$.  We define
\begin{equation}
\label{1.3}
\theta_{\w_{\mu-k}}(x) = (\theta_{j_{\mu-k}} \circ \theta_{j_{\mu-k-1}} \circ
\cdots \circ \theta_{j_{1}})(x),
\end{equation}
so
\begin{equation}
\label{1.4}
\theta_{\w}(x):= \theta_{\w_{\mu}}(x) = 
(\theta_{j_{\mu}} \circ \theta_{j_{\mu-1}} \circ
\cdots \circ \theta_{j_{1}})(x).
\end{equation}

For $\w \in \B_{\mu}$, we define $b_{\w}(x)$ inductively by 
$b_{\w}(x) = b_{j_1}(x)$
if $\w = (j_1) \in \B:=\B_1$,
$b_{\w}(x) = b_{j_2}(\theta_{j_1}(x)) b_{j_1}(x)$ if $\w = (j_1,j_2) \in \B_2$
and, for $\w = (j_1,j_2, \ldots j_{\mu}) \in \B_{\mu}$,
\begin{equation}
\label{1.5}
b_{\w}(x) = b_{j_\mu}(\theta_{\w_{j_{\mu-1}}}(x)) b_{\w_{\mu -1}}(x).
\end{equation}

If is not hard to show (see \cite{A}, \cite{Bourgain-Kontorovich}, \cite{E})
that
\begin{equation}
\label{1.6}
(\Lambda^{\mu}(f))(x) = \sum_{\w \in \B_{\mu}} b_{\w}(x) f(\theta_\w(x)).
\end{equation}

It is easy to prove (see \cite{E}) that $\Lambda$ defines a bounded linear map
of $C^m(\bar H) \to C^m(\bar H)$.  We shall let $\hat \Lambda$ denote the
complexification of $\Lambda$ and let $\sigma(\hat \Lambda)$ denote the
spectrum of $\hat \Lambda$.  We shall define $\sigma(\Lambda)
= \sigma(\hat \Lambda)$. If all the functions $b_j$ and $\theta_j$ are $C^N$,
then we can consider $\Lambda$ as a bounded linear operator $\Lambda_m:
C^m(\bar H) \to C^m(\bar H)$ for  $1 \le m \le N$, but one should note
that in general $\sigma(\Lambda_m)$ will depend on $m$.

To obtain a useful theory for $\Lambda$, we need a further crucial assumption.
For a given norm $\|\cdot \|$ on $\R^n$, we assume

(H5.3)  There exists a positive integer $\mu$ and a constant $\kappa <1$ such
that for all $\w \in \B_{\mu}$ and all $x,y \in H$,
\begin{equation}
\label{1.7}
\|\theta_\w(x) - \theta_\w(y)\| \le \kappa \|x-y\|.
\end{equation} 

If we define $c = \kappa^{1/\mu} <1$, it follows from (H5.3) that there
exists a constant $M$ such that for all $\w \in B_{\nu}$ and all $\nu \ge 1$,
\begin{equation}
\label{1.8}
\|\theta_\w(x) - \theta_\w(y)\| \le M c^{\nu} \|x-y\| \quad \forall x,y \in H.
\end{equation}
If the norm $\|\cdot \|$ in \eqref{1.8} is replaced by a different norm
$|\cdot |$, \eqref{1.8} remains valid, although with a different constant $M$.
This in turn implies that (H5.3) will also be valid with the same constant
$\kappa$, with $|\cdot|$ replacing $\|\cdot\|$ and with a possibly different
integer $\mu$.

The following theorem is a special case of Theorem 5.5 in \cite{E}.

\begin{thm}
\label{thm:1.1} 
Let $H$ be a bounded open subset of $\R^n$ and assume that $H$ is mildly
regular. Let $X = C^m(\bar H)$ and assume that (H5.1), (H5.2), and (H5.3) are
satisfied (where $m \ge 1$ in (H5.1) and (H5.2)) and that $\Lambda:X \to X$ is
given by \eqref{1.2}.  If $Y= C(\bar H)$, the Banach space of real-valued
continuous functions $f: \bar H \to \R$ and $L:Y \to Y$ is defined by
\eqref{1.2}, then $r(L) = r(\Lambda) >0$, where $r(L)$ denotes the spectral
radius of $L$ and $r(\Lambda)$ denotes the spectral radius of $\Lambda$.  If
$\rho(\Lambda)$ denotes the essential spectral radius of $\Lambda$ (see
\cite{B},\cite{A},\cite{N-P-L}, and \cite{L}), then $\rho(\Lambda) \le c^m
r(\Lambda)$ where $c= \kappa^{1/\mu}$ is as in \eqref{1.8}.  There exists $v
\in X$ such that $v(x) >0$ for all $x \in \bar H$ and
\begin{equation}
\label{1.9}
\Lambda(v) = r v, \qquad r = r(\Lambda).
\end{equation}
There exists $r_1 < r$ such that if $\xi \in \sigma(\Lambda)
\setminus\{r\}$,
then $|\xi| \le r_1$; and $r = r(\Lambda)$ is an isolated point of
$\sigma(\Lambda)$ and an eigenvalue of algebraic multiplicity 1. If $u \in X$
and $u(x) >0 \, \forall x \in \bar H$, there exists a real number $s_u >0$ such
that
\begin{equation}
\label{1.10}
\lim_{k \rightarrow \infty}\left(\frac{1}{r} \Lambda\right)^k (u) = s_u v,
\end{equation}
where the convergence in \eqref{1.10} is in the $C^m$ topology on $X$.
\end{thm}

\begin{remark}
\label{rem:1.2} If $\alpha$ is a multi-index with 
$|\alpha|\le m$, where $m \ge 1$
is as in (H5.1) and (H5.2), it follows from \eqref{1.10} that
\begin{equation}
\label{1.11}
\lim_{k \rightarrow \infty} \left(\frac{1}{r}\right)^k D^{\alpha} \Lambda^k(u) 
=s_u D^{\alpha} v,
\end{equation}
and
\begin{equation}
\label{1.12}
\lim_{k \rightarrow \infty} \left(\frac{1}{r}\right)^k \Lambda^k(u) 
=s_u v,
\end{equation}
where the convergence in \eqref{1.11} and \eqref{1.12} is in the topology
of $C(\bar H)$, the Banach space of continuous functions $f: \bar H \to \R$.
\end{remark}
It follows from \eqref{1.11} and \eqref{1.12} that for any
multi-index $\alpha$ with $|\alpha| \le m$,
\begin{equation}
\label{1.13}
\lim_{k \rightarrow \infty} \frac{(D^{\alpha} \Lambda^k(u))(x)}
{\Lambda^k(u)(x)} = \frac{(D^{\alpha} (v))(x)}
{v(x)},
\end{equation}
where the convergence in \eqref{1.13} is uniform in $x \in \bar H$.
If we choose $u(x) =1$ for all $x \in \bar H$, it follows from \eqref{1.6}
that for all multi-indices $\alpha$ with $|\alpha| \le m$, we have
\begin{equation}
\label{1.14}
\lim_{k \rightarrow \infty}  \frac{D^{\alpha} (\sum_{\w \in B_k} b_\w(x))}
{\sum_{\w \in B_k} b_\w (x)} = \frac{D^{\alpha} v(x)}{v(x)},
\end{equation}
where the convergence in \eqref{1.14} is uniform in $x \in \bar H$. We shall
use \eqref{1.14} in our further work to obtain explicit bounds on 
$\sup\left\{|D^{\alpha} v(x)|/v(x): x \in \bar H\right\}$.

We shall also need information about positive eigenvectors when the index set
$\B$ is countable, but not finite.  Direct analogues of Theorem 5.5 in
\cite{E} exist when $\B$ is countable, but not finite, but such analogues were
not stated or proved in \cite{E}.  Thus we shall make do with less precise
theorems concerning strictly positive Lipschitz eigenvectors.

Given a metric space $(S,d)$, a countable index set $\B$, and continuous maps
$\theta_{\beta}: S \to S$ and $b_{\beta}: S \to \R$ for $\beta \in \B$, we
shall say that the families $\{\theta_{\beta}: \beta \in \B\}$ and
$\{b_{\beta}: \beta \in \B\}$ are {\it uniformly Lipschitz} if there exist
constants $M_1$ and $M_2$, independent of $\beta \in \B$, such that
\begin{equation*}
d(\theta_{\beta}(x), \theta_{\beta}(y)) \le M_1 d(x,y), \ \forall x,y \in S
\text{ and }\forall \beta \in \B
\end{equation*}
and 
\begin{equation*}
|b_{\beta}(x)- b_{\beta}(y)| \le M_2 d(x,y), \ \forall x,y \in S
\text{ and }\forall \beta \in \B.
\end{equation*}
If $S$ is a subset of $\R^N$, we shall take the metric $d$ to be given
by some norm $\|\cdot\|$ on $\R^n$.

For $(S,d)$ a compact metric space, $C(S)$ will denote the real Banach
space of continuous functions $f:S \to \R$ with
$\|f\|:= \sup\{|f(x)|: x \in S\}$. If $b_{\beta}:S \to (0, \infty)$ is
a positive, continuous function for all $\beta \in \B$, we shall assume
that
\begin{equation}
\label{1.15}
\sum_{\beta \in \B} b_{\beta}(x) = b(x) < \infty
\end{equation}
for all $x \in S$ and $x \mapsto b(x)$ is continuous on $S$.  If $D_k$,
$k \ge 1$ is any increasing sequence of finite subsets $D_k \subset \B$ with
$\cup_{k \ge 1} D_k = \B$, Dini's theorem implies that
\begin{equation*}
\lim_{k \rightarrow \infty} \sum_{\beta \in D_k} b_{\beta}(x) = b(x)
\end{equation*}
and that the convergence is uniform in $x \in S$.  Using this fact, one can
define for $f \in C(S)$, $L(f) \in C(S)$ by
\begin{equation}
\label{1.16}
(Lf)(x) = \sum_{\beta \in \B} b_{\beta}(x) f(\theta_{\beta}(x)).
\end{equation}
Here, one is assuming that \eqref{1.15} holds with $x \mapsto b(x)$
continuous on $S$ and that $\theta_{\beta}: S \to S$ is continuous
for all $\beta \in S$, and under these assumptions, $L:C(S) \to C(S)$ is
a bounded linear operator. Also, one can see that for integers $\mu \ge 1$
that
\begin{equation}
\label{1.17}
(L^{\mu}f)(x) = \sum_{\omega \in \B_{\mu}} b_{\omega}(x) f(\theta_{\omega}(x)),
\end{equation}
where $b_{\omega}$ and $\theta_{\omega}$ are as defined in equations
\eqref{1.4} and \eqref{1.5}.

If $M$ is a fixed positive constant, we define a closed cone $K(M;S)
\subset C(S)$ by
\begin{multline}
\label{1.172}
K(M;S) 
\\
= \{f \in C(S) \, | \, f(x) \ge 0 \, \forall x \in S \text{ and }
f(y) \le f(x) \exp(M d(x,y)) \, \forall x,y \in S\}.
\end{multline}

Our next theorem follows easily from Lemma 5.3 in Section 5 of \cite{N-P-L}
and Theorem 5.3 on page 86 of \cite{A}.
\begin{thm}
\label{thm:1.9}
Let $H \subset \R^n$ be a bounded, open subset of $\R^n$ and let the metric
on $\bar H$ be given by a fixed norm $\|\cdot\|$ on $\R^n$. Let $\B$ be a
countable (not finite) index set and assume that $\theta_{\beta}: \bar H \to
\bar H$ and $b_{\beta}: \bar H \to (0, \infty)$, $\beta \in \B$, are
continuous functions and that $\{\theta_{\beta} \, | \, \beta \in \B\}$ and
$\{b_{\beta} \, | \, \beta \in \B\}$ are uniformly Lipschitz. Assume that
for all $x \in \bar H$, $\sum_{\beta \in \B} b_{\beta}(x):= b(x) < \infty$
and that $x \mapsto b(x)$ is continuous. Assume that there exists
an integer $\mu \ge 1$ and a constant $\kappa < 1$ such that for all
$\omega \in \B_{\mu}$, $\Lip(\theta_{\omega}) \le \kappa$.
Assume also that the family of maps $\{x \mapsto \log(b_{\beta}(x)):
\beta \in \B\}$ is uniformly Lipschitz.  Then there exists a constant $A$
such that for all integers $\nu \ge 1$ and for all $\omega \in\B_{\nu}$
\begin{equation}
\label{1.18}
\Lip(\theta_{\omega}) \le A c^{\nu}, \qquad c = \kappa^{1/\mu}.
\end{equation}
Also, for each integer $\nu \ge 1$, the family of maps 
$\{x \mapsto \log(b_{\omega}(x)):
\omega \in \B_{\nu}\}$ is uniformly Lipschitz, so there exists $M_0 > 0$ such
that $b_{\omega} \in K(M_0; \bar H)$(see \eqref{1.17} with $S := \bar H$)
for all $\omega \in \B_{\mu}$.  If $L:C(\bar H) \to C(\bar H)$ is given
by \eqref{1.16} with $\bar H:=S$, $L$ has a strictly positive eigenvector
$v \in K(M_0/(1-\kappa); \bar H)$ with eigenvalue $r = r(L) > 0$. The
algebraic multiplicity of the eigenvalue $r$ equals one, and $r$ is the
only eigenvalue of $L$ of modulus $r$.
\end{thm}
\begin{proof}
  In the following proof, we shall not distinguish in notation between $L$ and
  its complexification $\tilde L$, but of course $\sigma(L)$ refers to the
  spectrum of $\tilde L$ and eigenvalues refer to (possibly complex)
  eigenvalues of $\tilde L$.  We leave to the reader the proof of \eqref{1.17}
  and of the fact that for any $\nu \ge 1$, the set of maps $\{x \mapsto
  \log(b_{\omega}(x)): \omega \in \B_{\nu}\}$ is uniformly Lipschitz.  If we
  start with \eqref{1.17}, rather than \eqref{1.16}, Lemma 5.3 in \cite{N-P-L}
  shows that $L^{\mu}$ has a strictly positive eigenvector $v \in
  K(M_0/(1-\kappa); \bar H)$ with eigenvalue $r = r(L^{\mu}) = [r(L)]^{\mu}>
  0$. If we apply Theorem 5.3, p. 86 in \cite{A} to $L^{\mu}$, we find that
  $r^{\mu}$ is the only eigenvalue of $L^{\mu}$ of modulus $r^{\mu}$ and
  $r^{\mu}$ has algebraic multiplicity one as an eigenvalue of
  $L^{\mu}$. Since $[(1/r)L]^{\mu} v =v$ and $w = (1/r) L (v)$ is also a
  nonzero fixed point of $[(1/r)L]^{\mu}$, it must be that $[(1/r)L] v =
  \lambda v$ for some $\lambda \neq 0$.  We must have $\lambda >0$, because
  $(1/r)L(v)(x) >0$ for all $x \in \bar H$ and $v(x) >0$ for all $x \in \bar
  H$.  This implies that $\lambda^{\mu} v = v$ and $\lambda >0$; so $\lambda
  =1$ and $v$ is a strictly positive eigenvector of $L$ with eigenvalue
  $r(L)$.  If we now apply Theorem 5.3 of \cite{A} to $L$, we find that $r(L)$
  is an eigenvalue of $L$ of algebraic multiplicity one and $r(L)$ is the only
  eigenvalue of $L$ of modulus $r(L)$. Note, however, that $\sigma(L)$ may
well contain elements of modulus $r(L)$.
\end{proof}
\begin{cor}
\label{cor:1.4}
Let assumptions and notation be as in Theorem~\ref{thm:1.9}. Assume in addition
that $H$ is convex and that $b_{\beta} \in C^1(\bar H)$ for all $\beta \in \B$.
For each integer $\nu \ge 1$, define
\begin{equation*}
M_{\nu} = \sup \{\frac{\|\nabla b_{\omega}(x)\|}{b_{\omega}(x)}: \omega
\in \B_{\nu}, x \in \bar H\},
\end{equation*}
where we use the Euclidean norm on $\R^n$.  Define $M_{\infty}$ by
$M_{\infty} = \lim \inf_{\nu \rightarrow \infty} M_{\nu}$.
If $v$ is a strictly positive eigenvector of $L$ in \eqref{1.15}, $v
\in K(M_{\infty}, \bar H)$.
\end{cor}
\begin{proof}
If $x,y \in H$, then because we assume that $H$ is convex,
$x^t:=(1-t)x + ty \in H$ for $0 \le t \le 1$. (We use $t$ as a superscript
here.)  If $\omega \in \B_{\nu}$, $\nu \ge 1$, it follows that
\begin{multline*}
|\log(b_{\omega}(y) - \log(b_{\omega}(x) | = \left|\int_0^1
\frac{d}{dt} \log b_{\omega}(x^t) \, dt \right|
\\
= \left|\int_0^1 \frac{\nabla b_{\omega} \cdot (y-x)}{b_{\omega}(x^t)}
\, dt \right|
\le \int_0^1 \frac{\|\nabla b_{\omega}\| \|y-x\|}{b_{\omega}(x^t)}
\, dt.
\end{multline*}
This shows that $x \mapsto \log b_{\omega}(x)$ is Lipschitz on $\bar H$
with Lipschitz constant $\le M_{\nu}$, so $b_{\omega} \in K(M_{\nu}; \bar H)$
for $\omega \in \B_{\nu}$. If $ A c^{\nu} < 1$, the argument used in the proof
of Lemma 5.3 in \cite{N-P-L} now shows that $v \in K(M_{\nu}/(1 - A c^{\nu});
\bar H)$.  Since $\lim_{\nu \rightarrow \infty} A c^{\nu}=0$, we conclude that
$v \in K(M_{\infty}; \bar H)$.
\end{proof}
\begin{remark}
\label{rem:1.5}
Under slightly stronger assumptions, the bounded linear operator $L:C(S) \to
C(S):=Y$ induces a bounded linear operator $\Lambda: X \to X$, where
$X$ denotes the Banach space of Lipschitz functions $f:S \to \R$.  One
can prove that $r(\Lambda) = r(L)$ and $\rho(\Lambda) < r(\Lambda)$, where
$\rho(\Lambda)$ denotes the essential spectral radius of $\Lambda$. See
\cite{N-P-L} and Section 5 of \cite{A} for details.
\end{remark}

In some applications, the domain $H$ in Theorem~\ref{thm:1.1} or
Theorem~\ref{thm:1.9} possesses some symmetry or symmetries, and this is
often reflected in a corresponding symmetry of the unique, normalized
positive eigenvector $v$ in these theorems.
\begin{cor}
\label{cor:1.6}
Let assumptions and notation be as in Theorem~\ref{thm:1.1} or
Theorem~\ref{thm:1.9} and let $v$ denote the unique normalized strictly
positive eigenvector of $L$ in Theorem~\ref{thm:1.1} or
Theorem~\ref{thm:1.9}. Assume that $\pi: \bar H \to \bar H$ is a $C^m$
map, $m \ge 1$, such that $\pi(\pi(x)) =x$ for all $x \in \bar H$.  Assume that
there exists a one-one map $\beta \mapsto \bar \beta$ of $\B$ onto
$\B$ such that $\pi(\theta_{\bar \beta}(x)) = \theta_{\beta}(\pi(x))$
and $b_{\beta}(\pi(x)) = b_{\bar \beta}(x)$ for all $\beta \in \B$ and
all $x \in \bar H$. It then follows that $v(\pi(x)) = v(x)$ for all $x
\in \bar H$.
\end{cor}
\begin{proof}
Define $w(x) = v(\pi(x))$, so $w(\theta_{\beta}(x)) = v(\pi(
\theta_{\beta}(x)))$
for all $\beta \in \B$ and $x \in \bar H$.  If $\lambda:= r(\Lambda)$,
it follow that
\begin{equation*}
\lambda v(\pi(x)) = \lambda w(x) 
= \sum_{\beta \in \B} b_{\beta}(\pi(x)) v (\theta_{\beta}(\pi(x)))
= \sum_{\beta \in \B} b_{\bar \beta}(\pi(x)) v (\theta_{\bar \beta}(\pi(x))).
\end{equation*}
Since $b_{\bar \beta}(\pi(x)) = b_{\beta}(x)$ and $v (\theta_{\bar
  \beta}(\pi(x))) = v(\pi(\theta_{\beta}(x))) = w(\theta_{\beta}(x))$,
we find that
\begin{equation*}
\lambda w(x)  = \sum_{\beta \in \B} b_{\beta}(x) w (\theta_{\beta}(x))
= (\Lambda(w))(x),
\end{equation*}
so
\begin{equation*}
w_1(x) = \frac{v(x) + w(x)}{2} = \frac{v(x) + v(\pi(x))}{2}
\end{equation*}
is a strictly positive eigenvector of $\Lambda$ with eigenvalue $\lambda$ and
$w_1(\pi(x)) = w_1(x)$ for all $x \in \bar H$.  By the uniqueness of the
strictly positive eigenvector of $\Lambda$, there exists $\mu >0$ such that
$v(x) = \mu w(x)$ for all $x \in \bar H$, which implies that
$v(\pi(x)) = v(x)$ for all $x \in \bar H$.
\end{proof}
\begin{remark}
\label{rem:1.6}
Suppose that $H$ is a bounded, open mildly regular subset of $\C = \R^2$ and
for all $z=x+iy$, $\bar x = x-iy \in H$.  Define $\pi(z) = \bar z$ and
assume that the hypotheses of Corollary~\ref{cor:1.6} are satisfied, so
$v(\bar z) = v(z)$ for all $z \in H$.  Using this fact, the original
eigenvalue problem can be reduced to an equivalent problem on the closure
of $H_+$, where $H_+ = \{z \in H | \Im(z) >0\}$.
\end{remark}

\section{Estimates for derivatives of $v_s$: 
the one dimensional case}
\label{sec:1d-deriv}
Throughout this section, we shall assume that $H \subset \R^1$ is a bounded,
open set such that $H = \cup_{j=1}^n (c_j,d_j)$, where
$[c_j,d_j] \cap [c_k,d_k] = \emptyset$ whenever $1 \le j \le n$,
$1 \le k \le n$, and $j \neq k$. $\B$ will denote a finite index set.
For $\beta \in \B$ and some integer $m \ge 1$, we assume

\noindent (H6.1:) For each $\beta \in \B$, $b_{\beta} \in C^m(\bar H)$,
$\theta_{\beta} \in C^m(\bar H)$, $b_{\beta}(x) >0$ for all $x \in \bar H$ and
$\theta_{\beta}(H) \subset H$.  There exist an integer $\mu \ge 1$ and a real
number $\kappa <1$ such that for all $\omega \in \B_{\mu}:= \{(\beta_1,
\beta_2, \cdots, \beta_{\mu}) \, | \, \beta_j \in \B$ for $1 \le j \le \mu\}$
and for all $x,y \in \bar H$, $|\theta_{\omega}(x) - \theta_{\omega}(y)| \le
\kappa |x-y|$, where $\theta_{\omega}:= \theta_{\beta_{\mu}} \circ
\theta_{\beta_{\mu-1}} \circ \cdots \circ \theta_{\beta_1}$ for $\omega =
(\beta_1, \beta_2, \cdots, \beta_{\mu}) \in \B_{\mu}$.

As in Section~\ref{sec:exist}, we define $Y = C(\bar H)$ and $X_m = C^m(\bar
H)$.  Assuming (H6.1), we define for $s \ge 0$, a bounded linear operator
$L_s:Y \to Y$ by
\begin{equation}
\label{3.1} (L_s f)(x) = \sum_{\beta \in \B} [b_{\beta}(x)]^s f(\theta_{\beta}(x)).
\end{equation} 
As in Section~\ref{sec:exist}, $L_s(X_m) \subset X_m$ and $L_s |_{X_m}$
defines a bounded linear map of $X_m$ to $X_m$ which we denote by $\Lambda_s$.
Theorem~\ref{thm:1.1} is now directly applicable (replace $b_{\beta}(x)$ in
Theorem~\ref{thm:1.1} by $b_{\beta}(x)^s$) and yields information about
$\sigma(\Lambda_s)$.  In particular, $r(L_s) = r(\Lambda_s) >0$ and there
exists a unique (to with normalization) strictly positive, $C^m$ eigenvector
$v_s$ of $\Lambda_s$ with eigenvalue $r(\Lambda_s)$.

If $\omega = (\beta_1, \beta_2, \ldots, \beta_p) \in \B_p$, recall that
we define $b_{\omega}(x)$ by
\begin{equation*}
b_{\omega}(x) = b_{\beta_p}(\theta_{\beta_{p-1}} \circ \theta_{\beta_{p-2}}
\circ \cdots \circ \theta_{\beta_{1}}(x)) \cdots
b_{\beta_3}((\theta_{\beta_{2}} \circ \theta_{\beta_{1}})(x))
b_{\beta_2}((\theta_{\beta_{1}}(x)) b_{\beta_1}(x),
\end{equation*}
and
\begin{equation}
\label{3.2}
(L_s^p f)(x) = \sum_{\omega \in \B_p} [b_{\omega} (x)]^s f(\theta_{\omega}(x)).
\end{equation}

Notice that $L_s^p$ is of the same form as $L_s$ and Theorem~\ref{thm:1.1}
is also directly applicable to  $L_s^p$.  Since $v_s$ is also an eigenvector
of $L_s^p$, we could also work with \eqref{3.2} instead of \eqref{3.1}:
$\B_p$ is an index set corresponding to $\B$, $b_{\omega}$, $\omega \in
\B_p$, corresponds to $b_{\beta}$, $\beta \in \B$, and $\theta_{\omega}$,
$\omega \in \B_p$, corresponds to $\theta_{\beta}$, $\beta \in \B$.  In
our subsequent work in this section, we shall start from \eqref{3.1}, but
the theorems we shall obtain translate directly to the case of using
\eqref{3.2}; and indeed it is sometimes desirable to start from
\eqref{3.2} for some $p >1$.

If $m$ is as in (H6.1) and $k$ is a positive integer with $k \le m$,
we define $D = d/dx$, so $(Df)(x) = f^{\prime}(x)$ and
$(D^kf)(x) = f^{(k)}(x)$.  We are interested in obtaining estimates for
\begin{equation}
\label{3.3}
\sup \{|D^k v_s(x)|/v_s(x) : x \in \bar H\}.
\end{equation}



Hypothesis (H6.1) implies that there exist constants $M >0$ and $c =
\kappa^{1/\mu}$, (so $c <1$), such that for all $x,y \in \bar H$, for all
integers $\nu \ge 1$ and all $\w \in \B_{\nu}$,
\begin{equation}
\label{3.4}
|\theta_\w(x) - \theta_\w(y)| \le M c^{\nu} |x-y|.
\end{equation}
It follows that if we define $\epsilon_0 =1$, and for $\nu \ge 1$,
\begin{equation}
\label{3.5}
\epsilon_{\nu}:= \sup \Big\{|\theta_\w(x) - \theta_\w(y)|/|x-y| :
\w \in \B_{\nu} \text{ and } x,y \in H, x \neq y \Big\},
\end{equation}
we have that $\epsilon_{\nu} \le M c^{\nu}$ and $\sum_{\nu =1}^{\infty}
\epsilon_{\nu} < \infty$.

We define constants $C_1$ and $C_1(s)$ for $s > 0$ by
\begin{equation}
\label{3.6}
C_1 = \sup \Big\{\frac{|D b_{\beta}(x)|}{b_{\beta}(x)}: \beta \in \B, x \in H
\Big\}
\end{equation}
and
\begin{equation*}
C_1(s) = \sup \Big\{\frac{|D b_{\beta}(x)^s|}{b_{\beta}(x)^s}: 
\beta \in \B, x \in H \Big\}.
\end{equation*}
A calculation shows that for all $\w \in \B_{\nu}$, $\nu \ge 1$,
\begin{equation}
\label{3.7}
\frac{|D b_{\w}(x)^s|}{b_{\w}(x)^s}
= s \frac{D b_{\w}(x)}{b_{\w}(x)},
\end{equation}
so
\begin{equation}
\label{3.8}
C_1(s) = s C_1, \text{ for } s >0.
\end{equation}

We begin by considering \eqref{3.3} for the case $k=1$. In our applications,
we shall only need the case $s >0$, so we shall restrict our attention to
this case.
\begin{thm}
\label{thm:3.1}
Assume that (H6.1) is satisfied.
If $C_1$ and $\epsilon_{\nu}$, $\nu \ge 1$
are as in \eqref{3.6} and \eqref{3.5} respectively, then, for $s >0$,
\begin{equation}
\label{3.9}
\sup \Big\{\frac{|D v_s(x)|}{v_s(x)} : x \in \bar H\Big\} \le C_1 s
\Big(\sum_{\nu =0}^{\infty} \epsilon_{\nu}\Big).
\end{equation}
If $\delta \in \{0,1\}$ and $(-1)^{\delta}(D b_\w)(x)/b_\w(x))\le 0$
for all $\w \in \B_{\nu}$, all $\nu \ge 1$ and all $x \in \bar H$, then
$(-1)^{\delta}D v_s(x) \le 0$ for all $x \in \bar H$ and all $s >0$.
\end{thm}
\begin{proof}
For a fixed $\w = (j_1,j_2, \ldots,j_\nu) \in \B_{\nu}$, for notational
convenience define $\xi_k(x) = 
(\theta_{j_k} \circ \theta_{j_{ k-1}} \circ \cdots \circ \theta_{j_1})(x)$ and
$\xi_0(x) = x$ for all $x \in \bar H$, so
\begin{equation*}
b_\w(x) = b_{j_{\nu}}(\xi_{\nu-1}(x)) b_{j_{\nu-1}}(\xi_{\nu-2}(x)) \cdots
b_{j_1}(\xi_{0}(x)).
\end{equation*}
By the chain rule and product rule for differentiation, we find
\begin{equation}
\label{3.10}
D b_\w(x) = b_\w(x) \left[\sum_{k=0}^{\nu -1} \frac{b_{j_{k+1}}^{\prime}
(\xi_k(x))\xi_k^{\prime}(x)}{b_{j_{k+1}}(\xi_k(x))}\right].
\end{equation}
It follows from \eqref{3.10} that
\begin{equation}
\label{3.11}
\frac{|D b_\w(x)|}{b_\w(x)} \le \sum_{k=0}^{\nu-1} C_1 \epsilon_k \le C_1
\sum_{k=0}^{\infty} \epsilon_k.
\end{equation}
Using \eqref{3.7}, we see that
\begin{equation}
\label{3.12}
\frac{|D (b_\w(x)^s)|}{b_\w(x)^s} \le s C_1 \sum_{k=0}^{\infty} \epsilon_k,
\end{equation}
and it follows that
\begin{equation*}
\frac{|D (\sum_{\w \in \B_{\nu}} (b_\w(x)^s)|}{\sum_{\w \in \B_{\nu}}
  (b_\w(x)^s)}
\le \frac{\sum_{\w \in \B_{\nu}}|D (b_\w(x)^s)|}{\sum_{\w \in \B_{\nu}}
  b_\w(x)^s}
\le s C_1 \sum_{k=0}^{\infty} \epsilon_k.
\end{equation*}
Using Theorem~\ref{thm:1.1} and \eqref{1.14}, we conclude by letting
$\nu \rightarrow \infty$ that
\begin{equation}
\label{3.13}
\frac{|D v_s(x)|}{v_s(x)} \le s C_1 \sum_{k=0}^{\infty} \epsilon_k.
\end{equation}
The final statement of the theorem follows by a similar argument, using
\eqref{3.7} and \eqref{1.14}.  Details are left to the reader.
\end{proof}

To facilitate the analysis of \eqref{3.3} when $k=2$, we first prove
a lemma.
\begin{lem}
\label{lem:3.2}
Assume that (H6.1) is satisfied with $m \ge 2$ and define $M_0$ by
\begin{equation}
\label{3.14}
M_0 = \sup \{|D^2 \theta_{\beta}(x)|: \beta \in \B, x \in \bar H\}.
\end{equation}
If $\w \in \B_k$, then
\begin{equation}
\label{3.15}
|D^2 \theta_{\w}(x)| \le M_0 \sum_{j=0}^{k-1} \epsilon_j^2 \epsilon_{k-j-1}.
\end{equation}
\end{lem}
\begin{proof}
  Recall, for a fixed $\w = (\beta_1, \beta_2, \cdots, \beta_k) \in \B_k$, we
  have defined $\xi_0(x) = x$ and $\xi_p(x)$, $1 \le p \le k$ by
\begin{equation*}
\xi_p(x) = (\theta_{\beta_{p}} \circ \theta_{\beta_{p-1}} \circ \cdots
\circ \theta_{\beta_{1}})(x).
\end{equation*}
Hence,
\begin{equation}
\label{3.16}
\xi_k^{\prime}(x) = \prod_{j=1}^k \theta_{\beta_j}^{\prime}(\xi_{j-1}(x)).
\end{equation}
By differentiating \eqref{3.16}, we find that
\begin{equation}
\label{3.17}
\xi_k^{\prime\prime}(x) = \sum_{j=0}^{k-1}(\theta_{\beta_{j+1}}^{\prime\prime}
(\xi_j(x))\xi_j^{\prime}(x))\Big[\prod_{\substack{p=1 \\ p \neq j+1}}^k
\theta_{\beta_p}^{\prime}(\xi_{p-1}(x))\Big].
\end{equation}
If there exist two distinct integers $p$ and $q$ with $1 \le p \le k$
and $1 \le q \le k$ such that $\theta_{\beta_p}^{\prime}(\xi_{p-1}(x))
=0$ and $\theta_{\beta_q}^{\prime}(\xi_{q-1}(x))=0$, \eqref{3.17} implies
that $\xi_k^{\prime\prime}(x) =0$.  If there exists exactly
one integer $q$ with  $1 \le q \le k$ such that
$\theta_{\beta_q}^{\prime}(\xi_{q-1}(x))=0$, \eqref{3.17} implies that
\begin{equation}
\label{3.18} 
\xi_k^{\prime\prime}(x) = \left[\theta_{\beta_q}^{\prime\prime}(\xi_{q-1}(x))
\xi_{q-1}^{\prime}(x)\right]\left[\prod_{p=1}^{q-1} 
\theta_{\beta_p}^{\prime}(\xi_{p-1}(x))\right]\left[
\prod_{p=q+1}^{k}  \theta_{\beta_p}^{\prime}(\xi_{p-1}(x))\right],
\end{equation}
where we interpret $\prod_{p=q+1}^{k}  \theta_{\beta_p}^{\prime}(\xi_{p-1}(x))
= 1 = \epsilon_0$ if $q =k$.  It follows from \eqref{3.18} that
\begin{equation}
\label{3.19}
|\xi_k^{\prime\prime}(x)| \le M_0 \epsilon_{k-q} \epsilon_{q-1}^2.
\end{equation}

If there does not exist $p$, $1 \le p \le k$, with 
$\theta_{\beta_p}^{\prime}(\xi_{p-1}(x))=0$, \eqref{3.16} implies that
$\xi_k^{\prime}(x) \neq 0$, and we obtain from \eqref{3.17} that
\begin{equation}
\label{3.20}
\xi_k^{\prime\prime}(x) = \sum_{q=1}^k
(\theta_{\beta_q}^{\prime\prime}(\xi_{q-1}(x))
\xi_{q-1}^{\prime}(x))
\left[\prod_{p=1}^{q-1} 
\theta_{\beta_p}^{\prime}(\xi_{p-1}(x))\right]\left[
\prod_{p=q+1}^{k}  \theta_{\beta_p}^{\prime}(\xi_{p-1}(x))\right].
\end{equation}
Then \eqref{3.20} implies that
\begin{equation}
\label{3.21}
|\xi_k^{\prime\prime}(x)| \le \sum_{q=1}^k M_0 \epsilon_{k-q} \epsilon_{q-1}^2,
\end{equation}
which completes the proof.
\end{proof}

If $M$ and $c$, $0 <c <1$, are chosen as in \eqref{3.4}, Lemma~\ref{lem:3.2}
implies that for all $\w \in \B_k$ and $k \ge 2$
\begin{multline}
\label{3.22} |D^2 \theta_{\w}(x)| \le M_0 M c^{k-1}
+ \sum_{q=2}^{k-1} M_0 M^3 c^{k-2} c^q + M_0 M^2 c^{2k-2}
\\
= M_0 M c^{k-1}(1 + M c^{k-1}) + M_0 M^3 c^k \frac{1 - c^{k-2}}{1-c}.
\end{multline}

\begin{lem}
\label{lem:3.3}
Assume that (H6.1) is satisfied with $m \ge 2$ and define a constant $C_2$ by
\begin{equation}
\label{3.23} 
C_2 = \sup\Big\{\frac{|D^2 b_{\beta}(x)|}{b_{\beta}(x)}: \beta \in \B, 
x \in H \Big\}.
\end{equation}
Let $C_1, C_2$, and $M_0$ be as in \eqref{3.6}, \eqref{3.23}, and
\eqref{3.14}.  Then for $s >0$, and for $\w \in \B_{\nu}$, with $\nu \ge 1$,
we have the estimates
\begin{equation}
\label{3.24}
\frac{D^2(b_{\w}(x)^s)}{b_{\w}(x)^s} \le
s^2 C_1^2\Big(\sum_{k=0}^{\infty} \epsilon_k\Big)^2
+ s \Big(\sum_{k=0}^{\infty} \epsilon_k^2\Big)\Big[C_2 + C_1 M_0 
\Big(\sum_{k=0}^{\infty} \epsilon_k\Big)\Big]
\end{equation}
and
\begin{equation}
\label{3.25}
\frac{D^2(b_{\w}(x)^s)}{b_{\w}(x)^s} \ge
- s \Big(\sum_{k=0}^{\infty} \epsilon_k^2\Big)\Big[C_1^2 + C_2 + C_1 M_0 
\Big(\sum_{k=0}^{\infty} \epsilon_k\Big)\Big].
\end{equation}
\end{lem}
\begin{proof}
For a fixed $\omega = (j_1, j_2, \ldots, j_{\nu}) \in \B_{\nu}$, let
$\xi_k(x)$ be as defined in the proof of Theorem~\ref{thm:3.1}. A
calculation gives
\begin{equation}
\label{3.26}
\frac{D^2(b_{\w}(x)^s)}{b_{\w}(x)^s} = s(s-1) 
\left(\frac{D(b_{\w}(x))}{b_{\w}(x)}\right)^2 + s 
\frac{D^2(b_{\w}(x))}{b_{\w}(x)}.
\end{equation}
Using \eqref{3.10} we see that
\begin{equation}
\label{3.27}
D^2 b_{\w}(x) = b_{\w}(x)\left(\frac{D(b_{\w}(x)}{b_{\w}(x)}\right)^2
+ b_{\w}(x) D \left( \sum_{k=0}^{\nu-1}
\frac{b_{j_{k+1}}^{\prime}(\xi_k(x)) \xi_k^{\prime}(x)}
{b_{j_{k+1}}(\xi_k(x))}\right),
\end{equation}
which gives
\begin{equation}
\label{3.28}
\frac{D^2 (b_{\w}(x)^s)}{ b_{\w}(x)^s}
 = s^2\left(\frac{D(b_{\w}(x)}{b_{\w}(x)}\right)^2
+ s D \left( \sum_{k=0}^{\nu-1}
\frac{b_{j_{k+1}}^{\prime}(\xi_k(x)) \xi_k^{\prime}(x)}
{b_{j_{k+1}}(\xi_k(x))}\right) := T_1 + T_2.
\end{equation}
A calculation gives
\begin{multline}
\label{3.29}
s D \left( \sum_{k=0}^{\nu-1}
\frac{b_{j_{k+1}}^{\prime}(\xi_k(x)) \xi_k^{\prime}(x)}
{b_{j_{k+1}}(\xi_k(x))}\right)
\\
= s \sum_{k=0}^{\nu-1}\frac{b_{j_{k+1}}^{\prime\prime}(\xi_k(x)) 
(\xi_k^{\prime}(x))^2 + b_{j_{k+1}}^{\prime}(\xi_k(x)) \xi_k^{\prime\prime}(x)}
{b_{j_{k+1}}(\xi_k(x))}
- s \sum_{k=0}^{\nu-1}\frac{[b_{j_{k+1}}^{\prime}(\xi_k(x)) 
\xi_k^{\prime}(x)]^2}{[b_{j_{k+1}}(\xi_k(x))]^2}.
\end{multline}
It follows that
\begin{equation}
\label{3.30}
T_2 \le s\left(\sum_{k=0}^{\nu-1} C_2 \epsilon_k^2 + \sum_{k=0}^{\nu-1}
C_1 |\xi_k^{\prime\prime}(x)|\right)
\end{equation}
and
\begin{equation}
\label{3.31}
T_2 \ge -s\left(\sum_{k=0}^{\nu-1} C_2 \epsilon_k^2 + \sum_{k=0}^{\nu-1}
C_1 |\xi_k^{\prime\prime}(x)|
+ \sum_{k=0}^{\nu-1} C_1^2 \epsilon_k^2\right).
\end{equation}
Lemma~\ref{lem:3.2} implies that
\begin{equation}
\label{3.32}
 \sum_{k=0}^{\nu-1} |\xi_k^{\prime\prime}(x)| \le M_0
\sum_{k=0}^{\infty} \Big(\sum_{q=1}^k \epsilon_{k-q} \epsilon_{q-1}^2\Big)
= M_0 \Big(\sum_{q=0}^{\infty} \epsilon_{q}^2\Big)\Big(\sum_{k=0}^{\infty}
\epsilon_k\Big),
\end{equation}
so we obtain from \eqref{3.30} and \eqref{3.31} that
\begin{equation}
\label{3.33}
T_2 \le s\left(C_2 \sum_{k=0}^{\infty} \epsilon_k^2 + C_1 M_0
\Big(\sum_{q=0}^{\infty}\epsilon_q^2\Big) 
\Big(\sum_{k=0}^{\infty}\epsilon_k\Big)\right)
\end{equation}
and
\begin{equation}
\label{3.34}
T_2 \ge -s\left( (C_2 + C_1^2) \Big(\sum_{k=0}^{\infty} \epsilon_k^2\Big)
+ C_1 M_0 \Big(\sum_{q=0}^{\infty}\epsilon_q^2\Big) 
\Big(\sum_{k=0}^{\infty}\epsilon_k\Big)\right).
\end{equation}
Combining equations \eqref{3.28}, \eqref{3.33}, and \eqref{3.34} and
using \eqref{3.11}, we obtain for $s >0$ and $\w = \B_{\nu}$
\begin{equation}
\label{3.35}
\frac{D^2 (b_{\w}(x)^s)}{ b_{\w}(x)^s}
 \le s^2 C_1^2 \Big(\sum_{k=0}^{\infty} \epsilon_k\Big)^2
+ s \Big( \sum_{k=0}^{\infty} \epsilon_k^2 \Big)
\left[C_2 + C_1 M_0  \Big(\sum_{k=0}^{\infty} \epsilon_k\Big)\right]
\end{equation}
and
\begin{equation}
\label{3.36}
\frac{D^2 (b_{\w}(x)^s)}{ b_{\w}(x)^s}
 \ge -s \Big(\sum_{k=0}^{\infty} \epsilon_k\Big)^2 
\left[(C_2 + C_1^2) + C_1 M_0 \Big(\sum_{k=0}^{\infty} \epsilon_k\Big)\right],
\end{equation}
which completes the proof.
\end{proof}
\begin{thm}
\label{thm:3.4}
Assume that (H6.1) is satisfied with $m \ge 2$, and for $s >0$, let $v_s$
denote the strictly positive, normalized $C^m$ eigenvector of $L_s$ in
\eqref{3.1}.  For integers $\nu \ge 0$, define $\epsilon_{\nu}$ by $\epsilon_0
=1$ and by \eqref{3.5} for $\nu \ge 1$. In addition, let $C_1$, $C_2$, and
$M_0$ be constants given by \eqref{3.6}, \eqref{3.23}, and \eqref{3.14},
respectively. Then for all $x \in \bar H$, we have the following estimates.
\begin{equation*}
\frac{D^2 v_s(x)}{ v_s(x)}
 \le s^2 C_1^2 \Big(\sum_{k=0}^{\infty} \epsilon_k\Big)^2
+ s \Big( \sum_{k=0}^{\infty} \epsilon_k^2 \Big)
\left[C_2 + C_1 M_0  \Big(\sum_{k=0}^{\infty} \epsilon_k\Big)\right]
\end{equation*}
and
\begin{equation*}
\frac{D^2 v_s(x)}{ v_s(x)}
 \ge -s \sum_{k=0}^{\infty} \epsilon_k^2 
\left[(C_2 + C_1^2) + C_1 M_0 \Big(\sum_{k=0}^{\infty} \epsilon_k\Big)\right].
\end{equation*}
\end{thm}
\begin{proof}
Theorem~\ref{thm:3.4} follows immediately from \eqref{1.14} and \eqref{3.35}
and \eqref{3.36}.
\end{proof}

The estimates given in Theorems~\ref{thm:3.1} and \ref{thm:3.4} are somewhat
crude.  If one has more information about the coefficients $b_{\beta}(\cdot)$
and the maps $\theta_{\beta}(\cdot)$, $\beta \in \B$, one can obtain
much sharper results.  An example is provided by the following theorem.

\begin{thm}
\label{thm:3.5}
Assume that (H6.1) is satisfied with $m \ge 2$.
Assume, also, that  $H=(a_1,a_2)$ is a bounded open interval
in $\R$ and that
$\theta_{\beta}^{\prime}(u) \ge 0$, $\theta_{\beta}^{\prime\prime}(u) \ge 0$,
$b_{\beta}^{\prime}(u) \ge 0$, $b_{\beta}^{\prime\prime}(u) \ge 0$, and
\begin{equation}
\label{3.37}
b_{\beta}^{\prime\prime}(u) b_{\beta}(u) - (1-s)[b_{\beta}^{\prime}(u)]^2
\ge 0
\end{equation}
for all $\beta \in \B$, for all $u \in H$, and for a given real number $s$. If
$ s>0$ and $v_s$ is the strictly positive $C^m$ eigenvector of $\Lambda_s$, it
follows that for all $u \in \bar H$
\begin{equation}
\label{3.38}
D v_s(u) \ge 0 \text{ and } D^2 v_s(u) \ge 0.
\end{equation}
If, in addition, there exists a set $F \subset \bar H$ (possibly empty) such
that for all $u \in \bar H \setminus F$ and all $\beta \in \B$, 
$b_{\beta}^{\prime}(u) > 0$ and strict inequality holds in \eqref{3.37}, then
for all $u \in \bar H \setminus F$,
\begin{equation}
\label{3.39}
D v_s(u) >  0 \text{ and } D^2 v_s(u) >  0.
\end{equation}
\end{thm}
\begin{proof}
For $\nu \ge 1$, let $\w =(j_1, j_2, \cdots,, j_{\nu})$ denote a fixed
element of $\B_{\nu}$ and for $0 \le k \le \nu$, let $\xi_k(x)$ be as
defined in the proof of Theorem~\ref{thm:3.1}.  We leave to the reader
the simple proof that $\xi_k^{\prime}(x) \ge 0$ and $\xi_k^{\prime\prime}(x)
\ge 0$
for all $x \in \bar H$ and $0 \le k \le \nu$.  Using \eqref{3.7}
and \eqref{3.10}, it follows that
\begin{equation}
\label{3.40}
\frac{D (b_{\w}(x)^s)}{ b_{\w}(x)^s}
 = s \frac{D b_{\w}(x)}{b_{\w}(x)}
= s \sum_{k=0}^{\nu-1}
\frac{b_{j_{k+1}}^{\prime}(\xi_k(x)) \xi_k^{\prime}(x)}
{b_{j_{k+1}}(\xi_k(x))} \ge s \frac{b_{j-1}^{\prime}(x)}{b_{j-1}(x)} \ge 0.
\end{equation}
Using \eqref{1.14} and taking the limit as $\nu \rightarrow \infty$, we
conclude that $D v_s(x)/v_s(x) \ge 0$ for all $x \in \bar H$. If, in addition,
there exists a set $F$ as in the statement of Theorem~\ref{thm:3.5} and if
$x \notin F$, it follows that
\begin{equation*}
\inf\Big\{s \frac{b_{\beta}^{\prime}(x)}{b_{\beta}(x)} : \beta \in \B \Big\} := s
\delta_1(x) >0,
\end{equation*}
so \eqref{3.40} then implies that
\begin{equation*}
\frac{D( b_{\w}(x)^s)}{b_{\w}(x)^s} \ge s \delta_1(x).
\end{equation*}
Again using \eqref{1.14} and letting $\nu \rightarrow \infty$, we conclude
that $D v_s(x) \ge s  \delta_1(x) >0$ for all $x \in \bar H \setminus F$.
Because all terms in the summation in \eqref{3.40} are nonnegative,
we conclude that
\begin{equation}
\label{3.41}
s^2 \left(\frac{D b_{\w}(x)}{ b_{\w}(x)}\right)^2
\ge s^2 \sum_{k=0}^{\nu-1}
\frac{[b_{j_{k+1}}^{\prime}(\xi_k(x))]^2 [\xi_k^{\prime}(x)]^2}
{[b_{j_{k+1}}(\xi_k(x))]^2}.
\end{equation}
If one replaces $s^2[D b_{\w}(x)/b_{\w}(x)]^2$ in \eqref{3.28} by the lower
bound in \eqref{3.41} and if one then uses \eqref{3.29} and simplifies,
one obtains
\begin{multline}
\label{3.42}
\frac{D^2(b_{\w}(x)^s)}{b_{\w}(x)^s}
\ge \sum_{k=0}^{\nu-1}
\frac{[s b_{j_{k+1}}^{\prime\prime}(\xi_k(x)) b_{j_{k+1}}(\xi_k(x))
-(s-s^2) b_{j_{k+1}}^{\prime}(\xi_k(x))^2] [\xi_k^{\prime}(x)]^2}
{[b_{j_{k+1}}(\xi_k(x))]^2}
\\
+ s \sum_{k=0}^{\nu-1} \frac{b_{j_{k+1}}^{\prime}(\xi_k(x)) \xi_k^{\prime\prime}(x)}
{b_{j_{k+1}}(\xi_k(x))}.
\end{multline}
If \eqref{3.37} is satisfied, one deduces from \eqref{3.42} that
$D^2(b_{\w}(x)^s)/b_{\w}(x)^s \ge 0$; and again using \eqref{1.14} and letting
$\nu \rightarrow \infty$, one obtains that $D^2 v_s(x) \ge 0$ for all
$x \in \bar H$. If a set $F$ exists and if $x \notin F$ and one only takes
the term $k=0$ in the summation in \eqref{3.42}, then because we assume
that strict inequality holds in \eqref{3.37} for all $\beta \in \B$ and
all $x \notin F$, we find that there is a number $\delta_2(x;s)>0$ such that
\begin{equation*}
\frac{D^2(b_{\w}(x)^s)}{b_{\w}(x)^s} \ge \delta_2(x;s).
\end{equation*}
Again, using \eqref{1.14} and letting $\nu \rightarrow \infty$, this implies
that for $x \notin F$,
\begin{equation*}
\frac{D^2 v_s(x)}{v_s(x)} \ge \delta_2(x;s) >0,
\end{equation*}
which completes the proof.
\end{proof}

Example 3.6: To illustrate the methods of this section, we consider a simple
example which nevertheless has some interest because of a failure of
smoothness which makes techniques in \cite{Jenkinson-Pollicott} inapplicable.
For $0 \le \lambda \le 1$, define
\begin{equation*}
\theta_1(x) = \frac{1}{3 + 2 \lambda}(x + \lambda x^{7/2}), \qquad
\theta_2(x) = \theta_1(x) + \frac{2 + \lambda}{3 + 2 \lambda},
\end{equation*}
so $\theta_j:[0,1] \to [0,1]$, $\theta_1(0) =0$, and $\theta_2(1) = 1$.  For
simplicity we suppress the dependence of $\theta_j(x)$ on $\lambda$ in our
notation. If $\lambda =0$, one obtains the iterated function system which
gives the {\it middle thirds} Cantor set.  If $\B=\{1,2\}$ and $\lambda >0$
and $\w = (j_1, j_2, \ldots, j_{\nu}) \in \B_{\nu}$, notice that $D^3
\theta_{\w}(x)$ is defined and H\"older continuous for all $x \in [0,1]$; but
if $j_1 =1$, $D^4 \theta_{\w}(x)$ is not defined.  If $0 \le \lambda \le 1$,
one can check that $0 < \theta_j^{\prime}(x) <1$ for $0 \le x \le 1$; and it
follows that there exists a unique compact,
nonempty set $J_{\lambda} \subset [0,1]$ such that
\begin{equation*}
J_{\lambda} = \theta_1(J_{\lambda}) \cup \theta_2(J_{\lambda}).
\end{equation*}
Note that $J_0$ is the {\it middle thirds} Cantor set.

For $\lambda \in [0,1]$ fixed, and $0 < s \le 1$, let
$X = C^2[0,1]$ and $Y = C[0,1]$, and define 
\begin{equation*}
b_1(x):= b_2(x):= b(x) := D \theta_1(x)= \frac{1}{3+2 \lambda}(1 + \tfrac{7}{2}
\lambda x^{5/2}).
\end{equation*}
As in Section~\ref{sec:intro}, define $\Lambda_s:X \to X$ and
$L_s:Y \to Y$ by the same formula:
\begin{equation}
\label{3.43} 
(\Lambda_s(f))(x) = b(x)^s[f(\theta_1(x)) + f(\theta_2(x))].
\end{equation}
Theorem~\ref{thm:1.1} implies that $r(L_s) = r(\Lambda_s)$; and it follows, for
example, from theorems in \cite{N-P-L} that the Hausdorff dimension of
$J_{\lambda}$ is the unique value of $s$, $0 < s \le 1$, for which
$r(\Lambda_s) =1$.

If $f \in Y$ is a nonnegative function, we have for $0 \le \lambda \le 1$ that
\begin{equation*}
(L_s(f))(x) \ge \Big(\frac{1}{3+2 \lambda}\Big)^s
[f(\theta_1(x)) + f(\theta_2(x))] \ge \Big(\frac{1}{5}\Big)^s
[f(\theta_1(x)) + f(\theta_2(x))] .
\end{equation*}
If $u(x) =1$ for $0 \le x \le 1$, it follows that
\begin{equation*}
L_s(u) \ge \Big(\frac{1}{5}\Big)^s (2u),
\end{equation*}
which implies that $r(L_s) \ge 2 (1/5)^s$.
If $\log$ denotes the natural logarithm and 
$0 \le s < \log(2)/\log(5)$, it follows that $r(L_s) >1$.  Thus we
restrict attention to $0 \le \lambda \le 1$ and
$s \ge  \log(2)/\log(5) \approx .4307$.  A calculation gives, for
$0 < x \le 1$ that
\begin{multline*}
b^{\prime\prime}(x) b(x) - (1-s)[b^{\prime}(x)]^2 = (\tfrac{7}{2})(\tfrac{5}{2})
\lambda\{(\tfrac{3}{2}) x^{1/2} - (\tfrac{7}{2}) 
\lambda[(\tfrac{3}{2}) - (\tfrac{5}{2})(1-s)]x^3\}
\\
\ge (\tfrac{7}{2}) \lambda\{(\tfrac{3}{2}) 
- (\tfrac{7}{2})\lambda[(\tfrac{3}{2}) - (\tfrac{5}{2})(1-s)]x^3\} >0.
\end{multline*}
It follows from Theorem~\ref{thm:3.5} that 
$D v_s(u) >0$ and $D^2 v_s(u) >0$
for $0 <u \le 1$.  If $\lambda=0$, $v_s$ is a constant and
$r(L_s) = (2/3^s)$, and one obtains the well-known result that the Hausdorff
dimension of the Cantor set is $\log(2)/\log(3)$.

It remains to apply Theorems~\ref{thm:3.1} and \ref{thm:3.4} in our example.
Because $D \theta_j(x) = b(x)$ and
$0 < b(x) \le \kappa(\lambda) = (2+7 \lambda)/(6 +  \lambda)$, we can
define $\epsilon_{\nu}(\lambda):= \epsilon_{\nu} = \kappa(\lambda)^{\nu}$,
where $\epsilon_{\nu}$ is defined as in \eqref{3.5}.  Because
$b_1(x) = b_2(x) = b(x) = (3+2 \lambda)^{-1}( 1 + (\tfrac{7}{2})
\lambda x^{5/2})$,
to compute $C_1(\lambda) = C_1$ as in \eqref{3.6}, we need to compute
\begin{align*}
C_1(\lambda):&= C_1 = \sup\{[(\tfrac{7}{2}) \lambda (\tfrac{5}{2}) x^{3/2})]
[1 + (\tfrac{7}{2}) \lambda x^{5/2})]^{-1}: 0 \le x \le 1\}
\\
&= (\tfrac{5}{2}) \sup\{[(\tfrac{7}{2}) \lambda u^3]
[1 + (\tfrac{7}{2}) \lambda u^5]^{-1}: 0 \le u \le 1\}.
\end{align*}
An elementary but tedious calculus argument, which we leave to the reader,
yields
\begin{equation}
\label{3.44}
C_1(\lambda) = \begin{cases}
[(\frac{5}{2})(\frac{7}{2}) \lambda][1 + (\frac{7}{2}) \lambda]^{-1}, 
& 0 < \lambda \le \tfrac{3}{7}
\\
(\frac{7 \lambda}{2})(\frac{3}{7 \lambda})^{3/5}, & \tfrac{3}{7} 
\le \lambda \le 1.
\end{cases}
\end{equation}
It also follows from Theorems~\ref{thm:3.1}
and \ref{thm:3.5} that for $0 < x \le 1$,
\begin{multline}
\label{3.45}
0 < \frac{D v_s(x)}{v_s(x)} \le s C_1(\lambda) \Big(\sum_{\nu =0}^{\infty}
\epsilon_{\nu} \Big)
\\
= s C_1(\lambda) [1- \kappa(\lambda)]^{-1} = s C_1(\lambda) (6 + 4 \lambda)
/(4 - 3 \lambda).
\end{multline}

If $C_2 = C_2(\lambda)$ is as in \eqref{3.23}, we have to compute
\begin{align*}
C_2:= C_2(\lambda) &= \sup\{[(\tfrac{7}{2}) \lambda (\tfrac{5}{2})
(\tfrac{3}{2}) x^{1/2})]
[1 + (\tfrac{7}{2}) \lambda x^{5/2})]^{-1}: 0 \le x \le 1\}
\\
&= (\tfrac{5}{2}) (\tfrac{3}{2}) \sup\{[(\tfrac{7}{2}) \lambda u]
[1 + (\tfrac{7}{2}) \lambda u^5]^{-1}: 0 \le u \le 1\}.
\end{align*}
A simple calculus exercise yields
\begin{equation}
\label{3.46}
C_2(\lambda) = \begin{cases}
[(\frac{15}{4})(\frac{7}{2}) \lambda][1 + (\frac{7}{2}) \lambda]^{-1}, 
& 0 < \lambda \le \tfrac{1}{14}
\\
3 (\frac{1}{4})^{1/5} [(\frac{7}{2}) \lambda]^{4/5}, & \tfrac{1}{14}
\le \lambda \le 1.
\end{cases}
\end{equation}
If we recall the definition of $M_0$, we also obtain from Example 3.6 that
\begin{equation}
\label{3.47}
M_0 = M_0(\lambda) = \sup\Big\{\frac{1}{(3 + 2 \lambda)}\frac{7}{2} \lambda
\frac{5}{2} x^{3/2} : 0 \le x \le 1 \Big\} = \frac{(35 \lambda)}{4}
\frac{1}{(3 + 2 \lambda)}.
\end{equation}

If we now refer to Theorem~\ref{thm:3.4}, we find for $0 < x \le 1$, 
$.4307 \le s$ , and $0 < \lambda \le 1$, that
\begin{multline}
\label{3.48}
0 < \frac{D^2 v_s(x)}{v_s(x)} \le s^2 [C_1(\lambda)]^2
\Big(\frac{6+4 \lambda}{4 - 3 \lambda}\Big)
\\
+ s \frac{(6+4 \lambda)^2}{(4 - 3 \lambda)(8 + 11 \lambda)}
\left[C_2(\lambda) + C_1(\lambda) M_0(\lambda) 
\frac{(6+4 \lambda)}{(4 - 3 \lambda)}\right].
\end{multline}
As was shown in Section~\ref{sec:1dexps} (see Theorem~\ref{thm:posev1dreg} and
Table~\ref{tb:t3}), with the aid of \eqref{3.48}, we can obtain rigorous, high
accuracy estimates (upper and lower bounds) for the Hausdorff dimension of
$J_{\lambda}$ for $0 < \lambda \le 1$.


\section{The Case of M\"obius Transformations}
\label{sec:mobius}
By working with partial derivatives and using methods like those in
Section~\ref{sec:1d-deriv}, it is possible to obtain explicit
estimates on partial derivatives of $v_s(x)$ in the generality of
Theorem~\ref{thm:1.1}.  However, for reasons of length and in view of
the immediate applications in this paper, we shall not treat the
general case here and shall now specialize to the case that the
mappings $\theta_{\beta}(\cdot)$ are given by M\"obius transformations
which map a given bounded open subset $H$ of $\C:=\R^2$ into
$H$. Specifically, throughout this section we shall usually assume:

\noindent (H7.1): $\gamma \ge 1$ is a given real number and $\B$ is a finite
collection of complex numbers $\beta$ such that $\Re(\beta) \ge \gamma$ for
all $\beta \in \B$. For each $\beta \in \B$, $\theta_{\beta}(z):= 1/(z +
\beta)$ for $z \in \C \setminus \{- \beta\}$.

As we note in Remark~\ref{rem:4.13}, the assumption in (H7.1) that $\gamma
\ge 1$ is only a convenience; and the results of this section can be
proved under the weaker assumption that $\gamma >0$.

For $\gamma >0$ we define $G_{\gamma} \in \C$ by
\begin{equation} 
\label{4.1}
G_{\gamma} = \{ z \in \C: |z - 1/(2 \gamma)|
< 1/(2 \gamma) \}.
\end{equation}
It is easy to check that if $w \in \C$ and $\Re(w) > \gamma$, then
$(1/w) \in G_{\gamma}$.  It follows that if $\Re(z) >0$, $\beta \in
\C$ and $\Re(\beta) \ge \gamma >0$, then $\theta_{\beta}(z) \in \bar
G_{\gamma}$. Let $H$ be a bounded, open, mildly regular subset of $\C
= \R^2$ such that $H \supset G_{\gamma}$ and $H \subset \{z \, | \, \Re(z)
>0\}$, and let $\B$ denote a finite set of complex numbers such that
$\Re(\beta) \ge \gamma >0$ for all $\beta \in \B$. We define a bounded
linear map $\Lambda_s:C^m(\bar H) \to C^m(\bar H)$, where $m$ is a
positive integer and $s \ge 0$, by
\begin{equation} 
\label{4.2}
(\Lambda_s(f))(z) = \sum_{\beta \in\B} \Big|\frac{d}{dz}
\theta_{\beta}(z)\Big|^s f(\theta_{\beta}(z)):= \sum_{\beta \in\B} 
\frac{1}{|z+ \beta|^{2s}} f(\theta_{\beta}(z)).
\end{equation}
As in Section~\ref{sec:intro}, $L_s: C(\bar H) \to C(\bar H)$ is defined
by \eqref{4.2}.  We use different letters to emphasize that $\sigma(\Lambda_s)
\neq \sigma(L_s)$, although $r(\Lambda_s) = r(L_s)$.

If all elements of $\B$ are real, we can restrict attention to the real line
and, as we shall see, the analysis is much simpler.  In this case we abuse
notation and take $G_{\gamma} = (0, 1/\gamma) \subset \R^2$ and $H= (0,a)$,
$a \ge 1/\gamma$.  For $f \in C^m(\bar H)$ and $x \in \bar H$, \eqref{4.2} takes
the form
\begin{equation}
\label{4.3} 
(\Lambda_s(f))(x) = \sum_{\beta \in\B} \frac{1}{(x+ \beta)^{2s}}
f(\theta_{\beta}(x)).
\end{equation}

If, for $1 \le j \le n$, $M_j = \bigl( \begin{smallmatrix} a_j & b_j \\ c_j &
  d_j \end{smallmatrix} \bigr)$ is a $2 \times 2$ matrix with complex entries
and $\det(M_j) = a_j d_j - b_j c_j$, define a M\"obius transformation
$\psi_j(z) = (a_j z + b_j)/(c_j z + d_j)$.  It is well-known that
\begin{equation}
\label{4.4}
(\psi_1 \circ \psi_2 \circ \cdots \circ \psi_n)(z) = (A_n z + B_n)/(C_n z +
D_n),
\end{equation}
where
\begin{equation}
\label{4.5}
\begin{pmatrix} A_n & B_n \\ C_n &   D_n \end{pmatrix}
= M_1 M_2 \cdots  M_n.
\end{equation}

If $\B$ is a finite set of complex numbers $\beta$ such that
$\Re(\beta) \ge \gamma >0$ for all $\beta \in \B$, we define $\B_{\nu}$
as before by
\begin{equation*}
\B_{\nu} = \{ \w = (\beta_1, \beta_2, \ldots, \beta_{\nu}) \, | \, \beta_j \in
\B \text{ for } 1 \le j \le \nu\}
\end{equation*}
and $\theta_{\w} = \theta_{\beta_n} \circ \theta_{\beta_{n-1}} \cdots
\theta_{\beta_1}$. Given $\w = (\beta_1, \beta_2, \ldots, \beta_{\nu}) 
\in \B_{\nu}$, we define
\begin{equation}
\label{4.6} 
\tilde \w = (\beta_{\nu}, \beta_{\nu-1}, \ldots, \beta_{1})
\end{equation}
so
\begin{equation}
\label{4.7}
\theta_{\tilde \w} = \theta_{\beta_1} \circ \theta_{\beta_{2}} \cdots
\theta_{\beta_n}.
\end{equation}
For $\Lambda_s$ as in \eqref{4.2} $\nu \ge 1$, and $f \in C^m(\bar H)$,
recall that
\begin{equation}
\label{4.8}
(\Lambda_s^{\nu}(f))(z) = \sum_{\w \in \B_{\nu}}
\Big|\frac{d \theta_{\w}(z)}{dz}\Big|^s f(\theta_{\w}(z))
= \sum_{\w \in \B_{\nu}} \Big|\frac{d \theta_{\tilde \w}(z)}{dz}\Big|^s 
f(\theta_{\tilde \w}(z)).
\end{equation}

The following lemma allows us to apply Theorem~\ref{thm:1.1} to $\Lambda_s$ in
\eqref{4.2}.

\begin{lem}
\label{lem:4.1}
(Compare Remark~\ref{4.13}.)
Let $\beta_1$ and $\beta_2$ be complex numbers with $\Re(\beta_j) \ge \gamma
\ge 1$ for $j =1,2$.  If $\psi_j(z) = 1/(z + \beta_j)$ for
$\Re(z) \ge 0$ and $\theta = \psi_1 \circ \psi_2$, then for all $z, w$
with $\Re(z) \ge 0$ and $\Re(w) \ge 0$,
\begin{equation}
\label{4.9}
|\theta(z) - \theta(w)| \le (\gamma^2 +1)^{-2} |z-w|.
\end{equation}
\end{lem}
\begin{proof}
It suffices to prove that $|(d \theta/dz)(z)| \le (\gamma^2 +1)^{-2}$
for all $z \in \C$ with $\Re(z) \ge 0$.  Using \eqref{4.4} and \eqref{4.5}
we see that
\begin{equation*}
|(d \theta/dz)(z)| = |\beta_1|^{-2} |z + (1/\beta_1) + \beta_2|^{-2},
\end{equation*}
so it suffices to prove that
$|\beta_1|^2 \, |z + (1/\beta_1) + \beta_2|^{2} \ge (\gamma^2 +1)^{2}$
for $\Re(z) \ge 0$. If we write $\beta_1 = u+ iv$ with $u \ge \gamma$,
\begin{equation*}
\Re(z + (1/\beta_1) + \beta_2) \ge u/(u^2 + v^2)  + \gamma,
\end{equation*}
so
\begin{equation*}
|z + (1/\beta_1) + \beta_2|^{2} \ge [u/(u^2 + v^2)  + \gamma]^2
\end{equation*}
and
\begin{multline*}
|\beta_1|^2 \, |z + (1/\beta_1) + \beta_2|^{2} \ge (u^2 + v^2)
\Big[\frac{u^2}{(u^2 + v^2)^2} + \frac{2 u \gamma}{(u^2 + v^2)} + \gamma^2
\Big]
\\
= \frac{u^2}{(u^2 + v^2)} + 2 u \gamma + \gamma^2(u^2 + v^2) = g(u,v).
\end{multline*}
Because $u \ge \gamma$, $g(u,0) = 1 + 2 \gamma^2 + \gamma^4
= (\gamma^2 +1)^2$.  Using the fact that $u \ge \gamma \ge 1$, we also see
that for $v \ge 0$
\begin{equation*}
\frac{ \partial g(u,v)}{\partial v} = \frac{-u^2(2v)}{(u^2 + v^2)^2}
+ 2 \gamma^2 v \ge 0,
\end{equation*}
which implies that $g(u,v) \ge g(u,0) = (\gamma^2 +1)^{2}$ for $u \ge \gamma$
and $v \ge 0$.  Since $g(u,-v) = g(u,v)$,
$g(u,v) \ge (\gamma^2 +1)^{2}$ for $v \le 0$ and $u \ge \gamma$.
\end{proof}

With the aid of Lemma~\ref{lem:4.1}, the following theorem is an immediate
corollary of Theorem~\ref{thm:1.1}.

\begin{thm}
\label{thm:4.2}
Assume (H7.1) and let $H$ be a bounded, open mildly regular subset of
$\{z \in \C \, | \, \Re(z) >0\}$ such that $H \supset G_{\gamma}$, where
$G_{\gamma}$ is defined by \eqref{4.1}.  For a
given positive integer $m$ and for $s > 0$, let
$X = C^m(\bar H)$ and $Y = C(\bar H)$ and let $
\Lambda_s: X \to X$ and $L_s:Y \to Y$ be given by \eqref{4.2}.
If $r(\Lambda_s)$ (respectively, $r(L_s)$) denote the spectral radius
of $\Lambda_s$ (respectively, $L_s$), we have $r(\Lambda_s) > 0$ and
$r(\Lambda_s) = r(L_s)$. If $\rho(\Lambda_s)$ denotes the essential
spectral radius of $\Lambda_s$,
\begin{equation}
\label{4.10}
\rho(\Lambda_s) \le (\gamma^2 +1)^{-m} r(\Lambda_s).
\end{equation}
For each $s >0$, there exists $v_s \in X$ such that $v_s(z) >0$ for all
$z \in \bar H$ and $\Lambda_s(v_s) = r(\Lambda_s)v_s$.  All the statements
of Theorem~\ref{thm:1.1} are true in this context whenever $\Lambda$ and $L$
in Theorem~\ref{thm:1.1} are replaced by $\Lambda_s$ and $L_s$ respectively.
\end{thm}

In the notation of Theorem~\ref{thm:4.2}, it follows from \eqref{1.14}
that for any multi-index $\alpha=(\alpha_1,\alpha_2)$ and for $z = x+iy =
(x,y)$
\begin{equation}
\label{4.11}
\lim_{\nu \rightarrow \infty} \frac{D^{\alpha}\left(\sum_{\w \in \B_{\nu}}
\Big|\frac{d}{dz} \theta_{\w}(z)\Big|^s\right)}
{\sum_{\w \in \B_{\nu}} \Big|\frac{d}{dz} \theta_{\w}(z)\Big|^s}
= \frac{D^{\alpha} v_s(x,y)}{v_s(x,y)},
\end{equation}
where the convergence is uniform in $(x,y):=z \in \bar H$
and $D^{\alpha} = (\partial/\partial x)^{\alpha_1}
(\partial/\partial y)^{\alpha_2}$.

\begin{lem}
\label{lem:4.3}
Let $\beta_j$, $j \ge 1$ be a sequence of complex numbers with
$\Re(\beta_j) \ge \gamma \ge 0$ for all $j$.  For complex numbers $z$,
define $\theta_{\beta_j}(z) = (z + \beta_j)^{-1}$ and define matrices
$M_j =
\bigl(\begin{smallmatrix} 0 & 1 \\
  1 & \beta_j \end{smallmatrix}\bigr)$.  Then for $n \ge 1$,
\begin{equation}
\label{4.12}
M_1 M_2 \cdots M_n = \begin{pmatrix} A_{n-1} & A_n \\ B_{n-1} & B_n
\end{pmatrix},
\end{equation}
where $A_0 =0 $, $A_1 =1$, $B_0 =1$, $B_1 = \beta_1$ and for $n \ge 1$,
\begin{equation}
\label{4.13}
A_{n+1} = A_{n-1} + \beta_{n+1} A_n \text{ and }
B_{n+1} = B_{n-1} + \beta_{n+1} B_n.
\end{equation}
Also,
\begin{equation*}
(\theta_{\beta_1} \circ \theta_{\beta_{2}} \cdots \theta_{\beta_n})(z)
= (A_{n-1} z + A_n)/(B_{n-1} z + B_n),
\end{equation*}
and we have
\begin{equation}
\label{4.14}
\Re(B_n/B_{n-1}) \ge \gamma
\end{equation}
and
\begin{equation}
\label{4.15}
\Big|\frac{d}{dz}\Big[\frac{A_{n-1} z + A_n}{B_{n-1} z + B_n}\Big]\Big|^s
= |B_{n-1}|^{-2s} |z + B_n/B_{n-1}|^{-2s}.
\end{equation}
\end{lem}
\begin{proof}
Equation \eqref{4.12} follows by induction on $n$. It is obviously true
for $n=1$. If we assume that \eqref{4.12} is satisfied for some $n \ge 1$,
then
\begin{equation*}
M_1 M_2 \cdots M_n M_{n+1}= \begin{pmatrix} A_{n-1} & A_n \\ B_{n-1} & B_n
\end{pmatrix} \begin{pmatrix} 0 & 1 \\ 1 & \beta_{n+1} \end{pmatrix}
= \begin{pmatrix} A_n & A_{n-1} + \beta_{n+1} A_n \\
B_n & B_{n-1} + \beta_{n+1} B_n \end{pmatrix},
\end{equation*}
which proves \eqref{4.12} with $A_{n+1}$ and $B_{n+1}$ defined by
\eqref{4.13}.
Similarly, we prove \eqref{4.14} by induction on $n$.  The case $n=1$ is
obvious,  Assuming that \eqref{4.13} is satisfied for some $n \ge 1$,
we obtain from \eqref{4.13} that
\begin{equation*}
B_{n+1}/B_n = B_{n-1}/B_n + \beta_{n+1}.
\end{equation*}
Because $\Re(w) \ge \gamma$, where $w:= B_n/B_{n-1}$, we see that
$|1/w - 1/(2\gamma)| \le 1/(2 \gamma)$ and
 $\Re(1/w) = \Re(B_{n-1}/B_n) \ge 0$, so
\begin{equation*}
\Re(B_{n+1}/B_n) \ge \Re(B_{n-1}/B_n)  + \Re(\beta_{n+1}) \ge \gamma.
\end{equation*}
Hence \eqref{4.13} is satisfied for all $n \ge 1$.
Because $\det(M_j) = -1$ for all $j \ge 1$, we get that
$\det \bigl(\begin{smallmatrix} A_{n-1} & A_n \\ B_{n-1} & B_n
\end{smallmatrix}\bigr) = (-1)^n$, and \eqref{4.15} follows.
\end{proof}

Before proceeding further, it will be convenient to establish some
elementary calculus propositions.  For $(u,v) \in \R^2 \setminus\{(0,0)\}$
and $s >0$, define
\begin{equation*}
G(u,v;s) = (u^2 + v^2)^{-s}.
\end{equation*}
Define $D_1 = (\partial/\partial u)$, so $D_1^m = (\partial/\partial u)^m$
for positive integers $m$; similarly, let 
$D_2 = (\partial/\partial v)$ and $D_2^m = (\partial/\partial v)^m$.
\begin{lem}
\label{lem:4.4}
For positive integers $m$, there exist polynomials in $u$ and $v$ with
coefficients depending on $s$, $P_m(u,v;s)$ and $Q_m(u,v;s)$, such that
\begin{equation*}
D_1^m G(u,v;s) = P_m(u,v;s) G(u,v;s+m), \
D_2^m G(u,v;s) = Q_m(u,v;s) G(u,v;s+m).
\end{equation*}
Furthermore, we have $P_1(u,v;s) = -2 s u$, $Q_1(u,v;s) = -2 s v$,
and for positive integers $m$,
\begin{equation*}
P_{m+1}(u,v;s) = (u^2 + v^2) (D_1 P_m(u,v;s)) - 2(s+m) u P_m(u,v;s)
\end{equation*}
and
\begin{equation*}
Q_{m+1}(u,v;s) = (u^2 + v^2) (D_2 Q_m(u,v;s)) - 2(s+m) v Q_m(u,v;s).
\end{equation*}
\end{lem}
\begin{proof}
If $m =1$,
\begin{equation*}
D_1 G(u,v;s) = (-2 s u) \, G(u,v;s+1), \qquad 
D_2 G(u,v;s) = (-2 s v) (u^2 + v^2; s+1),
\end{equation*}
so $P_1(u,v;s) = -2 su$ and $Q_1(u,v;s) = -2 sv$. 

We now argue by induction and assume we have proved the existence
of $P_j(u,v;s)$ and $Q_j(u,v;s)$ for $1 \le j \le m$.  It follows that
\begin{multline*}
D_1^{m+1} G(u,v;s) = D_1[P_m(u,v;s) G(u,v;s+m)]
\\
 = [D_1 P_m(u,v;s)] G(u,v;s+m)] + P_m(u,v;s)[-2(s+m)u]G(u,v;s+m+1)
\\
 = [(u^2 + v^2) (D_1 P_m(u,v;s)) -2(s+m) u P_m(u,v;s)] G(u,v;s+m+1).
\end{multline*}
This proves the lemma with
\begin{equation*}
P_{m+1}(u,v;s):= (u^2 + v^2) (D_1 P_m(u,v;s)) -2(s+m) u P_m(u,v;s).
\end{equation*}
An exactly analogous argument, which we leave to the reader, shows that
\begin{equation*}
Q_{m+1}(u,v;s):= (u^2 + v^2) (D_2 Q_m(u,v;s)) -2(s+m) v Q_m(u,v;s).
\end{equation*}
\end{proof}
An advantage of working with M\"obius transformations is that one can
easily obtain tractable formulas for expressions like
$(\theta_{\beta_1} \circ \theta_{\beta_{2}} \cdots \theta_{\beta_n})(z)$. Such
formulas allow more precise estimates for the left hand side of \eqref{1.14}
than we obtained in Section~\ref{sec:1d-deriv}.

\begin{lem}
\label{lem:4.5}
In the notation of Lemma~\ref{lem:4.4}, for all $(u,v) \in \R^2 \setminus 
\{(0,0)\}$, for all $s >0$, and all positive integers $m$, 
$P_m(u,v;s) = Q_m(v,u;s)$.
\end{lem}
\begin{proof}
Fix $s >0$.  We have $P_1(u,v;s) = Q_1(v,u;s)$ for all $(u,v) \neq (0,0)$.
Arguing by mathematical induction, assume that for some positive integer $m$
we have proved that $P_m(u,v;s) = Q_m(v,u;s)$ for all $(u,v) \neq (0,0)$.
For a fixed $(u,v) \neq (0,0)$, we obtain, by virtue of the recursion
formula in Lemma~\ref{lem:4.4},
\begin{align*}
P_{m+1}(v,u;s) &= (u^2 + v^2) \lim_{\Delta v \rightarrow 0}
\frac{P_m(v+ \Delta v, u; s) - P_m(v,u;s)}{\Delta v}
\\
&\qquad -2(s+m)v P_m(v,u;s)
\\
&= (u^2 + v^2) \lim_{\Delta v \rightarrow 0}
\frac{Q_m(u,v+ \Delta v; s) - Q_m(u,v;s)}{\Delta v}
\\
&\qquad -2(s+m)v Q_m(u,v;s) 
\\
&= Q_{m+1}(u,v;s).
\end{align*}
By mathematical induction, we conclude that 
$P_n(u,v;s) = Q_n(v,u;s)$ for all positive integers $n$.
\end{proof}

\begin{remark}
\label{rem:4.6}
By using the recursion formula in Lemma~\ref{lem:4.4}, one can easily
compute $P_j(u,v;s)$ for $1\le j\le 4$.
\begin{align*}
P_1(u,v;s) &= -2s u,
\\
P_2(u,v;s) &= 2s(2s+1)u^2 - 2s v^2,
\\
P_3(u,v;s) &= -2s(2s+1)(2s+2)u^3 + (2s)(2s+2)(3) u v^2,
\\
P_4(u,v;s) &= (2s)(2s+2)[(2s+1)(2s+3)u^4 - 6(2s+3)u^2 v^2 + 3 v^4].
\end{align*}
\end{remark}
By virtue of Lemma~\ref{lem:4.5}, we also obtain formulas for 
$Q_j(v,u;s) = P_j(u,v;s)$. Also, Lemmas~\ref{lem:4.4} and \ref{lem:4.5}
imply that
\begin{equation*}
\frac{D_1^j G(u,v;s)}{G(u,v;s)} = \frac{P_j(u,v;s)}{(u^2+v^2)^j}, \qquad
\frac{D_2^j G(u,v;s)}{G(u,v;s)} = \frac{P_j(v,u;s)}{(u^2+v^2)^j}
\end{equation*}
and the latter formulas will play a useful role in this section. In particular,
for a given constant $\gamma >0$, we shall need good estimates for
\begin{equation*}
\sup \Big\{\frac{D_k^j G(u,v;s)}{G(u,v;s)}: u \ge \gamma, v \in \R\Big\}
\text{ and }
\inf \Big\{\frac{D_k^j G(u,v;s)}{G(u,v;s)}: u \ge \gamma, v \in \R\Big\}
\end{equation*}
where $k=1,2$ and $1 \le j \le 4$.  Although the arguments used to prove these
estimates are elementary, these results will play a crucial role in our
later work.
\begin{lem}
\label{lem:4.7}
Let $\gamma >0$ be a given constant and assume that $u \ge \gamma$ and $v \in
\R$.  Let $D_1 = (\partial/\partial u)$ and $G(u,v;s) = (u^2 + v^2)^{-s}$, where
$s >0$. For $j \ge 1$ we have
\begin{equation*}
\frac{D_1^j G(u,v;s)}{G(u,v;s)}= \frac{P_j(u,v;s)}{(u^2+v^2)^j},
\end{equation*}
where $P_j(u,v;s)$ is as defined in Remark~\ref{rem:4.6}; and the following
estimates are satisfied.
\begin{gather*}
-\frac{2s}{\gamma} \le \frac{D_1 G(u,v;s)}{G(u,v;s)} <0,
\\
-\frac{s}{4\gamma^2(s+1)} \le \frac{D_1^2 G(u,v;s)}{G(u,v;s)} 
\le \frac{2s(2s+1)}{\gamma^2},
\\
-\frac{2s(2s+1)(2s+2)}{\gamma^3} \le \frac{D_1^3 G(u,v;s)}{G(u,v;s)} 
\le \frac{2s(2s+2)}{\gamma^3(s+2)^2},
\\
-\frac{2s(s+1)(2s+2)(3)}{\gamma^4} \le \frac{D_1^4 G(u,v;s)}{G(u,v;s)} 
\le \frac{2s(2s+1)(2s+2)(2s+3)}{\gamma^4}.
\end{gather*}
\end{lem}
\begin{proof}
By Remark~\ref{rem:4.6}, 
\begin{equation*}
\frac{D_1^j G(u,v;s)}{G(u,v;s)}= \frac{P_j(u,v;s)}{(u^2+v^2)^j},
\end{equation*}
and Remark~\ref{rem:4.6} provides formulas for $P_j(u,v;s)$.  It follows that
\begin{equation*}
\frac{D_1^j G(u,v;s)}{G(u,v;s)} = \frac{-2su}{u^2+v^2} < 0.
\end{equation*}
Since
\begin{equation*}
\frac{2su}{u^2+v^2} \le \frac{2su}{u^2} \le \frac{2s}{\gamma},
\end{equation*}
we also see that
\begin{equation*}
\frac{D_1 G(u,v;s)}{G(u,v;s)} \ge -\frac{2s}{\gamma}.
\end{equation*}
Using Remark~\ref{rem:4.6}, we see that
\begin{equation*}
\frac{D_1^2 G(u,v;s)}{G(u,v;s)} = \frac{2s(2s+1)u^2 -2s v^2}{(u^2 + v^2)^2},
\end{equation*}
so
\begin{equation*}
\frac{D_1^2 G(u,v;s)}{G(u,v;s)} \le \frac{2s(2s+1)u^2}{(u^2 + v^2)^2}.
\end{equation*}
Since
\begin{equation*}
\frac{u^2}{(u^2+v^2)^2} \le \frac{u^2}{u^4} \le \frac{1}{\gamma^2},
\end{equation*}
we find that
\begin{equation*}
\frac{D_1^2 G(u,v;s)}{G(u,v;s)} \le \frac{2s(2s+1)}{\gamma^2},
\end{equation*}
If we write $v^2 = \rho u^2$, we see that
\begin{equation*}
\frac{D_1^2 G(u,v;s)}{G(u,v;s)} = \frac{2s(2s+1-\rho)}{u^2 (1+ \rho)^2},
\end{equation*}
and if $0 \le \rho \le 2s+1$, we obtain the upper bound given above and a
lower bound of zero. If $\rho > 2s+1$, we see that
\begin{equation*}
\frac{D_1^2 G(u,v;s)}{G(u,v;s)} \ge \frac{2s}{\gamma^2}
\inf\left\{\frac{2s+1-\rho}{(1+ \rho)^2}: \rho > 2s+1\right\}.
\end{equation*}
It is a simple calculus exercise to show that
\begin{equation*}
\inf\left\{\frac{2s+1-\rho}{(1+ \rho)^2}: \rho > 2s+1\right\}
= - \frac{1}{8(s+1)},
\end{equation*}
achieved at $\rho = 4s+3$; and this gives the lower estimate
$-s/[4\gamma^2(s+1)]$ of the lemma.

Using Remark~\ref{rem:4.6} again, we see that
\begin{equation*}
\frac{D_1^3 G(u,v;s)}{G(u,v;s)} = \frac{2s(2s+2)u[-(2s+1)u^2 + 3v^2]}
{(u^2 + v^2)^3}.
\end{equation*}
It follows that
\begin{multline*}
\frac{D_1^3 G(u,v;s)}{G(u,v;s)} \ge -2s(2s+1)(2s+2) \left[\frac{u}{(u^2 + v^2)}
\right]^3
\\
\ge -2s(2s+1)(2s+2) \left[\frac{1}{u}\right]^3
\ge -2s(2s+1)(2s+2) \frac{1}{\gamma^3}.
\end{multline*}
On the other hand, if we write $v^2 = \rho u^2$, then
\begin{multline*}
\frac{D_1^3 G(u,v;s)}{G(u,v;s)} = \frac{2s(2s+2)}{u^3}
\frac{[3 \rho -(2s+1)]}{(1+ \rho)^3}
\\
\le \frac{2s(2s+2)}{\gamma^3}
\sup \left\{\frac{3 \rho -(2s+1)}{(1+ \rho)^3}: \rho \ge 0\right\}.
\end{multline*}
Once again, a straightforward calculus argument shows that
\begin{equation*}
\sup \left\{\frac{3 \rho -(2s+1)}{(1+ \rho)^3}: \rho \ge 0\right\}
= \frac{1}{(s+2)^2}
\end{equation*}
and the supremum is achieved at $\rho = s+1$.  Using this fact, we obtain
the upper estimate of the lemma.

Finally, we obtain from Remark~\ref{rem:4.6} that
\begin{equation*}
\frac{D_1^4 G(u,v;s)}{G(u,v;s)} = \frac{2s(2s+2)[(2s+1)(2s+3)u^4
- 6(2s+3) u^2v^2 + 3 v^4]}{(u^2 + v^2)^4}.
\end{equation*}
Dropping the negative term in the numerator and observing that
$3 \le (2s+1)(2s+3)$ and $u^4 + v^4 \le (u^2 + v^2)^2$, we see that
\begin{align*}
\frac{D_1^4 G(u,v;s)}{G(u,v;s)} &\le \frac{(2s)(2s+1)(2s+2)(2s+3)(u^4+v^4)}
{(u^2 + v^2)^4}
\\
&\le \frac{(2s)(2s+1)(2s+2)(2s+3)}{(u^2 + v^2)^2}
\le \frac{(2s)(2s+1)(2s+2)(2s+3)}{\gamma^4}.
\end{align*}
On the other hand, because $-u^4 - v^4 \le -2u^2 v^2$, we obtain that
\begin{align*}
- \frac{D_1^4 G(u,v;s)}{G(u,v;s)} &\le \frac{(2s)(2s+2)
[-3 u^4 + 6(2s+3)u^2 v^2 - 3 v^4]}{(u^2 + v^2)^4}
\\
&\le \frac{3(2s)(2s+2)[-2 u^2v^2 + (4s+6)u^2 v^2]}{(u^2 + v^2)^4}
\\
&\le \frac{3(2s)(2s+2)[4(s+1)(u^2+ v^2)^2/4]}{(u^2 + v^2)^4}
\\
&\le \frac{3(2s)(2s+2)(s+1)}{(u^2 + v^2)^2} 
\le \frac{3(2s)(2s+2)(s+1)}{\gamma^4},
\end{align*}
which gives the lower estimate of Lemma~\ref{lem:4.7}.
\end{proof}

The following lemma gives analogous estimates for
\begin{equation*}
\frac{D_2^j G(u,v;s)}{G(u,v;s)} = \frac{P_j(v,u;s)}{(u^2 + v^2)^j}.
\end{equation*}
\begin{lem}
\label{lem:4.8}
Let $\gamma >0$ be a given real number, $D_2 = (\partial/\partial v)$ 
and for $s>0$ and $(u,v) \in \R^2 \setminus\{(0,0)\}$, define
$G(u,v;s) = (u^2 + v^2)^{-s}$, If $u \ge \gamma$ and $v \in \R$, we have
the following estimates.
\begin{gather*}
\frac{|D_2 G(u,v;s)|}{G(u,v;s)} \le \frac{s}{\gamma},
\\
-\frac{2s}{\gamma^2} \le \frac{D_2^2 G(u,v;s)}{G(u,v;s)} 
\le \frac{2s(2s+1)}{4\gamma^2},
\\
\frac{|D_2^3 G(u,v;s)|}{G(u,v;s)} 
\le \frac{2s(2s+2)}{\gamma^3} \max\left\{\frac{25\sqrt{5}}{72}
, \frac{2s+1}{8}\right\}
\\
-\frac{2s(s+1)(2s+2)(3)}{\gamma^4} \le \frac{D_2^4 G(u,v;s)}{G(u,v;s)} 
\le \frac{2s(2s+1)(2s+2)(2s+3)}{\gamma^4}.
\end{gather*}
\end{lem}
\begin{proof}
By Remark~\ref{rem:4.6}, $P_1(v,u;s) = -2sv$, so
\begin{equation*}
\frac{|D_2 G(u,v;s)|}{G(u,v;s)} = \frac{2s |v|}{u^2 + v^2}.
\end{equation*}
The map $w \mapsto w/(u^2+w^2)$ has its maximum on $[0, \infty)$ at $w=u$,
so $(2s|v|/(u^2+v^2) \le s/u \le s/\gamma$; and we obtain the first
inequality in Lemma~\ref{lem:4.8}.  Using Remark~\ref{rem:4.6} again, we see
that
\begin{equation*}
\frac{D_2^2 G(u,v;s)}{G(u,v;s)} = \frac{2s[(2s+1)v^2 - u^2]}{(u^2 + v^2)^2}.
\end{equation*}
It follows that
\begin{equation*}
\frac{D_2^2 G(u,v;s)}{G(u,v;s)} = 2s(2s+1)\frac{|v|^2}{(u^2 + v^2)^2}.
\end{equation*}
The map 
$v \mapsto |v|/(u^2+v^2)$ has its maximum at $|v|=u$, so $[|v|/(u^2+v^2)]^2
\le 1/(4u^2)\le 1/(4 \gamma^2)$, and
\begin{equation*}
\frac{D_2^2 G(u,v;s)}{G(u,v;s)} = \frac{2s(2s+1)}{4\gamma^2}.
\end{equation*}
Similarly, one obtains
\begin{equation*}
\frac{D_2^2 G(u,v;s)}{G(u,v;s)} \ge - \frac{2s u^2}{(u^2 + v^2)^2}
\ge - \frac{2s}{u^2} \ge - \frac{2s}{\gamma^2}.
\end{equation*}

With the  aid of Remark~\ref{rem:4.6} again, we see that
\begin{equation*}
\frac{D_2^3 G(u,v;s)}{G(u,v;s)} = 2s(2s+2)v \frac{[-(2s+1)v^2 + 3u^2]}
{(u^2+v^2)^3} := A(u,v).
\end{equation*}
For a fixed $u \ge \gamma$, $v \mapsto A(u,v)$ is an odd function of
$v$, so if
$\alpha(u) = \sup \{A(u,v): v \in \R\}$, $- \alpha(u)
= \inf\{A(u,v): v \in \R\}$. If $v \le 0$,
\begin{multline*}
A(u,v) \le (2s)(2s+1)(2s+2) \left[\frac{|v|}{u^2+v^2}\right]^3
\le (2s)(2s+1)(2s+2) \left[\frac{u}{2u^2}\right]^3
\\
\le \frac{(2s)(2s+1)(2s+2)}{8 \gamma^3}.
\end{multline*}
If $v >0$,
\begin{equation*}
A(u,v) \le (2s)(2s+2)(3u^2) \frac{v}{(u^2+v^2)^3}.
\end{equation*}
A calculation shows that $v \mapsto v/(u^2+v^2)^3$ achieves its maximum
for $v \ge 0$ at $v = u/\sqrt{5}$, so for $v >0$,
\begin{equation*}
A(u,v) \le (2s)(2s+2)(3u^{-3}) [\sqrt{5}(6/5)^3]^{-1}
\le (2s)(2s+2) \gamma^{-3} (25\sqrt{5}/72).
\end{equation*}
Note that $25\sqrt{5}/72 \approx .7764 < 1$.
Using Remark~\ref{rem:4.6} again, we see that
\begin{equation*}
\frac{D_2^4 G(u,v;s)}{G(u,v;s)} = 2s(2s+2) \frac{[(2s+1)(2s+3)v^4
-6(2s+3) u^2 v^2 + 3 u^4]}{(u^2+v^2)^4}.
\end{equation*}
Since $u^4 + v^4 \le (u^2+v^2)^2$, it follows easily that
\begin{equation*}
\frac{D_2^4 G(u,v;s)}{G(u,v;s)} \le 2s(2s+2)(2s+1)(2s+3) \frac{u^4 + v^4}
{(u^2+v^2)^4} \le 2s(2s+2)(2s+1)(2s+3) \gamma^{-4}.
\end{equation*}

Similarly, we see that
\begin{multline*}
(2s+1)(2s+3)v^4- 6(2s+3) u^2v^2+ 3 u^4 \ge
3(u^4+v^4) - 6(2s+3)[(u^2+v^2)/2]^2
\\
\ge 3 (u^2+v^2)^2 - 6 [(u^2+v^2)/2]^2 -6(2s+3) [(u^2+v^2)/2]^2.
\end{multline*}
This implies that
\begin{multline*}
\frac{D_2^4 G(u,v;s)}{G(u,v;s)} \ge 2s(2s+2) \frac{3-3/2-3/2(2s+3)}
{(u^2+v^2)^2}
\\
\ge -(2s)(2s+2)3(s+1)(u^2+v^2)^{-2} 
\ge -(2s)(2s+2)(3s+3)\gamma^{-4},
\end{multline*}
which completes the proof of Lemma~\ref{lem:4.8}. Note that 
$(2s)(2s+1)(2s+2)(2s+3) \ge 2s(2s+2)(3s+3)$.
\end{proof}
\begin{remark}
\label{rem:4.9}
Lemmas~\ref{lem:4.7} and \ref{lem:4.8} show that whenever $u \ge \gamma >0$,
$s >0$, $k=1$ or $k=2$, and $1 \le j \le 4$,
\begin{equation*}
\frac{|D_k^j G(u,v;s)|}{G(u,v;s)} \le (2s)(2s+1)\cdots (2s+j-1) \gamma^{-j}.
\end{equation*}
We have not determined whether the above inequality holds for all $j \ge 1$.
\end{remark}

Using Lemmas~\ref{lem:4.7} and \ref{lem:4.8}, we can give uniform estimates
for the quantities $(\partial/\partial x)^j v_s(x,y)/v_s(x,y)$ and
$(\partial/\partial y)^j v_s(x,y)/v_s(x,y)$, where $s >0$, $1 \le j \le 4$,
and $v_s(x,y)$ is the unique (to within normalization) strictly positive
eigenvector of the linear operator $\Lambda_s: C^m(\bar H) \to
C^m(\bar H)$ in \eqref{4.2} for $m \ge 1$.

\begin{thm}
\label{thm:4.10}
Let $\B$ be a finite set of complex numbers $\beta$ such that
$\Re(\beta) \ge \gamma \ge 1$ for all $\beta \in \B$.  For $\beta \in
\B$ and $s >0$, define $\theta_{\beta}(z) = (z + \beta)^{-1}$. Let $H$
be a bounded, mildly regular open subset of $\C:= \R^2$ such that $H
\supset G_{\gamma}=\{z\in \C : |z - 1/(2 \gamma)| < 1/(2
\gamma)\}$, and $\Re(z) >0$ for all $z \in H$, so $\theta_{\beta}(H)
  \subset G_{\gamma}$ for all $\beta \in \B$. For a positive integer
$m$, define a real Banach space $C^m(\bar H) = X$ and defined a bounded
linear operator $\Lambda_s: X \to X$ by
\begin{equation*}
(\Lambda_s f)(z) = \sum_{\beta \in \B} \Big| \frac{d}{dz} \theta_{\beta}(z)\Big|^s
f(\theta_{\beta}(z)).
\end{equation*}
Then $\Lambda_s$ has a unique (to within normalization) positive
eigenvector $v_s \in X$ and $v_s \in C^{\infty}$.  Furthermore,
we have the following estimates for $(x,y) \in \bar H$.
\begin{gather}
\label{4.16}
-\frac{2s}{\gamma} \le \frac{\partial v_s(x,y)}{\partial x}
[v_s(x,y)]^{-1} \le 0,
\\
\label{4.17}
-\frac{s}{4\gamma^2(s+1)} \le \frac{\partial^2 v_s(x,y)}{\partial x^2}
[v_s(x,y)]^{-1} \le \frac{2s(2s+1)}{\gamma^2},
\\
\label{4.18}
-\frac{2s(2s+1)(2s+2)}{\gamma^3} \le \frac{\partial^3 v_s(x,y)}{\partial
x^3} [v_s(x,y)]^{-1} \le \frac{(2s)(2s+2)}{\gamma^3(s+2)^2},
\\
\label{4.19}
-\frac{2s(2s+2)(3s+3)}{\gamma^4} \le
\frac{\partial^4 v_s(x,y)}{\partial x^4} [v_s(x,y)]^{-1}
\le \frac{(2s)(2s+1)(2s+2)(2s+3)}{\gamma^4},
\\
\label{4.20}
\Big| \frac{\partial v_s(x,y)}{\partial y}\Big| [v_s(x,y)]^{-1}
\le \frac{s}{\gamma},
\\
\label{4.21}
-\frac{2s}{\gamma^2} \le \frac{\partial^2 v_s(x,y)}{\partial y^2}
[v_s(x,y)]^{-1} \le \frac{2s(2s+1)}{4\gamma^2},
\\
\label{4.22}
\Big|\frac{\partial^3 v_s(x,y)}{\partial y^3}\Big| [v_s(x,y)]^{-1} \le 
\frac{(2s)(2s+2)}{\gamma^3} \max\{25\sqrt{5}/72, (2s+1)/8\},
\\
\label{4.23}
-\frac{2s(2s+2)(3s+3)}{\gamma^4} \le
\frac{\partial^4 v_s(x,y)}{\partial y^4}[v_s(x,y)]^{-1}
\le \frac{(2s)(2s+1)(2s+2)(2s+3)}{\gamma^4}.
\end{gather}
Hence, if $D_1 = \partial/\partial x$ and $D_2 = \partial/\partial y$, we have
for $k=1,2$ and $1 \le j \le 4$ that
\begin{equation}
\label{4.24}
\frac{|D_k^j v_s(x,y)|}{v_s(x,y)}
 \le \frac{(2s)(2s+1) \cdots (2s+j-1)}{\gamma^j}.
\end{equation}
\end{thm}
\begin{proof}
For any integer $m \ge 1$, we can view $\Lambda_s$ as a bounded linear
operator of $C^m(\bar H)$ to $C^m(\bar H)$.  We know that $\Lambda_s$ has a
strictly positive eigenvector $v_s(x,y) \in C^m(\bar H)$ such that
$\sup\{v_s(x,y) \, | \, (x,y) \in\bar H\} =1$.  By the uniqueness of this
eigenvector, $v_s(x,y)$ must actually be $C^{\infty}$.

Using the notation of \eqref{4.6} and \eqref{4.7} and also using \eqref{4.15}
in Lemma~\ref{lem:4.3}, we see that
\begin{equation*}
\Big|\frac{d}{dz} \theta_{\tilde \omega}(z)\Big|^s = |B_{n-1}|^{-2s}
|z + B_n/B_{n-1}|^{-2s}.
\end{equation*}
By Lemma~\ref{lem:4.3}, $\Re(B_n/B_{n-1}) \ge \gamma_{\omega} \ge \gamma$, so
writing $\Im(B_n/B_{n-1}) = \delta_{\omega}$, we obtain that for $k=1,2$ and
$1 \le j$,
\begin{multline}
\label{4.25}
D_k^j\left(\Big|\frac{d}{dz} \theta_{\tilde \omega}(z)\Big|^s\right) 
\Big|\frac{d}{dz} \theta_{\tilde \omega}(z)\Big|^s
\\
= D_k^j\Big[(x+ \gamma_{\omega})^2 + (y+ \delta_{\omega})^2\Big]^{-s}
\Big[(x+ \gamma_{\omega})^2 + (y+ \delta_{\omega})^2\Big]^{s}.
\end{multline}
However, if we write $(x+ \gamma_{\omega}) = u \ge \gamma$ and
$(y+ \delta_{\omega}) =v$, we see that
\begin{multline}
\label{4.26}
\left(\Big(\frac{\partial}{\partial x}\Big)^j
\Big[(x+ \gamma_{\omega})^2 + (y+ \delta_{\omega})^2\Big]^{-s}\right)
\Big[(x+ \gamma_{\omega})^2 + (y+ \delta_{\omega})^2\Big]^{-s}\
\\
= \left[\Big(\frac{\partial}{\partial u}\Big)^jG(u,v;s)\right]
\left[G(u,v;s\right]^{-1},
\end{multline}
where the right hand side of the above equation is evaluated at 
$u = x+ \gamma_{\omega}$ and $v = y+ \delta_{\omega}$.  If we combine
\eqref{4.25} and \eqref{4.26} with the estimates in Lemma~\ref{lem:4.7}
and if we then use \eqref{4.11}, we obtain the estimates on
$(\partial/\partial x)^j v_s(x,y)$ given in \eqref{4.16} - \eqref{4.19}.

Similarly, we have
\begin{multline}
\label{4.27}
\left(\Big(\frac{\partial}{\partial y}\Big)^j
\Big[(x+ \gamma_{\omega})^2 + (y+ \delta_{\omega})^2\Big]^{-s}\right)
\Big[(x+ \gamma_{\omega})^2 + (y+ \delta_{\omega})^2\Big]^{-s}\
\\
= \left[\Big(\frac{\partial}{\partial v}\Big)^jG(u,v;s)\right]
\left[G(u,v;s\right]^{-1}.
\end{multline}
If we combine \eqref{4.25} and \eqref{4.27} with the estimates in
Lemma~\ref{lem:4.8} and if we then use \eqref{4.11}, we obtain the
estimates on $(\partial/\partial y)^j v_s(x,y)$ given in \eqref{4.20}
- \eqref{4.23}.
\end{proof}

\begin{remark}
\label{rem:8.3}
It turns out that exactly the same estimates given in Theorem~\ref{thm:4.10}
hold for a more general class of Perron-Frobenius operators which we shall
need later.  Let notation be as in Theorem~\ref{thm:4.10}, so $\Re(\beta)
\ge \gamma \ge 1$ for $\beta \in \B$ and $\theta_{\beta}(z) = 1/(z + \beta)$
for $\beta \in \B$.  Let $\A$ be a finite index set (possibly empty) of
integers disjoint from $\B$ and for $j \in \A$, let $z_j \in H$ be a given
point, $a_j$ a positive real, and $\theta_j:H \to G$ the map defined
by $\theta_j(z) = z_j$ for all $z \in H$, so $\Lip(\theta_j) =0$. If $m$ is a
positive integer and $s \ge0$, define a bounded linear map
$A_s:X := C^m(\bar H) \to X$ by
\begin{equation}
\label{8.28}
(A_s f)(z) = \sum_{\beta \in \B} \frac{1}{|z+ \beta|^{2s}}
f(\theta_{\beta}(z)) + \sum_{j \in \A} a_j^s f(\theta_{j}(z)).
\end{equation}
Notice that $A_s$ satisfies all the hypotheses of Theorem~\ref{thm:1.1},
so all the conclusions of  Theorem~\ref{thm:1.1} hold. In particular,
$A_s$ has a unique (to within normalization) strictly positive eigenvector
$w_s \in C^m(\bar H)$.  Because the eigenvector $w_s$ is unique and $m \ge 1$
is arbitrary, $w_s$ is $C^{\infty}$ on $H$.

Define an index set $\D = \A \cup \B$ and for $\delta \in \D$, define
$b_{\delta}(z) = 1/|z+ \beta|^{2s}$ if $\delta = \beta \in \B$ and 
$b_{\delta}(z) = a_j$ if $\delta = j \in \A$. As usual, if $\mu$ is a
positive integer, let
\begin{equation*}
\D_{\mu} = \{ \omega = (\delta_1, \delta_2, \ldots, \delta_{\mu}) \, | \,
\delta_k \in \D \text{ for } 1 \le k \le \mu\}.  
\end{equation*}
If $D_1 = \partial/\partial x$ and $D_2= \partial/\partial y$,
for $k \ge 1$ and $p=1$ or $2$, we know that (writing $z = x+ iy:=(x,y)$)
\begin{equation*}
\frac{D_p^k w_s(x,y)}{w_s(x,y)} = \lim_{\mu \rightarrow \infty}
\frac{D_p^k \Big(\sum_{\w \in \D_{\mu}} b_{\w}(x,y)\Big)}{\sum_{\w \in \D_{\mu}}
  b_{\w}(x,y)}.
\end{equation*}
If $\w = (\delta_1, \delta_2, \ldots, \delta_{\mu}) \in \D_{\mu}$ and
$\delta_k \notin \A$ for $1 \le k \le \mu$, we have seen in
Lemmas~\ref{lem:4.7} and \ref{lem:4.8} that $D_p^k
b_{\w}(x,y)/b_{\w}(x,y)$ satisfies the same estimates given for $D_p^k
v_s(x,y)/v_s(x,y)$ in equations \eqref{4.16}- \eqref{4.24}.  Thus assume that
$\delta_t \in \A$ for some $t$, $1 \le t \le \mu$ and $\delta_{t^{\prime}}
\notin \A$ for $t^{\prime} < t$.  A little thought shows that if $t=1$,
$b_{\w}(z)$ is a constant.  If $t=2$, $b_{\w}(z) = c(\w,z) b_{\delta_1}(z)$,
where $c(\w,z)$ is a positive constant.  Generally, if $2 \le t \le \mu$,
$b_{\w}(z) = c(\w,z)b_{\w_{t-1}}(z),$ where $c(\w,z)$ is a positive constant
and $\w_{t-1} = (\delta_{t-1}, \delta_{t-2}, \ldots, \delta_1) \in \D_{t-1}$
and $\delta_1, \delta_2, \ldots, \delta_{t-1} \in \B$.  It follows that $D_p^k
b_{\w}(x,y)/b_{\w}(x,y) =0$ if $t=1$ and otherwise
\begin{equation*}
D_p^k b_{\w}(x,y)/b_{\w}(x,y) = D_p^k b_{\w_{t-1}}(x,y)/b_{\w_{t-1}}(x,y).
\end{equation*}
But using Lemmas~\ref{lem:4.7} and \ref{lem:4.8} again, it follows
that if $\delta_t \in \A$ for some $t$, $1 \le t \le \mu$,
$D_p^k b_{\w}(x,y)/b_{\w}(x,y)$ is identically zero or 
$D_p^k b_{\w}(x,y)/b_{\w}(x,y)$ satisfies the same estimates given for
$D_p^k v_s(x,y)/v_s(x,y)$.  It follows that $D_p^k w_s(x,y)/w_s(x,y)$
satisfies the same estimates given for $D_p^k v_s(x,y)/v_s(x,y)$
in Theorem~\ref{thm:4.10}.
\end{remark}

\begin{cor}
\label{cor:4.11}
Let notation and hypotheses be as in Theorem~\ref{thm:4.10}.  
If $z_0 =(x_0,y_0) \in H$ and $z_1 =(x_1,y_1) \in H$ ,
\begin{equation}
\label{vzrelate}
v_s(z_0) \le v_s(z_1) \exp{[(\sqrt{5}s/\gamma) |z_1 - z_0|]}.
\end{equation}
\end{cor}
\begin{proof}
Let $H_1 \supset H$ be a convex, bounded open set such that $\Re(z) >0$
for all $z \in H_1$.  (As usual, we identify $x+iy$ with $(x,y)$.)  For
$z \in H_1$ and $\Lambda_s$ given by \eqref{4.2} and viewed as a bounded
linear operator $\Lambda_s: C^m(\bar H_1) \to C^m(\bar H_1)$, $\Lambda_s$
has a strictly positive eigenvector $\hat v_s: \bar H_1 \to (0, \infty)$
in $C^m(\bar H_1)$.  By uniqueness, $\hat v_s(z) = v_s(z)$ for all
$z \in H$.  Thus, by replacing $H$ by $H_1$, we can assume that $H$ is convex.

Define $z_t = (1-t)z_0 + t z_1$ for $0 \le t \le 1$ and note that
\begin{multline*}
\Big|\int_0^1 \frac{d}{dz} \log(v_s(z_t)) \, dt \Big| = \Big|\log\Big(
\frac{v_s(z_1)}{v_s(z_0)}\Big)\Big|
\\
\le \int_0^1\Big|\frac{D_1 v_s(z_t)}{v_s(z_t)}(x_1-x_0)
+ \frac{D_2v_s(z_t)}{v_s(z_t)} (y_1 - y_0) \Big| \, dt.
\end{multline*}
Using \eqref{4.16} and \eqref{4.20}, we obtain
\begin{equation*}
\Big|\log\Big(
\frac{v_s(z_1)}{v_s(z_0)}\Big)\Big| \le 
\int_0^1 \Big| \frac{2s}{\gamma}(x_1-x_0) + \frac{s}{\gamma}(y_1 - y_0)
\Big| \, dt \le \frac{\sqrt{5} s}{\gamma} \sqrt{(x_1-x_0)^2 + (y_1 - y_0)^2},
\end{equation*}
which gives the estimate in Corollary~\ref{cor:4.11}.
\end{proof}

\begin{remark}
\label{rem:8.4}
If $\B \subset \C$ is an infinite, countable index set with $\Re \beta \ge
\gamma \ge 1$ and $\theta_{\beta}(z) = 1/(z + \beta)$, we can consider, in the
notation of Corollary~\ref{cor:4.11}, $L_s:C(\bar H_1) \to C(\bar H_1)$ given
by $(L_sf)(z) = \sum_{\beta \in \B} |\theta_{\beta}^{\prime}(z)|^s
f(\theta_{\beta}(z))$ for $s > \sigma(\B)$, Here we assume that there exists
$z_* \in G_{\gamma}$ and $s_* >0$ such that $\sum_{\beta \in \B}
[b_{\beta}(z_*)]^{s_*} < \infty$, where we have written $b_{\beta}(z):=
|\theta_{\beta}^{\prime}(x)|$.  If then follows there exists a number
$\sigma(\B) \ge 0$ such that for all $s > \sigma(\B)$, the map $L_s$ defines a
bounded linear map of $C(\bar H_1)$ to $C(\bar H_1)$, while $\sum_{\beta \in
  \B} [b_{\beta}(z)]^s = \infty$ for all $z \in \bar H_1$ and all $s$ with $0
\le s < \sigma(\B)$. By using Lemmas~\ref{lem:4.7} and \ref{lem:4.8}
and the argument in Corollary~\ref{cor:4.11}, we see that for all $\w \in
\B_{\mu}$, and all $\mu \ge 1$, $b_{\w}^s(x,y) \in K(\sqrt{5}s/\gamma; \bar
H_1)$. With the aid of Lemma 5.3 in \cite{N-P-L} and Theorem 5.3 on page
86 in \cite{A}, we see that $L_s$ has a unique strictly positive eigenvector
$v_s$, and with the aid of Corollary~\ref{cor:1.4} in Section~\ref{sec:exist},
we conclude that $v_s \in K(\sqrt{5}s/\gamma; \bar H_1)$. In other words,
the conclusion of Corollary~\ref{cor:4.11} is also valid when $\B$ is
countable but not finite.
\end{remark}

If (H7.1) is satisfied and all elements of $\B$ are real, we can, as
already noted, restrict attention to the real line, take $G_{\gamma}:=
(0, \gamma)$ and $H:= (0,a)$ to be open intervals with $a \ge \gamma$
and let $\Lambda_s$ be given by \eqref{4.3} with $f \in C^m(\bar H)$.
Then \eqref{4.11} remains valid, but with $z$ replaced by $x \in \bar
H \subset \R$ and $D^{\alpha}$ replaced by $D_1^{\nu}$, and $D_1 =
d/dx$. Furthermore, for a fixed $\w = (\beta_1, \beta_2, \ldots,
\beta_{\nu}) \in \B_{\nu}$, Lemma~\ref{lem:4.3} implies that there exists
$\gamma_{\omega}$, dependent on $\omega$ such that $\gamma_{\omega} \ge \gamma$,
so that after using \eqref{4.15}, we obtain
\begin{equation*}
D_1^{\nu} (|D_1 \theta_{\tilde \w}(x)|^s)
(|D_1 \theta_{\tilde \w}(x)|^s)^{-1} 
= D_1^{\nu}[(x+\gamma_{\w})^{-2s}] (x+\gamma_{\w})^{2s}.
\end{equation*}
In this case, it is easy to carry out the calculation explicitly and obtain
for all $\w \in \B_{\nu}$ and $x \in \bar H$ that
\begin{equation}
\label{4.28}
0 < (-1)^{\nu} D_1^{\nu} (|D_1 \theta_{\tilde \w}(x)|^s)
(|D_1 \theta_{\tilde \w}(x)|^s)^{-1}
\le \frac{(2s)(2s+1) \cdots (2s+ \nu-1)}{\gamma^{\nu}}.
\end{equation}
By using \eqref{4.11} and \eqref{4.28}, we thus obtain the following theorem.

\begin{thm}
\label{thm:4.12}
Let $\gamma \ge 1$ be a fixed real number and let $\B$ be a finite set
of real numbers $\beta$ with $\beta \ge \gamma$ for all $\beta \in
\B$.  Let $G_{\gamma} = (0, \gamma)$ and $H = (0,a) \supset
G_{\gamma}$ be open intervals of real numbers, and for a positive integer $m$,
let $X_m$ denote the real Banach space $C^m(\bar H)$. For $s >0$ define
a bounded linear operator $\Lambda_s: X_m \to X_m$ by
\begin{equation*}
(\Lambda_s(f))(x) = \sum_{\beta \in \B} (x+ \beta)^{-2s}f(1/(x+ \beta)).
\end{equation*}
Then $\Lambda_s$ has a unique, normalized, 
strictly positive eigenvector $v_s(x)$, $v_s$ is $C^{\infty}$, and for
all $\nu \ge 1$ and $x \in \bar H$,
\begin{equation*}
0 \le (-1)^{\nu}\frac{D_1^{\nu} v_s(x)}{v_s(x)} \le
\frac{(2s)(2s+1) \cdots (2s+ \nu-1)}{\gamma^{\nu}}.
\end{equation*}
\end{thm}

\begin{remark}
\label{rem:4.12}
One can prove that the eigenvector $v_s(x)$ in Theorem~\ref{thm:4.12} extends to
an analytic function $v_s(z)$ defined on $\{z \in \C \, | \, \Re(z) >0\}$.
In fact, much more general analyticity results of this type can be
established.  Since we shall not utilize such analyticity theorems in this
paper, we omit the proofs.
\end{remark}
\begin{remark}
\label{rem:4.13}
Throughout this section we have assumed for convenience that 
$1 \le \gamma \le \Re(\beta)$ 
for all $\beta \in \B$, where $\B$ is a finite set of
complex numbers.  In fact, the main results of this section can be obtained
under the assumption that $\Re(\beta) \ge \gamma > 0$ for all $\beta \in \B$.
In the notation of this section, the key point is to prove that there exists
an integer $\nu \ge 1$ and a constant $\kappa <1$ such that for all 
$z,w \in \C$ with $\Re(z) >0$ and $\Re(w) >0$ and for all $\omega \in \B_{\nu}$,
one has
\begin{equation}
\label{4.29}
|\theta_{\omega}(z) -  \theta_{\omega}(w)| \le \kappa^{\nu}|z-w|.
\end{equation}
Inequality \eqref{4.29} can be established with the aid of the
Carath\'eodory-Reiffen-Finsler metric (see \cite{R} for the definition
and basic results about the CRF metric) and the argument given in Section 6
of  \cite{N-P-L}.  Once \eqref{4.29} has been established in the case
$\Re(\beta) \ge \gamma >0$, all the theorems of this section follow by
the same arguments.
\end{remark}

\section{Computing the Spectral Radius of $A_s$ and $B_s$}
\label{sec:compute-sr}
In previous sections, we have constructed matrices $A_s$ and $B_s$ such that
$r(A_s) \le r(L_s) \le r(B_s)$.  The $m \times m$ matrices $A_s$ and $B_s$
have nonnegative entries, so the Perron-Frobenius theory for such matrices
implies that $r(B_s)$ is an eigenvalue of $B_s$ with corresponding
nonnegative eigenvector, with a similar statement for  $A_s$. One might
also hope that standard theory (see \cite{D}) would imply that $r(B_s)$,
respectively $r(A_s)$, is an eigenvalue of $B_s$ with algebraic multiplicity
one and that all other eigenvalues $z$ of $B_s$ (respectively, of $A_s$)
satisfy $|z| < r(B_s)$ (respectively, $|z| < r(A_s)$). Indeed, this would
be true if $B_s$ were {\it primitive}, i.e., if $B_s^k$ had all positive
entries for some integer $k$.  However, typically $B_s$ has many zero
columns and $B_s$ is neither primitive nor {\it irreducible} (see \cite{D});
and the same problem occurs for $A_s$.  Nevertheless, the desirable spectral
properties mentioned above are satisfied for both $A_s$ and $B_s$. Furthermore
$B_s$ has an eigenvector $w_s$ with all positive entries and with eigenvalue
$r(B_s)$; and if $x$ is any $m \times 1$ vector with all positive entries,
\begin{equation*}
\lim_{k \rightarrow \infty} \frac{B_s^k(x)}{\|B_s^k(x)\|} =
  \frac{w_s}{\|w_s\|},
\end{equation*}
where the convergence rate is geometric.  Of course, corresponding results
hold for $A_s$.  Such results justify standard numerical algorithms for
approximating $r(B_s)$ and $r(A_s)$.

In this section, we shall prove these results in the one dimensional case.
Similar theorems can be proved in the two dimensional case, but for reasons
of  length, we shall restrict our attention here to the one dimensional case
and delay a more comprehensive discussion of the underlying issues to a
later paper.  The basic point, however, is simple: Although $A_s$ and $B_s$
both map the cone $K$ of nonnegative vectors in $\R^m$ into itself, $K$ is
{\it not} the natural cone in which such matrices should be studied.  We
shall define below, for large positive real $M$, a cone $K_M \subset K$ such
that $A_s(K_M) \subset K_M$ and $B_s(K_M) \subset K_M$.  The cone $K_M$ is
the discrete analogue of a cone which has been used before in the infinite
dimensional case (see \cite{N-P-L}, Section 5 of \cite{A}, Section 2 of
\cite {F} and \cite{Bumby1}).  Once we have proved that $A_s(K_M) \subset K_M$
and $B_s(K_M) \subset K_M$, we shall see that the desired spectral properties
of $A_s$ and $B_s$ follow easily.
In a later paper, we shall consider higher order piecewise polynomial
approximations to the positive eigenvector $v_s$ of $L_s$.  We shall show
that the corresponding matrices $A_s$ and $B_s$ no longer have all
nonnegative entries, but still, under appropriate assumptions, map
$K_M$ into $K_M$.

Throughout this section, $[a,b]$ will denote a fixed, closed, bounded interval
and $s$ a fixed nonnegative real.  For a given positive integer $n \ge 2$, and
for integers $j$, $0 \le j \le n$, we shall write $h = (b-a)/n$ and $x_j = a +
jh$. $C$ will denote a fixed constant and we shall always assume at least that
$h \le 1$ and
\begin{equation}
\label{9.1}
|C| h /4 \le 1.
\end{equation}
In our applications $C$ will depend on $s$, but we shall not indicate
the dependence here.
If $f: \{x_j \, | \, 0 \le j \le n\} \to \R$, one can extend $f$ to a map
$f^I:[a,b] \to \R$ by linear interpolation, so
\begin{equation}
\label{9.2}
f^I(x) = \frac{x-x_j}{h}f(x_{j+1}) + \frac{x_{j+1} - x}{h} f(x_j), \quad
\text{for } x_j \le x \le x_{j+1}, \quad 0 \le j < n.
\end{equation}

We shall denote by $X_n$ (or $X$, if $n$ is obvious), the real vector space
of maps $f:\{x_j \, | \, 0 \le j \le n\} \to \R$; obviously $X_n$ is
linearly isomorphic to $\R^{n+1}$.  For a given positive real $M$, we shall
denote by $K_M \subset X_n$ the closed cone given by
\begin{multline}
\label{9.3}
K_M = \{f \in X_n \, | \, f(x_{j+1}) \le f(x_j) \exp(Mh) 
\\
\text { and } f(x_j) \le f(x_{j+1})\exp(Mh), \quad 0 \le j < n\}.
\end{multline}

The reader can verify that if $f \in K_M$, $f(x_j) \ge 0$ for
$0 \le j \le n$, and either $f(x_j) >0$ for all $0 \le j \le n$,
or $f(x_j) =0$ for all $0 \le j \le n$.

If $x_j$, $0 \le j \le n$, are as above, define a map $Q:[a,b] \to [0,h^2/4]$
by
\begin{equation*}
Q(u) = (x_{j+1} - u)(u-x_j), \quad \text{for } x_j \le u \le x_{j+1}, \quad
0 \le j < n.
\end{equation*}

\begin{lem}
\label{lem:9.1}  Assume that $\beta \in K_{M_0} \setminus \{0\}$ for some
$M_0 >0$, that $0 < h \le 1$ and that $h$ and $C$ satisfy \eqref{9.1}.
Let $\theta:[a,b] \to [a,b]$ and define $\hat \beta_s \in X_n$ by
\begin{equation}
\label{9.4} 
\hat \beta_s(x_k) = [1 + \tfrac{1}{2} C Q( \theta(x_k))] [\beta(x_k)]^s.
\end{equation}
Then $\hat \beta_s \in K_{M_1}$, where $M_1 = sM_0 + (1+h)/2 \le M_0 + 1$.
\end{lem}
\begin{proof}
Define $\psi \in X_n$ by
\begin{equation*}
\psi(x_k) = 1 + \tfrac{1}{2} C Q(\theta(x_k))
\end{equation*}
and suppose we can prove that $\psi \in K_{(1+h)/2}$.  For notational
convenience define $b(x_k) = [\beta(x_k)]^s$.
Then for $0 \le k <n$,
we obtain
\begin{multline*}
\psi(x_k) b(x_k) \le \psi(x_{k+1}) \exp([1+h]h/2) b(x_{k+1}) \exp(sM_0 h)
 \\
= \psi(x_{k+1}) b(x_{k+1}) \exp(M_1h),
\end{multline*}
and the same calculation gives
\begin{equation*}
\psi(x_{k+1}) b(x_{k+1}) \le \exp(M_1 h) \psi(x_k) b(x_k),
\end{equation*}
which implies that $x_k \mapsto \psi(x_k) b(x_k)$ is an element
of $K_{M_1}$.

Define $\delta = (1+h)/2$. Since $\psi(x_k) >0$ for $0 \le k \le n$, one
can check that $\psi(\cdot) \in K_{\delta}$ if and only if, for $0 \le k < n$,
\begin{equation*}
|\log( \psi(x_{k+1})) - \log( \psi(x_{k}))| = 
\Big| \log \Big(\frac{\psi(x_{k+1})}{\psi(x_{k})}\Big)\Big| \le \delta h.
\end{equation*}
Given $x_k$ and $x_{k+1}$ with $0 \le k <n$, write $\xi = \theta(x_k)$
and $\eta = \theta(x_{k+1})$.  Define $u:= \tfrac{1}{2} C Q(\theta(x_k))$ and
$v = \tfrac{1}{2} CQ(\theta(x_{k+1}))$, so $\psi(x_k) = 1 +u$
and $\psi(x_{k+1}) = 1 + v$.  Because $u$ and $v$ both lie in the
interval $[0, Ch^2/8]$, \eqref{9.1} implies that $|u-v| \le h/2$,
$|u| \le h/2$ and $|v| \le h/2$.  It follows
that
\begin{equation*}
|\log( \psi(x_{k})) - \log( \psi(x_{k+1}))| = |\log(1+u) - \log(1+v)|
= \Big|\int_{1+v}^{1+u} (1/t) \, dt \Big|.
\end{equation*}
Because $0 \le 1/t \le 1/(1 - h/2) \le 1 + h$ for all $t \in [1+v,1+u]$,
we obtain
\begin{equation*}
|\log( \psi(x_{k})) - \log( \psi(x_{k+1}))| \le (1+h)|u-v| \le (1+h) h/2,
\end{equation*}
which proves the lemma.
\end{proof}

\begin{lem}
\label{lem:9.2}
Let assumptions and notation be as in Lemma~\ref{lem:9.1}.  Let $\delta$
denote a fixed positive real and $s$ a fixed nonnegative real.  Assume,
in addition that $\theta:[a,b] \to [a,b]$ is a Lipschitz map with
$\Lip(\theta) \le c <1$ and that, for $h = (b-a)/n$ and $M_1$ as in
 Lemma~\ref{lem:9.1}, $\exp(-[M_1 + \delta] h)
\ge (1+c)/2$ and $M >0$ is such that $\exp(Mh) \ge 2$.  Define a linear map
$L_s:X_n \to X_n$ by $L_s(f) = g$, where
\begin{equation*}
g(x_k) := f^I(\theta(x_k)) \hat \beta_s(x_k), \quad 0 \le k \le n.
\end{equation*}
Then, if $K_M \subset X_n$ is defined by \eqref{9.3}, $L_s(K_M) \subset
K_{M-\delta}$.
\end{lem}
\begin{proof}
For a fixed $k$, $0 \le k <n$, define $\xi = \theta(x_k)$
and $\eta = \theta(x_{k+1})$.  We must prove that if $h$ and $M$ satisfy the
above constraints and $f \in K_M$, then
\begin{align*}
f^I(\xi) \hat \beta_s(x_k) &\le \exp([M-\delta]h) f^I(\eta) \hat
\beta_s(x_{k+1}),
\\
f^I(\eta) \hat \beta_s(x_{k+1}) &\le \exp([M-\delta]h) f^I(\xi) \hat
\beta_s(x_{k}).
\end{align*}
Using Lemma~\ref{lem:9.1}, we see that $x_k \mapsto \hat \beta_s(x_{k})$
is an element of $K_{M_1}$, so the above inequalities will be satisfied if
\begin{align}
\label{9.5}
f^I(\xi) &\le \exp([M-M_1-\delta]h) f^I(\eta),
\\
\label{9.6}
f^I(\eta) &\le \exp([M- M_1-\delta]h) f^I(\xi).
\end{align}
For notational convenience, we write $M_2 = M_1 + \delta$. By interchanging
the roles of $\xi$ and $\eta$, we can assume that $\eta \le \xi$, and it
suffices to prove that \eqref{9.5} and \eqref{9.6} are satisfied for $M$
and $h$ as in the statement of the Lemma.  Define $j = n-1$ if $\xi \ge
x_{n-1}$ and otherwise define $j$ to be the unique integer, $0 \le j < n-1$,
such that $x_j \le \xi < x_{j+1}$.  Because $0  \le \xi- \eta \le c h <h$,
there are only two cases to consider: either (i) $x_j \le \eta \le \xi$ 
or (ii) $x_{j-1} < \eta <x_j$ and $x_j \le \xi < x_{j+1}$.

We first assume that we are in case (i), so $\xi, \eta \in [x_j,x_{j+1}]$
and $0 \le \xi - \eta \le c h$,  Using \eqref{9.2}, we see that \eqref{9.5}
is equivalent to proving
\begin{multline}
\label{9.5p}
(x_{j+1} - \xi) f(x_j) + (\xi - x_j)  f(x_{j+1}) 
\\
\le
\exp([M-M_2]h)[(x_{j+1} - \eta) f(x_j) + (\eta - x_j)f(x_{j+1})].
\end{multline}
Subtracting $(x_{j+1} - \eta) f(x_j) + (\eta - x_j)f(x_{j+1})$ from both sides
of \eqref{9.5p} shows that \eqref{9.5p} will be satisfied if
\begin{multline}
\label{9.5pp}
(\xi - \eta) [f(x_{j+1}) -f(x_j)] 
\\
\le [\exp([M-M_2]h) -1][(x_{j+1} - \eta) f(x_j) + (\eta - x_j)f(x_{j+1})].
\end{multline}
Equation \eqref{9.5pp} will certainly be satisfied if $f(x_{j+1}) \le f(x_j)$,
so we can assume that $f(x_{j+1}) - f(x_j) >0$ and 
$1 < f(x_{j+1})/f(x_j) \le \exp(Mh)$.
If we divide both sides of \eqref{9.5pp} by $f(x_j)$ and recall that
$\xi-\eta \le ch$, we see that the left hand side of \eqref{9.5pp} is
dominated by $ch[\exp(Mh) -1]$, while the right hand side of \eqref{9.5pp}
is $\ge [\exp([M-M_2]h) -1]h$,  Thus, \eqref{9.5pp} will be satisfied if
\begin{equation}
\label{9.7}
c \le \frac{\exp([M-M_2]h) -1}{\exp(Mh) -1} = \exp(-M_2 h)
+ \frac{\exp(-M_2 h) -1}{\exp(Mh) -1}.
\end{equation}
If $h >0$ is chosen so that $\exp(-M_2 h) \ge (1+c)/2$, a calculation
shows that \eqref{9.7} will be satisfied if
\begin{equation}
\label{9.8} M \ge \log(2)/h,
\end{equation}
where $\log$ denotes the natural logarithm.  Thus, if $h >0$ satisfies
\eqref{9.1}, $M \ge \log(2)/h$, and $\exp(-M_2 h) \ge (1+c)/2$, 
\eqref{9.5} is satisfied in case (i).  Under the same conditions on
$h$ and $M$, an exactly analogous argument shows that (in case (i)),
\eqref{9.6} is also satisfied.

We next consider case (ii), so $\xi \in [x_j, x_{j+1}]$, 
$\eta \in [x_{j-1}, x_j]$ and $0 \le \xi - \eta \le ch$. It follows that
$\xi - x_j = c_1 h$ and $x_j - \eta = c_2 h$, where $c_1 \ge 0$,
$c_2 \ge 0$, and $c_1 + c_2 \le c <1$. As before, we need to show that
inequalities \eqref{9.5} and \eqref{9.6} are satisfied. Inequality \eqref{9.6}
takes the form
\begin{multline}
\label{9.9}
f^I(\eta) = \frac{\eta-x_{j-1}}{h} f(x_j) + \frac{x_j- \eta}{h} f(x_{j-1})
\\
\le \exp([M-M_2]h)
\Big[\frac{\xi-x_{j}}{h} f(x_{j+1}) + \frac{x_{j+1}- \xi}{h} f(x_{j})\Big],
\end{multline}
which is equivalent to
\begin{equation}
\label{9.9p}
(\eta-x_{j-1}) + (x_j- \eta) \frac{f(x_{j-1})}{f(x_j)}
\le \exp([M-M_2]h)
\Big[(\xi-x_{j}) \frac{f(x_{j+1})}{f(x_j)} + (x_{j+1}- \xi)\Big],
\end{equation}
Since $f(x_{j-1})/f(x_j) \le \exp(Mh)$, $f(x_{j+1})/f(x_j) \ge \exp(-Mh)$,
$x_j - \eta = c_2 h$ and $\xi - x_j = c_1 h$, \eqref{9.9p} will be satisfied
if
\begin{equation}
\label{9.9pp}
(1-c_2) + c_2 \exp(Mh) \le \exp([M-M_2]h)[c_1 \exp(-Mh) + (1-c_1)].
\end{equation}
Because $c_2 \le c-c_1$, we have
\begin{equation*}
(1-c_2) + c_2 \exp(Mh) \le (1-c+c_1) + (c-c_1) \exp(Mh),
\end{equation*}
and inequality \eqref{9.9pp} will be satisfied if
\begin{equation}
\label{9.10}
(1+ c_1 -c) + (c-c_1) \exp(Mh) \le \exp(-M_2 h)[c_1 + (1-c_1) \exp(Mh)].
\end{equation}
A necessary condition that \eqref{9.10} be satisfied is that
$\exp(-M_2 h) \ge (c-c_1)/(1-c_1)$.  Since $(c-c_1)/(1-c_1) \le c$ and
$c < (1+c)/2$, we choose $h = (b-a)/n >0$ sufficiently small that
\begin{equation}
\label{9.11}
\exp(-M_2 h) \ge (1+c)/2.
\end{equation}
For this choice of $h$, \eqref{9.10} will be satisfied if
\begin{equation*}
(1+ c_1 -c) + (c-c_1) \exp(Mh) \le \frac{1+c}{2}[c_1 + (1-c_1) \exp(Mh)],
\end{equation*}
which is equivalent to
\begin{equation}
\label{9.12}
(1 + c_1/2)(1-c) \le [(1+c_1)(1-c)/2] \exp(Mh).
\end{equation}
Since $(2+c_1)/(1+c_1) \le 2$, \eqref{9.12} will be satisfied if
\begin{equation}
\label{9.13}
2 \le \exp(Mh).
\end{equation}
Thus \eqref{9.9} will be satisfied if $h$ satisfies \eqref{9.11} and, for
this $h$, $M$ satisfies \eqref{9.13}.

Inequality \eqref{9.5} will be satisfied in case (ii) if
\begin{equation}
\label{9.14}
(\xi-x_j) \frac{f(x_{j+1})}{f(x_j)} + (x_{j+1} - \xi)
\le \exp([M-M_2]h)
\Big[(\eta-x_{j-1}) + (x_{j}- \eta)
\frac{f(x_{j-1})}{f(x_j)}\Big].
\end{equation}
The same reasoning as above shows that if $h >0$ satisfies \eqref{9.11}
and $M$ then satisfies \eqref{9.13}, \eqref{9.14} will be satisfied.
Details are left to the reader.
\end{proof}

\begin{thm}
\label{thm:9.3}
Let $N$ denote a positive integer or $N = \infty$. For $1 \le j \le N$, assume
that $\theta_j:[a,b] \to [a,b]$ is a Lipschitz map with $Lip(\theta_j) \le c
<1$, $c$ independent of $j$.  For $1 \le j \le N$, assume that $\beta_j \in
K_{M_0} \setminus \{0\} \subset X_n$, where $M_0$ is independent of $j$.  For
$j \ge 1$, let $C_j$ be a real number with $|C_j| \le C$, where $C$ is
independent of $j$; and for a fixed $s \ge 0$, define $\hat \beta_{j,s} \in
X_n$ by
\begin{equation*}
\hat \beta_{j,s}(x_k) = [1 + \tfrac{1}{2} C_j Q(\theta_j(x_k))] [\beta_j(x_k)]^s,
\quad  0 \le k \le n.
\end{equation*}
Let $\delta >0$ be a given real number and for $j \ge 1$ define
a linear map $L_{j,s}: X_n \to X_n$ by
\begin{equation*}
(L_{j,s} f)(x_k) = \hat \beta_{j,s}(x_k) f^I(\theta_j(x_k)).
\end{equation*}
If $N =\infty$, assume that there exists $k_0$, $0 \le k_0 \le n$,
such that $\sum_{j=1}^{\infty} [\hat \beta_j(x_k)]^s < \infty$ and define a
linear map $L_s: X_n \to X_n$ by $L_s = \sum_{j=1}^N L_{j,s}$.
Assume that $h = (b-a)/n \le 1$ and $C h /4 \le 1$ and define
$M_2 = [sM_0 + (1+h)/2] + \delta$.  Assume also that
$\exp(-M_2 h) \ge (1+c)/2$ and that $M \in \R$ is such that $\exp(Mh) \ge 2$.
Then we have that $L_s(K_M \setminus \{0\}) \subset K_{M- \delta} \setminus
\{0\}$. 
\end{thm}
\begin{proof}
Lemma~\ref{lem:9.1} implies that $x_k \mapsto \hat \beta_{j,s}(x_k)$ is
an element of $K_{M_1}$, where $M_1 = s M_0 + (1+h)/2$.  It follows that if
$N = \infty$ and $\sum_{j=1}^N \hat \beta_{j,s}(x_{k_0}) < \infty$, it must
be true that $\sum_{j=1}^N \hat \beta_{j,s}(x_k) < \infty$ for all $k$,
$0 \le k \le N$; and $L_s:X_n \to X_n$ is also a well-defined bounded
linear map when $N=\infty$.  Under our hypotheses, Lemma~\ref{lem:9.2}
implies that $L_{j,s}(K_M \setminus \{0\}) \subset K_{M - \delta}\setminus
\{0\}$, so $L_{s}(K_M \setminus \{0\}) \subset K_{M - \delta}\setminus
\{0\}$.
\end{proof}

At this point we need to recall some general results relating to $u_0$-{\it
  positive linear operators.}  Recall that a closed subset $C$ of a Banach
space $Y$ is called a closed cone if (i) $ax + by \in C$ whenever $a$ and $b$
are nonnegative reals and $x,y \in C$ and (ii) $C \cap (-C) = \{0\}$, where $-C
= \{-x \, | \, x \in C\}$. A closed cone $C$ in a real Banach space $(Y,
\|\cdot \|)$ is called {\it reproducing} if $Y = \{f-g \, | \, f,g \in C \}$,
and a closed cone $C$ induces a partial ordering $\le _C$ on $Y$ by $x \le _C
y$ if and only if $y-x \in C$.  If $x \in C$ and $y \in C$, we shall say that
$x$ and $y$ are comparable (in the partial ordering $\le _C$) if there exist
positive reals $\alpha >0$ and $\beta >0$ such that $ \alpha x \le _C y$ and
$y \le_C \beta x$.  If $x,y$ are comparable, we shall write
\begin{equation*}
M(y/x;C) = \inf\{\beta >0 \, | \, y \le_C \beta x\}, \qquad
m(y/x;C) = \sup\{\alpha >0 \, | \, \alpha x \le_C y\}.
\end{equation*}
The following proposition can be found in \cite{X} and \cite{Y}.
\begin{prop}
\label{prop:9.4}
Let $C$ be a closed, reproducing cone in a real Banach space $Y$, and let $A:Y 
\to Y$ be a bounded linear operator such that $A(C) \subset C$. Assume that
there exists $v \in C \setminus \{0\}$ and $r >0$ such that $Av = rv$.
Assume (this is the $u_0$-positivity of $A$) that there exists $u_0 \in C
\setminus \{0\}$ with the following property: For every $x \in 
C \setminus \{0\}$, there exists a positive integer $m(x)$ and positive
reals $a(x)$ and $b(x)$ such that either (i) $a(x) u_0 \le_C
A^{m(x)} (x) \le_C b(x) u_0$ or (ii) $A^{m(x)}(x) =0$. If $\hat A$ denotes
the complexification of $A$, $r$ is an eigenvalue of $\hat A$ of algebraic
multiplicity $1$; and if $A w = \lambda w$ for some $w \in C \setminus \{0\}$
and $\lambda >0$, $\lambda =r$ and $w$ is a positive scalar multiple
of $v$.  If $z \in \C$ is an eigenvalue of $\hat A$ and $z \neq r$,
then $|z| < r$.
\end{prop}
\begin{remark}
\label{rem:9.5} Note that Proposition~\ref{prop:9.4} only gives information
about eigenvalues of $\hat A$.  If $\sigma(\hat A)$ denotes the spectrum
of $\hat A$, it is possible, under the assumptions of
Proposition~\ref{prop:9.4}, that there exists $z \in \sigma(\hat A)$
with $|z| = r$ and $z \neq r$.
\end{remark}
\begin{remark}
\label{rem:9.6}
Proposition~\ref{prop:9.4} can be derived from the so-called Birkhoff-Hopf
theorem, though we shall not do so here.  We refer the reader to the 
papers \cite{S}, \cite{V}, and \cite{AA} for the original work by
Birkhoff, Hopf, and Samelson. A general version of the  Birkhoff-Hopf
theorem, applications to spectral theory, and references to the literature
can be found in \cite{T} and \cite{U}; see also Appendix A of \cite{G}.
Section 2.2 of \cite{F} (particularly Lemma 2.12) is closely related to
our work here.
\end{remark}
\begin{thm}
\label{thm:9.7}
Let notation and assumptions be as in Theorem~\ref{9.3}.  Then $L_s$ has an
eigenvector $v \in K_{M-\delta} \setminus \{0\}$ with eigenvalue $r >0$.
If $\hat L_s$ denotes the complexification of $L_s$, $r$ is an eigenvalue
of $\hat L_s$ of algebraic multiplicity one; and if $L_s w = \lambda w$
for some $w \in K_M \setminus \{0\}$, $\lambda =r$, and $w$ is a positive
multiple of $v$.  If $z$ is an eigenvalue of $\hat L$ and $z \neq r$,
then $|z| <r$.
\end{thm}
\begin{proof}
We shall need a very special case of Lemma 2.12 in \cite{F}.  Because
$M- \delta <M$, Lemma 2.12 in \cite{F} implies that all elements
$x,y \in K_{M-\delta} \setminus \{0\}$ are comparable with respect to the
partial ordering $\le_{K_M}$ given by $K_M \supset K_{M-\delta}$.
Furthermore, we have
\begin{equation*}
\sup\{M(y/x;K_M)/m(y/x;K_M) : x,y \in  K_{M-\delta} \setminus \{0\}\} < \infty.
\end{equation*}

Since Theorem~\ref{thm:9.3} implies that $L_s(K_M \setminus \{0\}) \subset
K_{M-\delta} \setminus \{0\}$, it follows that if $u \in K_{M-\delta}
\setminus \{0\}$, $L_s u \in K_{M-\delta} \setminus \{0\}$ and $u$ and $L_s u$
are comparable, so $L_s u \ge_{K_M} \alpha u$ for some $\alpha >0$.  This
implies that $r(L_s) \ge \alpha >0$.  In our particular case, the cone $K_M$
has nonempty interior in $X_n$, although in the generality of Lemma 2.12, this
is not usually true.  The Kre{\u\i}n-Rutman theorem implies that $L_s$ has an
eigenvector $v_s \in K_M$ with eigenvalue $r = r(L_s) >0$; and since $r v_s =
L_s(v_s)$, $v_s \in K_{M- \delta}$. If we define $u_0:=v_s$, Lemma 2.12 in
\cite{F} implies that $L_s(x)$ is comparable to $v_s$ (with respect to the
partial ordering $\le_{K_M}$) for all $x \in K_M \setminus
\{0\}$. Theorem~\ref{thm:9.7} now follows directly from
Proposition~\ref{prop:9.4}.
\end{proof}
\begin{remark}
\label{rem:9.8}
Since the linear maps $A_s$ and $B_s$ are both of the form of the map
$L_s$ in Theorem~\ref{thm:9.3}, Theorem~\ref{thm:9.7} implies the
desired spectral properties of $A_s$ and $B_s$.  With greater care it is
possible to use results in \cite{U} to estimate the so-called spectral
clearance $q(L_s)$ of $L_s$, given by
\begin{equation*}
q(L_s):= \sup\{|z|/r : z \in \sigma(L_s) \text{ and } z \neq r(L_s)\} < 1.
\end{equation*}
\end{remark}

\begin{remark}
\label{rem:9.9}
We claim that there is a constant $E$, which can be easily estimated, such
that, for $h = (b-a)/n$ sufficiently small,
\begin{equation*}
r(B_s) \le r(A_s)(1 + E h^2).
\end{equation*}
(Of course we already know that $r(A_s) \le r(B_s)$.) For a fixed $s \ge 0$,
let $\beta_j(\cdot)$ and $\theta_j(\cdot)$ be as in Theorem~\ref{thm:9.3}.
We know that $A_s$ and $B_s$ are of the form of $L_s$ in
Theorem~\ref{thm:9.3}, so we can write, for $0 \le k \le n$,
\begin{align*}
(A_s f)(x_k) &= \sum_{j-1}^N [ 1 +(C_j/2) Q(\theta_j(x_k))] [\beta_j(x_k)]^s
f^I(\theta_j(x_k),
\\
(B_s f)(x_k) &= \sum_{j-1}^N [ 1 +(D_j/2) Q(\theta_j(x_k))] [\beta_j(x_k)]^s
f^I(\theta_j(x_k).
\end{align*}
We assume that $h \le 1$ and $C h/4 \le 1$, where $C$ is a positive
constant such that $\max(|C_j|, |D_j|) \le C$ for $1 \le j \le N$. We
assume also that for $1 \le j \le N$, $C_j \le D_j$.  Let
$K= \{ f \in X_n \, | \, f(x_k) \ge 0 \text{ for } 0 \le k \le n\}$, so
$A_s(K) \subset K$ and $B_s(K) \subset K$.  Define $\mu \ge 1$ by
\begin{equation*}
\mu = \sup\{ [1 + \frac{D_j}{2} Q(\theta_j(x_k))][ 1 + \frac{C_j}{2}
Q(\theta_j(x_k))]^{-1}: 1 \le j \le N, 0 \le k \le N\} \ge 1.
\end{equation*}
Then for all $f \in K$ and $0 \le k \le n$,
$(B_s(f))(x_k) \le \mu (A_s(f))(x_k)$, which implies that
$r(B_s) \le \mu r(A_s)$.  Since $Q(u) \le h^2/4$, a little thought shows
that  $\mu \le (1 + Ch^2/8)(1 - Ch^2/8)^{-1} \le 1 + E h^2$,
which gives the desired estimate.
\end{remark}

\section{Log convexity of the spectral radius of 
$\Lambda_s$}
\label{sec:logconvex}
Throughout this section we shall assume that hypotheses (H5.1), (H5.2), and
(H5.3) in Section~\ref{sec:exist} are satisfied and we shall also assume that
$H$ is a bounded, open, mildly regular subset of $\R^n$. As in
Section~\ref{sec:exist}, we shall write $X= C^m(\bar H)$ and $Y= C(\bar H)$.
For $s \in \R$, we define $\Lambda_s: X \to X$ and $L_s:Y \to Y$ by
\begin{equation}
\label{2.1}
(\Lambda_s(f))(x) = \sum_{\beta \in \B} (b_{\beta}(x))^s f (\theta
_{\beta}(x))
\end{equation}
and
\begin{equation}
\label{2.2}
(L_s(f))(x) = \sum_{\beta \in \B} (b_{\beta}(x))^s f (\theta
_{\beta}(x)).
\end{equation}
Theorem~\ref{thm:1.1} implies that $r(\Lambda_s)$ is an algebraically simple
eigenvalue of $\Lambda_s$ for $s \in \R$ and that $\sup \{|z|: z \in
\sigma(\Lambda_s), z \neq r(\Lambda_s)\} < r(\Lambda_s)$, 
where $\sigma(\Lambda_s)$ denotes the spectrum of $\Lambda_x$.

Let $\hat X$ denote of the complexification of $X$, so $\hat X$ is the Banach
space of $C^m$ maps $f:H \to \C$ such that $x \mapsto
(D^{\alpha}f)(x)$ extends continuously to $\bar H$ for all multi-indices
$\alpha$ with $|\alpha| \le m$. For $s \in \C$ one can define
$\hat \Lambda_s: \hat X \to \hat X$ by
\begin{equation}
\label{2.3}
(\hat \Lambda_s(f))(x) = \sum_{\beta \in \B} (b_{\beta}(x))^s f (\theta
_{\beta}(x)):= \sum_{\beta \in \B} \exp(s \log b_{\beta(x)})   f(\theta_{\beta}(x)).
\end{equation}
The reader can verify that $s \mapsto \hat \Lambda_s \in \Lc(\hat X, \hat
X)$ is an analytic map. Because $r(\hat \Lambda_s)$ is an algebraically simple
eigenvalue of $\hat \Lambda_s$ for $s \in \R$ and $\sup \{|z|: z \in
\sigma(\Lambda_s), z \neq r(\Lambda_s)\} < r(\Lambda_s)$, it follows from the
kind of argument used on pages 227-228 of \cite{BB} that there is an open
neighborhood $U$ of $\R \in \C$ and the map $s \in U \mapsto r(\hat
\Lambda_s)$ is analytic on $U$.
\begin{thm}
\label{thm:2.1}
Assume that hypotheses (H5.1), (H5.2), and (H5.3) are satisfied with $m \ge 1$
and that $H \subset \R^n$ is a bounded, open mildly regular set. For $s \in
\R$, let $\Lambda_s$ and $L_s$ be defined by \eqref{2.1} and \eqref{2.2}.
Then we have that $s \mapsto r(\Lambda_s)$ is log convex, i.e., $s
\mapsto log(r(\Lambda_s))$ is convex on $[0,\infty)$.
\end{thm}
\begin{proof}
Because Theorem~\ref{thm:1.1} implies that $r(L_s) = r(\Lambda_s)$ for all real
$s$, it suffices to take $s_0 <s_1$, and $0 < t < 1$ and prove that
\begin{equation*}
r(L_{(1-t)s_0 + t s_1}) \le r(L_{s_0})^{1-t} r(L_{s_1})^t.
\end{equation*}
We shall use an old trick (see \cite{CC} and the references therein). Let
$v_{s_j}(x)$, $j=0,1$ denote the strictly positive eigenvector of $L_{s_j}$
which is ensured by Theorem~\ref{thm:1.1}.  Then
\begin{equation*}
L_{s_j} v_{s_j} = r(L_{s_j}) v_{s_j}.
\end{equation*}
For a fixed $t$, $0<t<1$, define $s_t = (1-t)s_0 + t s_1$ and
\begin{equation*}
w_t(x)  = (v_{s_0}(x))^{1-t}(v_{s_1}(x))^t.
\end{equation*}
Then, using H\"older's inequality, we find that
\begin{multline}
\label{2.4}
(L_{s_t}(w_t))(x) = \sum_{\beta \in \B} (b_{\beta}(x)^{s_0} v_{s_0}(x))^{1-t}
 (b_{\beta}(x)^{s_1} v_{s_1}(x))^t
\\
\le \Big(\sum_{\beta \in \B} (b_{\beta}(x)^{s_0} v_{s_0}(x)\Big)^{1-t}
\Big(\sum_{\beta \in \B}  (b_{\beta}(x)^{s_1} v_{s_1}(x)\Big)^t
= [r(L_{s_0})^{1-t}r(L_{s_1})^t] w_t(x).
\end{multline}
Because $w_t(x) >0$ for all $x \in \bar H$, a standard argument (see Lemma 5.9
in \cite{N-P-L}) shows that
\begin{equation}
\label{2.5}
r(L_{s_t}) = \lim_{k \rightarrow \infty} \|L_{s_t}^k\|^{1/k} 
 = \lim_{k \rightarrow \infty} \|L_{s_t}^k(w_t)\|^{1/k} .
\end{equation}
Using inequalities \eqref{2.4} and \eqref{2.5}, we see that
\begin{equation*}
r(L_{s_t}) \le r(L_{s_0})^{1-t} r(L_{s_1})^t.
\end{equation*} 
\end{proof}

In general, if $V$ is a convex subset of a vector space $X$, we shall call
a map $f:V \to [0, \infty)$ log convex if (i) $f(x) = 0$ for all $x \in V$
or (ii) $f(x) >0$ for all $x \in V$ and $x \mapsto \log(f(x))$ is convex.
Products of log convex functions are log convex, and H\"olders inequality
implies that sums of log convex functions are log convex.

Results related to Theorem~\ref{thm:2.1} can be found in \cite{CC},
\cite{DD}, \cite{EE}, \cite{FF}, \cite{GG}, and \cite{HH}. Note that
the terminology {\it super convexity} is used to denote log convexity
in \cite{DD} and \cite{EE}, presumably because any log convex function
is convex, but not conversely.  Theorem~\ref{thm:2.1}, while adequate for
our immediate purposes, can be greatly generalized by a different argument
that does not require existence of strictly positive eigenvectors.  This
generalization (which we omit) contains Kingman's matrix log convexity
result in \cite{EE} as a special case.

In our applications, the map $s \mapsto r(L_s)$ will usually be strictly
decreasing on an interval $[s_1,s_2]$ with $r(L_{s_1}) >1$ and
$r(L_{s_2}) <1$, and we wish to find the unique $s_* \in (s_1,s_2)$ such that
$r(L_{s_*}) =1$.  The following hypothesis insures that $s \mapsto r(L_s)$
is strictly decreasing for all $S$.

\noindent (H9.1): Assume that $b_{\beta}(\cdot)$, $\beta \in \B$ satisfy the
conditions of (H5.1).  Assume also that there exists an integer $\mu \ge 1$
such that $b_\w(x) <1$ for all $\w \in \B_{\mu}$ and all $x \in \bar H$.

\begin{thm}
\label{thm:2.2}
Assume hypotheses (H5.1), (H5.2), (H5.3), and (H9.1) and let $H$ be
mildly regular. Then the map $s \mapsto r(\Lambda_s)$, $s \in \R$, is
strictly decreasing and real analytic and $\lim_{s \rightarrow \infty}
r(\Lambda_s) =0$.
\end{thm}
\begin{proof}
  If $L_s: C(\bar H) \to C(\bar H)$ is given by \eqref{1.2}, it is a standard
  result that $r(L_s^{\nu}) = (r(L_s))^{\nu}$ and $r(\Lambda_s^{\nu}) =
  (r(\Lambda_s))^{\nu}$ for all integers $\nu \ge 1$, and
  Theorem~\ref{thm:1.1} implies that $r(L_s) = r(\Lambda_s)$.  Thus it
  suffices to prove that for some positive integer $\nu$, $s \mapsto
  r(L_s^{\nu})$ is strictly decreasing and $\lim_{s \rightarrow \infty}
  r(L_s^{\nu}) =0$.

Suppose that $K$ denotes the set of nonnegative functions in $C(\bar H)$ and
$A: C(\bar H) \to C(\bar H)$ is a bounded linear map such that $A(K)
\subset K$.  If there exists $w \in C(\bar H)$ such that $w(x) > 0$ for
all $x \in \bar H$ and if $(A(w))(x) \le a w(x)$ for all $x \in \bar H$,
it is well-known (and easy to verify) that $r(A) \le a$, where $r(A)$
denotes the spectral radius of $A$.  In our situation, we take $\nu = \mu$,
where $\mu$ is as in (H9.1), and $A = (L_s)^{\mu}$. If $s <t$ and $v_s$ is the
strictly positive eigenvector for $(L_s)^{\mu}$, (H9.1) implies that
there is a constant $c <1$, $c = c(s,t)$, such that $c b_{\w}(x)^s \ge
b_\w(x)^t$ for all $\w \in \B_{\mu}$ and $x \in H$.  Thus we find that
\begin{equation*}
c r(L_s)^{\mu} v_s(x) = \sum_{\w \in \B_{\mu}} c b_{\w}(x)^s v_s(\theta_\w(x))
\ge \sum_{\w \in \B_{\mu}} b_{\w}(x)^t v_s(\theta_\w(x)) = (L_t^{\mu}(v_s))(x).
\end{equation*}
It follows that $r(L_t)^{\mu} \le c(s,t) r(L_s)^{\mu}$, so $r(L_t) < r(L_s)$,
for $s < t$.  Because $0 < b_\w(x) <1$ for all $x \in \bar H$ and $\w \in
\B_{\mu}$, it is also easy to see that $\lim_{t \rightarrow \infty}
\|(L_t)^{\mu}\| =0$; and since $\|(L_t)^{\mu}\| \ge r(L_t^{\mu})$, we see that
$\lim_{t \rightarrow \infty} r(L_t^{\mu}) =0$.
\end{proof}

\begin{remark}
\label{rem:2.3}
It is easy to construct examples for which (H9.1) is satisfied for some
$\mu >1$, but not satisfied for $\mu =1$.  The functions $\theta_1(x):=
9/(x+1)$ and $\theta_2(x):= 1/(x+2)$ both map the closed interval $\bar H
= [1/11, 9]$ into itself. There is a unique nonempty compact set $J \subset
\bar H$ such that
\begin{equation*}
J = \theta_1(J) \cup \theta_2(J).
\end{equation*}
For $s \in \R$, define $L_s:C(\bar H) \to C(\bar H)$ by
\begin{equation*}
(L_sf)(x):= \sum_{j=1}^2 |D \theta_j(x)|^s f(\theta_j(x)):=
\sum_{j=1}^2 b_j(x)^s f(\theta_j(x)),
\end{equation*}
where $D:= d/dx$.  The Hausdorff dimension of $J$ is the unique $s = s_*$,
$0 < s_* <1$, such that $r(L_s) =1$.  Our previous remarks show that
\begin{equation*}
(L_s^2 f)(x) = \sum_{j=1}^2 \sum_{k=1}^2 |D(\theta_j \circ \theta_k)(x)|^s
f(\theta_j \circ \theta_k)(x)).
\end{equation*}
One can check that (H9.1) is not satisfied for $\mu=1$, but is satisfied
for $\mu =2$.
\end{remark}
\begin{remark}
\label{rem:2.4}
Assume that the assumptions of Theorem~\ref{thm:2.2} are satisfied and define
$\psi(x) = \log(r(L_s)) = \log(r(\Lambda_s))$ (where $\log$ denotes the
natural logarithm), so $s \mapsto \psi(s)$ is a convex, strictly
decreasing function with $\psi(0) >1$ (unless $|\B| = p =1$) and
$\lim_{s \rightarrow \infty} \psi(s) = - \infty$. We are interested in finding
the unique value of $s$ such that $\psi(s) =0$. In general suppose that
$\psi:[s_1,s_2] \to \R$ is a continuous, strictly decreasing, convex
function such that $\psi(s_1) >0$ and $\psi(s_2) <0$, so there exists
a unique $s =s_* \in (s_1,s_2)$ with  $\psi(s_*) =0$.  If $t_1$ and $t_2$ are
chosen so that $s_1 \le t_1 < t_2 \le s_*$ and $t_{k+1}$ is obtained from
$t_{k-1}$ and $t_k$ by the secant method, an elementary argument show that
$\lim_{k \rightarrow \infty} t_k = s_*$.  If $s_* \le t_2 < t_1 < s_2$ and
$s_1 \le t_3$, a similar argument shows that $\lim_{k \rightarrow \infty} t_k
= s_*$. If $\psi \in C^3$, elementary numerical analysis implies
that the rate of convergence is faster than linear ($= (1 + \sqrt{5})/2)$.
In our numerical work, we apply these observations, not directly to $\psi(s) =
\log(r(\Lambda_s))$, but to convex decreasing functions which closely
approximate $\log(r(\Lambda_s))$.
\end{remark}

One can also ask whether the maps $s \mapsto r(B_s)$ and $s \mapsto r(A_s)$
are log convex, where $A_s$ and $B_s$ are the previously described
approximating matrices for $L_s$.  An easier question is whether the map
$s \mapsto r(M_s)$ is log convex, where $A_s$ and $B_s$ are obtained from
$M_s$ by adding error correction terms.  We shall prove that $s \mapsto
r(M_s)$ is log convex, at least in the one dimensional case.  The proof in
the two dimensional case is similar.

First, we need to recall a useful theorem of Kingman \cite{EE}.  Let
$M(s) = (a_{ij}(s))$ be an $m \times m$ matrix whose entries $a_{ij}(s)$ are
either strictly positive for all $s$ in a fixed interval $J$ or are
identically zero for all $s \in J$. Assume that $s \mapsto a_{ij}(s)$ is
log convex on $J$ for $1 \le i,j \le m$.  Under these assumptions,
Kingman \cite{EE} has proved that $s \mapsto r(M_s)$ is log convex.

Let $n \ge 2$ be a positive integer, and for $a < b$ given real numbers,
define $x_k = a + kh$, $-1 \le k \le n+1$, $h = (b-a)/n$. Let $X_n$
denote the vector space of real valued maps $f:\{x_k \, | \, 0 \le k \le n\}
\to \R$, so $X_n$ is a real vector space linearly isomorphic to $\R^{n+1}$.
As usual, if $f \in X_n$, extend $f$ to a map $f^I:[a,b] \to \R$ by
linear interpolation, so
\begin{equation*}
f^I(u) = \frac{u-x_k}{h} f(x_{k+1}) + \frac{x_{k+1}-u}{h} f(x_{k}), \qquad
x_k \le u \le x_{k+1}, 0 \le k \le n.
\end{equation*}
For $1 \le j \le N$, assume that $\theta_j:[a,b] \to [a,b]$ are given
maps and assume that $b_j:[a,b] \to (0, \infty)$ are given positive
functions.  For $s \in \R$, define a linear map $M_s:X_n \to X_n$ by
$M_s(f) = g$, where
\begin{equation*}
g(x_k) = \sum_{j=1}^N [b_j(x_k)]^s f^I(\theta_j(x_k)), \quad 0 \le k \le n,
\end{equation*}
so if $f(x_k) \ge 0$ for $0 \le k \le n$, $g(x_k) \ge 0$ for $0 \le k \le n$.
We can write $g(x_k) = \sum_{m=0}^n a_{km}(x) f(x_m)$, where for
$0 \le k$, $m \le n$,
\begin{multline*}
a_{km}(x) = \sum_{j, x_{m-1} \le \theta_j(x_k) \le x_m} [b_j(x_k)]^s
[\theta_j(x_k) - x_{m-1}]/h
\\
+  \sum_{j, x_{m} \le \theta_j(x_k) \le x_{m+1}} [b_j(x_k)]^s
[x_{m+1} - \theta_j(x_k)]/h.
\end{multline*}
If, for a given $k$ and $m$, there is no $j$, $1 \le j \le N$, with
$ x_{m-1} \le \theta_j(x_k) \le x_{m+1}$, we define $a_{km} =0$.
Since the sum of log convex functions is log convex, $s \mapsto a_{km}(s)$ is
log convex on $\R$.  It follows from Kingman's theorem that $s \mapsto r(M_s)$
is log convex, where $r(M_s)$ denotes the spectral radius of $M_s$.

\bibliographystyle{amsplain}
\bibliography{hdcomplt}
\end{document}